%% file: distance_kernel_transform.tex
\RequirePackage[l2tabu, orthodox]{nag}
\documentclass[12pt]{amsart}
\usepackage{hyperref}
\usepackage{backref}
\usepackage{fullpage,url,amssymb,enumerate,colonequals,pdfsync,float,mathtools,appendix,float,comment,amsmath,amsthm}
\usepackage{subcaption}
\usepackage[all]{xy} 

\usepackage{rotating}

\usepackage{tikz-cd,graphicx}
\usepackage{dsfont}

\makeatletter
\g@addto@macro\bfseries{\boldmath} 
\makeatother

\usepackage{microtype}

\usepackage[alphabetic]{amsrefs} 
\usepackage{color}

\usepackage[framemethod=tikz]{mdframed}

\usepackage{marginnote}
\long\def\@savemarbox#1#2{\global\setbox#1\vtop{\hsize\marginparwidth 
  \@parboxrestore\tiny\raggedright #2}}
\newcommand{\vol}{\operatorname{vol}}
\newcommand{\diam}{\operatorname{diam}}
\newcommand{\dist}{\operatorname{d}}
\newcommand{\C}{\mathbb{C}}
\newcommand{\R}{\mathbb{R}}
\newcommand{\N}{\mathbb{N}}

\newcommand{\toas}[2]{\xymatrix{#1 \ar[r]^-{\mathrm{a.s.}} & #2\\}}
\newcommand\asleq{\mathrel{\overset{\makebox[0pt]{\mbox{\normalfont\tiny\sffamily a.s.}}}{\leq}}}
\newcommand{\aseq}[1]{\mathrel{\overset{\makebox[0pt]{\mbox{\normalfont\tiny\sffamily #1}}}{=}}}
\newcommand{\leqtext}[1]{\mathrel{\overset{\makebox[0pt]{\mbox{\normalfont\tiny\sffamily #1}}}{\leq}}}

\newtheorem{theorem}{Theorem}[section]

\newtheorem{prop}[theorem]{Proposition}
\newtheorem{cor}[theorem]{Corollary}
\newtheorem{corollary}[theorem]{Corollary}
\newtheorem{mr}[theorem]{Main Result}
\newtheorem{lemma}[theorem]{Lemma}

\theoremstyle{definition}
\newtheorem{obs}[theorem]{Observation}
\newtheorem{definition}[theorem]{Definition}
\newtheorem{convention}[theorem]{Convention}
\numberwithin{figure}{section}
\newtheorem{remark}[theorem]{Remark}
\newtheorem{example}[theorem]{Example}

\title{Intrinsic Topological Transforms via the Distance Kernel Embedding}
\author{Cl\'{e}ment Maria, Steve Oudot, Elchanan Solomon}
\date{\today}

\begin{document}
	
	\maketitle
	\thispagestyle{empty}
	\begin{abstract}
	Topological transforms are parametrized families of topological invariants, which, by analogy with transforms in signal processing, are much more discriminative than single measurements. The first two topological transforms to be defined were the Persistent Homology Transform and Euler Characteristic Transform, both of which apply to shapes embedded in Euclidean space. The contribution of this paper is to define topological transforms that depend only on the intrinsic geometry of a shape, and hence are invariant to the choice of embedding. To that end, given an abstract metric measure space, we define an integral operator whose eigenfunctions are used to compute sublevel set persistent homology. We demonstrate that this operator, which we call the distance kernel operator, enjoys desirable stability properties, and that its spectrum and eigenfunctions concisely encode the large-scale geometry of our metric measure space. We then define a number of topological transforms using the eigenfunctions of this operator, and observe that these transforms inherit many of the stability and injectivity properties of the distance kernel operator.
	\end{abstract}

%

\section{Introduction}
\input{intro}

\section{The Distance Kernel}
\label{sec:distkernel}
\input{distkernel}

\section{The Distance Kernel Embedding}
\label{sec:distkernelembed}
\input{distkernelembed}

\section{Finite Approximation of the Distance Kernel Embedding}
\label{sec:discretization}
\input{discretization}

\section{Stability of the Distance Kernel Embedding}
\label{sec:stability}
\input{stability}

\section{Inverse Results}
\label{sec:stabinv}
\input{stabinv}

\section{Error Bounds}
\label{sec:embeddingconstants}
\input{embeddingconstants}

\section{The Intrinsic PHT/ECT}
\label{sec:intrinsicpht}
\input{intrinsicpht}


\section{Experiments}
\label{sec:experiments}
\input{experiments}

\section{Conclusion}
\label{sec:conclusion}
\input{conclusion}

\bibliographystyle{plain}
\bibliography{bib}

\end{document}

%% file: intro.tex
	Since its initial founding in the 1990s, the field of \emph{topological data analysis} (TDA) has matured in several directions: through the development of rigorous mathematical foundations, powerful computational algorithms, and increasingly sophisticated tools. Recent years have seen the introduction of a new class of TDA tools: \emph{topological transforms}. These are parametrized families of topological invariants that enjoy injectivity properties not found in single persistence diagrams. The first to be studied were the Persistent Homology Transform and Euler Characteristic Transform of Turner, Mukherjee, and Boyer \cite{turner2014persistent}, which apply to shapes embedded in Euclidean space.  Later, in \cite{dey2015comparing}, Dey et. al. defined a topological transform for abstract metric graphs, and studied its stability and computability. This was complemented by  work of Oudot and Solomon \cite{oudot2017barcode}, who provided a variety of injectivity results for this intrinsic transform.\\

	Although the intrinsic transform introduced by Dey et. al. can be defined quite generally, it becomes difficult to study and compute on more complicated, higher-dimensional spaces. The goal of this paper is to propose a new intrinsic topological transform, defined on large classes of spaces, such as Riemannian manifolds, using spectral methods.\\

	For a metric measure space $X$, we study the eigenfunctions and eigenvalues of an integral operator on $L^{2}(X)$ called the \emph{distance kernel operator}. We define coordinates on $X$ using appropriately scaled eigenfunctions, and demonstrate that the geometry of $X$ can largely be recovered from the resulting Euclidean embedding. We then use these coordinate functions to define a number of new topological transforms and study their properties.
	 
	\subsection{Prior Work}
	
	In \cite{turner2014persistent}, Turner et. al. defined the Persistent Homology Transform (PHT) and Euler Characteristic Transform (ECT). These transforms take as input a sufficiently regular subset of Euclidean space $S \subseteq \mathbb{R}^d$, and associate to every vector $v \in \mathbb{S}^{d-1}$ the persistence diagram, or Euler characteristic curve, of the sublevel-set filtration of $S$ induced by the function $f_{v}:S \to \mathbb{R}$:
	\[f_{v}(x) = v \cdot x.\]
	
	It was subsequently proven in \cite{curry2018many} and \cite{ghrist2018persistent} that these topological transforms are injective in all dimensions\footnote{By ``injective," we mean that two subsets of Euclidean space have the same transform if and only if they are identical. Thus, the transform is injective on the space of admissible subsets.}. Moreover, it was shown in \cite{curry2018many} and \cite{belton2018learning} that, for certain families of embedded shapes, these topological transforms can be computed in finitely many steps\footnote{That is, finitely many directions determine the entire transform, and these directions can be identified with finitely many geometric computations}. Complimenting these theoretical results, Crawford et. al. \cite{crawford2016topological} demonstrated how to use these topological transforms to build an improved classifier for glioblastoma patient outcomes.
	
	In \cite{oudot2017barcode}, Oudot and Solomon defined a topological transform for intrinsic metric spaces $(X,d_X)$. This transform associates to every basepoint $x_0 \in X$ the extended persistence diagram of the function $f_{x_0}:X \to \mathbb{R}$:
	\[f_{x_0}(x) = d_{X}(x_0,x).\]
	The resulting invariant, called the Intrinsic Persistent Homology Transform (IPHT), is the collection of all persistence diagrams arising from basepoints in $X$. By computing Euler characteristic curves instead of persistence diagrams, one obtains the Intrinsic Euler Characteristic Transform (IECT). This invariant was first studied, in the case of metric graphs, by Dey, Shi, and Wang in \cite{dey2015comparing}, where they proved stability and computability results and ran some experiments. The main result of \cite{oudot2017barcode} demonstrated that these invariants are injective\footnote{Similarly to the PHT and ECT, this injectivity means that two graphs having the same transform must be isometric.} on an appropriately generic subset of the space of metric graphs.
	
	We refer the reader to \cite{oudot2018inverse} for a survey of inverse problems in applied topology.
	
	As this paper is concerned with both applied topology and spectral geometry, let us now consider some results, both classical and modern, in the latter field. To begin, the data of a weighted graph can be encoded via its adjacency matrix, and the spectral theory of these matrices is deep and of great utility, seeing application in, e.g., graph clustering and Google's PageRank algorithm. Another matrix associated to a graph is its Laplacian, whose eigendecomposition forms the basis for the Laplacian Eigenmaps technique studied by Belkin and Niyogi in \cite{belkin2002laplacian,belkin2003laplacian,belkin2007convergence}. If one performs spectral analysis on the centered Gram matrix of a distance matrix, the resulting eigenvectors give the classical Multi-Dimensional Scaling embedding. Tenenbaum et. al. \cite{tenenbaum2000global} defined an extension of Multi-Dimensional Scaling for point clouds, called IsoMap, by building a neighborhood graph on the points, defining distances using shortest paths in this graph, and applying Multi-Dimensional Scaling to the resulting distance matrix. Diffusion map embeddings arise from studying the spectral theory of the diffusion operator, as defined by Coifman and Lafon in \cite{coifman2006diffusion}. Lastly, the X-ray transform of \cite{john1938ultrahyperbolic} takes as input a continuous, compactly supported function $f$ on $\mathbb{R}^d$, and outputs a function on the space of lines in $\mathbb{R}^d$ that encodes the corresponding line integral of $f$.

	Like the embeddings of \cite{belkin2003laplacian} and \cite{coifman2006diffusion}, we study the eigenfunctions and eigenvalues of an operator defined on our shape. However, our goal is not to find a low-dimensional representation of a noisy data set, but rather to use our eigenfunctions to compute Hausdorff distances, persistence diagrams, and Betti curves. As we are interested in injectivity results, we want these eigenfunctions to encode the large-scale geometry of our shape. The results that follow demonstrate that using the distance function as an integral kernel produces an operator well suited to this task.

	\subsection{Results}
	In Section \ref{sec:distkernel}, we define the distance kernel operator on metric measure spaces, proving that it is a compact, self-adjoint operator to which the spectral theorem can be applied (Propositions \ref{prop:selfadjoint} and \ref{prop:compact}). We then establish conventions for choosing eigenfunctions for our embedding.
	
	In Section \ref{sec:distkernelembed}, we use the eigenfunctions $\{\phi_i\}$ of the distance kernel operator to define complex-valued coordinates $\alpha_i = \sqrt{\lambda_i}\phi_i$ on our metric-measure space. We then prove that, for non-zero eigenvalues, these eigenfunctions, and their associated coordinates, are Lipschitz, with constant inversely proportional to the magnitude of the eigenvalue (Lemma \ref{lem:eiglip}). We then define a $\mathbb{C}^k$-valued map $\Phi_{k}$, called the distance kernel embedding, whose component functions are the $\alpha_{i}$, and study its injectivity. When infinitely many eigenfunctions are used, we show that the embedding is pointwise injective\footnote{This injectivity is defined for pairs of points on a fixed space $X$.}:
	
	\begin{mr}[Theorem \ref{lem:psiembed}]
		Let $(X,\dist_{X},\mu_X)$ be a compact, strictly positive\footnote{See Definition \ref{def:strictlypositive}.} metric measure space. Then the map $\Phi : X \to \mathbb{C}^{\infty}$ is injective.
	\end{mr}
	
When finitely many eigenfunctions are used, we show that the fiber of a vector in $\mathbb{C}^k$ has bounded diameter, with an upper bound on the diameter decreasing in the embedding dimension $k$. This result holds for Riemannian manifolds equipped with their volume measure and, more generally, metric measure spaces satisfying a regularity condition called $(a,b)$-standardness\footnote{See Definition \ref{def:abstandard}}:

\begin{mr}[Corollary \ref{cor:Rcoarseinj}]
	There exists a function $N(r,n): \mathbb{R}_{>0} \times \mathbb{N} \to \mathbb{N}$ with the following property.
	Let $X$ be a complete $d$-dimensional Riemannian manifold with positive injectivity radius $R$. For every $r \leq R/2$, and $k \geq N(r,n)$, if  $\Phi_{k}(x) =\Phi_{k}(y)$  then $\dist_{X}(x,y) \leq 3r$.
\end{mr}

\begin{mr}[Corollary \ref{cor:ABcoarseinj}]
	There exists a function $N(s,a,b): \mathbb{R}_{>0} \times \mathbb{R}_{>0} \times \mathbb{R}_{>0} \to \mathbb{N}$ with the following property.
	Let $(X,\dist_{X},\mu_X)$ be a compact $(a,b)$-standard metric measure space with threshold parameter $r$. For every $s \leq r$ and  $k \geq N(s,a,b)$, if  $\Phi_{k}(x) =\Phi_{k}(y)$ then $\dist_{X}(x,y) \leq 3s$.
\end{mr}

Section \ref{sec:discretization} demonstrates that the spectral properties of the distance kernel operator can be approximated, within arbitrary precision, by finite, random i.i.d. samples. The main result is the following:

\begin{mr}[Theorem~\ref{thm:approxDKE}]
	For any compact metric measure space $(X,\dist,\mu)$ with $(a,b)$-standard Borel measure, and $X_n$ an i.i.d. sample of X of size $n$,
	\[
		\dist_H^{L^2}(\Phi_k(X),\hat{\Phi}_k(X_n)) \xrightarrow[]{\text{a.s}}0 \ \text{as} \ n \to +\infty 
	\]
	where $\dist_H^{L^2}$ is the Hausdorff distance for the $L^2$ norm in $\C^k$, and $\hat{\Phi}_k(X_n)$ is the \emph{empirical distance kernel embedding}\footnote{See Section~\ref{sec:discretization} for the appropriate construction} defined on the sample $X_n$.
\end{mr}

	Section \ref{sec:stability} proves that the distance kernel embedding is stable for compact Riemannian manifolds. More precisely, Main Result \ref{mr:stab} proves that there is a choice of metric on the space of Riemannian manifolds, namely the (modified) Gromov Hausdorff Prokhorov distance\footnote{See Definitions~\ref{def:modifiedPdistance} and~\ref{def:modifiedGHPdistance}}, such that the Hausdorff distance between the $k$-dimensional distance kernel embeddings of two Riemannian manifolds is bounded by a function $F$ that depends on their distance in this metric as well as their $k$ largest (in absolute value) eigenvalues, and where the function $F$ goes to zero as the distance between the manifolds goes to zero.    


	\begin{mr}[Theorem \ref{thm:stability}]
		\label{mr:stab}
Let $(X,\dist_X,\mu)$ and $(Y,\dist_Y,\eta)$ be compact finite dimensional Riemannian manifolds with their volume measures, such that $\mu(X) = \eta(Y)$. Let $|\lambda_1| > \ldots > |\lambda_k| > 0$ and $|\nu_1| > \ldots > |\nu_k| > 0$ be the $k$ largest eigenvalues in absolute value of operators $D^X$ and $D^Y$ respectively, all non-trivial, of multiplicity one, and with distinct absolute values. Let $\Phi_k\colon X \to \mathbb{C}^k$ and $\Psi_k\colon Y\to \mathbb{C}^k$ be the DKE for $D^X$ and $D^Y$ respectively. 

\medskip

Define the intertwining $\Delta_k^{X,Y}$ of the spectra of $D^X$ and $D^Y$, measuring the separation between the spectra of the operators, by:
\[
	\Delta_k^{X,Y} = \min \left\{ |\lambda_i^2-\nu_j^2| \ : \ 1 \leq i,j \leq k, i \neq j \right\},
\]
and let:
\[
	|\tau_1| := \min \{|\lambda_1|,|\nu_1|\}, \ \ \text{and} \ \ |\tau_k| := \max \{|\lambda_k|,|\nu_k|\}.
\]

\medskip

Then, 
\[
	\dist_H^{L^2}(\Phi_k(X),\Psi_k(Y)) \leq \displaystyle\sqrt{k}\left( 4\sqrt{2} \frac{(\varepsilon+|\tau_1|)}{\Delta^{X,Y}_k}\sqrt{|\tau_1|}\right) \cdot \varepsilon + 
	\displaystyle\sqrt{k}\left( 2\sqrt{2} \sqrt{\frac{(\varepsilon + |\tau_1|)}{|\tau_k|}}\right) \cdot \sqrt{\varepsilon}, 
\]
where $\varepsilon := \dist_{G\bar{P}}(X,Y)$ stands for the modified Gromov-Prokhorov distance between $X$ and $Y$ (see Definitions~\ref{def:modifiedPdistance} and~\ref{def:modifiedGHPdistance} below).
	\end{mr}
	
	In Section \ref{sec:stabinv}, we study the global injectivity and stability properties of the distance kernel embedding. The first result is that, under mild regularity assumptions, the embedding encodes the entire metric data of the space:
	
	\begin{mr}[Lemma \ref{lem:PhiInj}]
			Fix a topological space $X$. Let $\mu_{1}$ and $\mu_{2}$ be strictly positive measures for the Borel $\sigma$-algebra on $X$, with $\mu_1$ absolutely continuous with respect to $\mu_2$, and $\dist_{1}$ and $\dist_{2}$ metrics on $X$ making $X_1 = (X,\dist_1,\mu_1)$ and $X_2 = (X,\dist_2,\mu_2)$ metric measure spaces. Let $D_1$ and $D_2$ be the resulting integral operators. If $\Phi(X_1) = \Phi(X_2)$, then $\dist_1 = \dist_2$. 
	\end{mr}

	Moving forward, we show that if the distance kernel embeddings of two metric measure spaces are at small Hausdorff distance from each other, then these spaces are close in the Gromov-Hausdorff distance:
	
	\begin{definition}
		For a compact metric measure space $(X, \dist_{X}, \mu_{X})$ and a positive integer $k$, define the error function (where $[\cdot,\cdot]$ is a bilinear form defined in Section \ref{sec:stabinv}):
		\[E_{X,k}(x,x') = |[\Phi_{k}(x),\Phi_{k}(x')] - \dist_{X}(x,x')|.\] 
		This error measures the extent to which the eigenfunction expansion of $\dist_{X}$ fails to approximate the metric.
	\end{definition}
	
	\begin{mr}[Theorem \ref{thm:invstabmeas}]
	\label{mr:invstabmeas}
	Let $(X,\dist_{X}, \mu_X)$ and $(Y,\dist_{Y}, \mu_Y)$ be compact metric measure spaces, with eigenvalues $\{\lambda_{i}\}$ and $\{\nu_{i}\}$ respectively. Take $k \in \mathbb{N}$ to be a positive integer, and let $\varepsilon = \dist_{H}^{L^2}(\Phi_{k}(X),\Phi_{k}(Y))$. Then
\[d_{GH}(X,Y) \leq 2\varepsilon \min \left\{  \max_{x \in X} \| \Phi_{k}(x) \|_{2}, \max_{y \in Y} \| \Phi_{k}(y) \|_{2}\right\} + \| E_{X,k} \|_{\infty} + \| E_{Y,k} \|_{\infty} + \varepsilon^2. \]
	\end{mr}
	 
	 Simulations conducted in Section \ref{sec:experiments} demonstrate the constants appearing in this result are well-behaved, with the multiplicative error remaining small, and the additive error going to $0$, as $k$ goes to infinity.
	 
	 We then study the embedding in detail in the special case of finite metric measure spaces, culminating with the following simplified version of Theorem \ref{thm:invstabmeas}:

	 \begin{mr}[Theorem \ref{thm:invstab}]
	 	Let $(X,\dist_{X}, \mu_X)$ and $(Y,\dist_{Y}, \mu_Y)$ be finite metric measure spaces, with eigenvalues $\{\lambda_{i}\}$ and $\{\nu_{i}\}$, and let $\theta = \min \{\min_{x \in X}\mu_{X}(x),\min_{y \in Y}\mu_{X}(y)\}$. Take $k \leq |X|,|Y|$, and suppose that $\dist_{H}^{L^{2}}(\Phi_{k}(X),\Phi_{k}(Y)) \leq \varepsilon$. Then,
\[\dist_{GH}(X,Y) \leq 2\varepsilon \frac{\min (\sqrt{|\lambda_{1}|},\sqrt{|\nu_{1}|})}{\theta} + \varepsilon^2 + \frac{|\lambda_{k+1}| + |\nu_{k+1}|}{\theta}.       \]	 
	 \end{mr}
 
 	The quality of the bound in Main Result~\ref{mr:invstabmeas} depends on the $L^2$ norms of the embedding vectors and the $L^\infty$ norms of the error functions. In Section~\ref{sec:embeddingconstants}, explicit bounds on these quantities are given in terms of the geometry of the metric measure spaces. These worst-case bounds are in fact pessimistic, and the simulations considered at the end of the paper suggest that they can be improved upon greatly.
	 
	 In Section~\ref{sec:intrinsicpht}, we study the topological features of the spectrum of the distance kernel operator. We demonstrate that our eigenfunctions have persistence diagrams and, under certain regularity assumptions, Betti and Euler curves:
	 
	 \begin{mr}[Proposition \ref{prop:pers}]
	 	Let $(X,\dist_{X},\mu_X)$ be a compact metric measure space homeomorphic to the geometric realization of a finite simplicial complex. Then, any finite linear combination $f = \sum_{i=1}^{n}c_i \phi_i$ of eigenfunctions of $D^X$ with nonzero eigenvalue has a well-defined sublevel set persistence diagram $PH(X,f)$.
	 \end{mr}
	 
	 \begin{mr}[Proposition \ref{prop:betticurves}]
	 	Suppose now that $X$ is homeomorphic to the geometric realization of a finite simplical complex which implies bounded degree-$q$ total persistence\footnote{See Definition \ref{def:boundedtotal}. This condition ensures the existence of our Euler curves}. Let $p=1/q$. Then for any homological degree $k$, the sum defining $\beta_{k}(X,f)$ converges in $L^p$. Under the same hypothesis, the sum defining $\chi(X,f)$ is finite, so that the Euler curve also exists as a function in $L^p$.
	 \end{mr}

 	In Definition \ref{def:transforms}, we define two types of topological transforms: (1) intrinsic transforms, namely the i-PKT and i-EKT, that compute the persistence diagrams and Euler curves of linear combinations of eigenfunctions, and (2) embedded transforms, namely the e-PKT and e-EKT, that amount to computing the PHT and ECT on the image of the distance kernel embedding. Like the PHT and ECT, all of these transforms give maps from a sphere of the appropriate dimension to a space of topological summaries (either persistence diagrams or Euler curves). We then prove that the i-PKT and e-PKT are Lipschitz continuous, whereas their Euler Characteristic counterparts are H\"{o}lder continuous:
 	
 	\begin{mr}[Proposition \ref{prop:iPKTlip}]
	Suppose that $X$ is homeomorphic to the geometric realization of a finite simplicial complex, and the barcode space is equipped with the bottleneck distance. Then, both the i-$PKT_{k}$ and e-$PKT_{k}$ are Lipschitz continuous on $\mathbb{S}^{2k-1}$.
 	\end{mr}
 
 \begin{mr}[Corollary \ref{cor:ECTlip}]
 	Let $q \leq 1$, and suppose that $X$ is homeomorphic to the geometric realization of a finite simplicial complex which implies bounded degree-$q$ total persistence. Suppose further that there is a uniform bound on the number of points in the persistence diagrams obtained when evaluating the i-$PKT_{k}$ and e-$PKT_{k}$ at an arbitrary vector $(u,v) \in \mathbb{S}^{2k-1}$. If we equip the sphere $\mathbb{S}^{2k-1}$ with the $\ell^1$ distance, and the space of Euler curves with the $L^{1/q}$ distance, the i-$ECT_{k}$ and e-$ECT_{k}$ are $q$-H\"{o}lder continuous on $\mathbb{S}^{2k-1}$.
 \end{mr}
%
%
%
%
%

 Lastly, as a consequence of Main Result \ref{mr:invstabmeas}, we show that the e-PKT and e-EKT have bounded inverses. To obtain this technical result, we need to assume that our distance kernel embeddings are \emph{definable} subsets of Euclidean space. The notion of definability comes from the theory of o-minimal geometry, and is necessary for the application of techniques from Euler calculus. C.f. \cite{curry2018many} \S2 for precise definitions. Finite embeddings are always definable. 
 
 \begin{mr}[Theorem \ref{thm:ePKTcoarseinj}]
Let $(X,\dist_{X}, \mu_X)$ and $(Y,\dist_{Y}, \mu_Y)$ be compact metric measure spaces, with eigenvalues $\{\lambda_{i}\}$ and $\{\nu_{i}\}$ respectively. Let $k \in \mathbb{N}$ be a positive integer, and suppose that $\Phi_{k}(X)$ and $\Phi_{k}(Y)$ are definable. Then if either e-PKT$_{k}(X)$ = e-PKT$_{k}(Y)$ or e-EKT$_{k}(X)$ = e-EKT$_{k}(Y)$, we have:

\[d_{GH}(X,Y) \leq  \| E_{X,k} \|_{\infty} + \| E_{Y,k} \|_{\infty}.\]
 \end{mr}
 
 The final section of the paper, Section \ref{sec:experiments}, consists of a variety of numerical experiments serving as proofs of concept for the distance kernel embedding. We consider a number of discrete metric spaces sampled from Lens spaces, tori, 2-spheres, and 3-spheres, and we look at the Hausdorff distances between their embeddings, the distribution of their eigenvalues, the distribution of mass of the scaled eigenvectors, the distribution of magnitudes of their embedding vectors, the values of the additive and multiplicative error bounds appearing in Theorem \ref{thm:invstabmeas}, and the values of the estimates for these bounds appearing in Lemmas \ref{lem:errorbound} and \ref{lem:embedbound}. We observe that, for our data sets, the values of the additive and multiplicative errour bounds in Theorem \ref{thm:invstabmeas} are much smaller than the estimates provided in Section \ref{sec:embeddingconstants}, and that the distance kernel embedding suffices to distinguish a variety of manifolds, including two Lens spaces with same homotopy type.

%% file: distkernel.tex
Before defining the operator of interest, we recall some basic definitions in metric geometry and functional analysis.

\begin{definition}
	A \emph{metric measure space} $(X,\dist_{X},\mu_X)$ is a metric space together with a Borel measure $\mu_X$ on the topology induced by $\dist_{X}$. In what follows, we assume that our measures are Radon. 
\end{definition}

\begin{definition}
	For a metric measure space $(X,\dist_{X},\mu_X)$, we define the volume of $X$, $\vol(X) = \mu_{X}(X)$, and the diameter $\operatorname{diam}(X) = \sup \{\dist_{X}(x,x') \colon x,x' \in X\}$. These may be finite or infinite.
\end{definition}

\begin{definition}
	An operator $T$ on a Hilbert space $\mathcal{H}$ is \emph{compact} if, for every bounded sequence $x_n \in \mathcal{H}$, the sequence $T(x_n) \in \mathcal{H}$ has a convergent subsequence.
\end{definition}

Let $(X, \dist_{X}, \mu_X)$ be a compact metric measure space. We define the following operator on $L^{2}(X)$, called the distance kernel operator:
\[(D^{X}f)(x) = \int_{X}f(y)\dist_{X}(x,y)d\mu_{X}(y). \]

\begin{prop}
	\label{prop:selfadjoint}
	$D^X$ is a self-adjoint operator.
\end{prop}
\begin{proof}
	By convention, $\mu_X$ is Radon. Since $X$ is compact, this implies that $\mu_{X}(X) < \infty$, and hence $(X,\mu_X)$ is $\sigma$-finite. We can thus apply Fubini's theorem, and the symmetry of the distance function $\dist_{X}$, to observe that, for two integrable functions $f$ and $g$,
	\begin{align*}
	\langle D^{X}f,g\rangle & = \int_{X} \left( \int_{X} f(y)\dist_{X}(x,y) d\mu_{X}(y) \right) g(x) d\mu_{X}(x)\\
	& = \int_{X} \int_{X} f(y)g(x) \dist_{X}(x,y) d\mu_{X}(x)d\mu_{X}(y)\\
	& = \int_{X} f(y) \left( \int_{X} g(x) \dist_{X}(y,x) d\mu_{X}(x) \right) d\mu_{X}(y)\\
	& = \langle f, D^{X}g \rangle
	\end{align*}
	demonstrating self-adjointness.
\end{proof}

\begin{prop}
	\label{prop:compact}
	$D^{X}$ is a compact operator.
\end{prop}
	\begin{proof}
Let $f_{n} \in L^{2}(X)$ be a bounded sequence of functions, $||f_{n}||_{L^2} \leq C$ for all $n$.

For $\dist_{X}(x,x') \leq \varepsilon$ and all $n$,

\begin{align*}
|D^{X}f_{n}(x) - D^{X}f_{n}(x')| & = \left|\int_{X} \left( \dist_{X}(x,y)f_{n}(y) - \dist_{X}(x',y)f_{n}(y) \right) d\mu_{X}(y)\right|\\
& \leq \int_{X}|\dist_{X}(x,y) - \dist_{X}(x',y)||f_{n}(y)| d\mu_{X}(y)\\
\text{(Cauchy-Schwarz)}& \leq \|\dist_{X}(x,y) - \dist_{X}(x',y)\|_{L^{2}(y)} \cdot \|f_{n}(y)\|_{L^{2}(y)}\\
\text{(triangle inequality)}& \leq \|\varepsilon\|_{L^{2}(y)} \cdot \|f_{n}(y)\|_{L^{2}(y)}\\
& \leq \varepsilon \sqrt{\vol (X)}  \cdot C .
\end{align*}

Thus, $D^{X}f_{n}$ is an equicontinuous family of functions on $X$, so, by the 
Arzel\`a-Ascoli theorem, it contains a uniformly convergent, and hence $L^2$-convergent, subsequence. This demonstrates compactness.
\end{proof} 

\begin{corollary}
	The spectral theorem for compact, self-adjoint operators on a Hilbert space implies that $L^{2}(X)$ admits a finite or countably infinite orthonormal basis consisting of eigenfunctions $\phi_{1}, \phi_{2}, \phi_{3}, \cdots$ of $D^{X}$ with eigenvalues $\lambda_{1}, \lambda_{2}, \lambda_{3}, \cdots$. By convention, we order our eigenvalues (and hence eigenfunctions) in decreasing order of absolute value, $|\lambda_1| \geq |\lambda_2| \geq |\lambda_3| \geq \ldots$, breaking the tie between positive-negative pairs by listing the positive eigenvalue first. Note that, as a consequence of the spectral theorem, $|\lambda_{i}|$ goes to zero as $i$ goes to infinity.
\end{corollary}

In general, it is difficult to estimate the spectrum of the distance kernel operator. However, we do have the following bound on the top eigenvalue.

\begin{lemma}
	For any compact metric measure space $(X,\dist_{X},\mu_X)$, $|\lambda_{1}| \leq \operatorname{diam}(X)\operatorname{vol}(X)$, and this bound is asymptotically sharp.
\end{lemma}
\begin{proof}
	Let $f \in L^2 (X)$ be a function whose absolute value attains a (not necessarily unique) maximum at $x$. Observe that:
	\[|(D^{X}f)(x)| = \left|\int_{X} \dist_{X}(x,y)f(y) dy \right| \leq \int_{X} \dist_{X}(x,y)|f(x)| dy \leq \operatorname{diam}(X)\operatorname{vol}(X)|f(x)|. \]
	Taking $f = \phi_{1}$ completes the proof.
	
	To demonstrate asymptotic sharpness, let $X$ consist of $n$ points, at distance $\delta>0$ from one another, and each with measure $1/n$, so that $X$ has unit measure. If $f$ is any nonzero constant function, then $D^{X}(f)/f = \delta \cdot \frac{n-1}{n}$, which converges to $\delta$ as $n \to \infty$.
\end{proof}

\begin{convention}
	\label{con:sign}
	The spectral theorem asserts the existence of the eigenfunctions $\phi_i$, but does not guarantee their uniqueness. Indeed, the choice is never unique. If the eigenvalue $\lambda_{i}$ has geometric multiplicity one, one has two choices: $\{\phi_i, -\phi_i \}$. If the eigenvalue has geometric multiplicity greater than one, there are infinitely many choices. In the rest of the paper, we make the generic assumption that all the eigenvalues have multiplicity one.
	
	\medskip
	
	In resolving the ambiguity between the pair $\{\phi_i, -\phi_i \}$, we have two options. In \cite{belkin2007convergence}[Note 1], the authors suggest fixing an arbitrary function $f$ for which $\langle f, \phi_i \rangle \neq 0 \, \, \forall i$, and asserting that $\langle f, \phi_i \rangle > 0 \,\, \forall i$. In order to maintain consistency of the sign convention, we would like $f$ to be canonically defined. For example, one might set $f$ to be the constant function $f(x) = 1$, which does not depend on the representation of the data.\\

	In addition to the convention adopted in \cite{belkin2007convergence}, we propose another convention of fixing the sign: asserting that $\langle \phi_i, |\phi_i | \rangle > 0 \,\, \forall i$. As before, this convention does not always work, as the integral in question may be equal to zero. However, let us consider what such an equality would imply. Let $X_{i}^{+}  = \{x \in X \colon \phi_{i}(x) > 0 \rangle$ and $X_{i}^{-}  = \{x \in X \colon \phi_{i}(x) < 0 \rangle$. If $\langle \phi_i, |\phi_i|\rangle = 0$, we have:
	\[\int_{X_i^{+}} \phi_{i}^{2}(x) d\mu_{X}(x) - \int_{X_i^{-}} \phi_{i}^{2}(x) d\mu_{X}(x) = 0.  \] 
	At the same time, because $\langle \phi_i, \phi_i \rangle = 1$, we have:
		\[\int_{X_i^{+}} \phi_{i}^{2}(x) d\mu_{X}(x) + \int_{X_i^{-}} \phi_{i}^{2}(x) d\mu_{X}(x) = 1. \] 
	Summing up these two equations, we obtain:
		\[\int_{X_i^{+}} \phi_{i}^{2}(x) d\mu_{X}(x) = \frac{1}{2}.  \]

	From this equality, we see that $\langle \phi, |\phi| \rangle  \neq 0$ is generically satisfied for a unit-norm function $\phi$. However, as with the condition of \cite{belkin2007convergence}, it remains to be shown that this genericity is maintained when restricting ourselves to the family of eigenfunctions in question.

\end{convention}

%

%% file: distkernelembed.tex
%

In this section, we consider the Hilbert-space embedding provided by the spectrum of the operator $D^{X}$ associated to a compact metric measure space $(X,\dist_{X},\mu_X)$. Following Convention \ref{con:sign}, we generically assume that our eigenvalues are of multiplicity one, and that we have a coherent convention for breaking the $\pm$-symmetry and picking ``positive" eigenfunctions. 

\begin{definition}\label{def:alphafunctions}
	Let $D^{X}: L^{2}(X) \to L^{2}(X)$ be the operator defined in Section \ref{sec:distkernel}, with corresponding orthonormal system of eigenfunctions and eigenvalues  $(\phi_i, \lambda_i)$. For an eigenfunction $\phi_{i}$, define $\alpha_{i} = \sqrt{\lambda_i}\phi_i: X \to \C$. 
	By convention, when $\lambda_{i}$ is negative, we take the square root with positive imaginary part. For a point $x \in X$, with associated distance-to-$x$ function $\dist_{x}$, we have:
	
	\begin{equation}
	\label{eqn:coeffs}
	\langle \dist_{x}, \phi_{i}\rangle_{L^2} = \int_{X}\dist_{X}(x,y)\phi_{i}(y) d\mu_{X}(y) = (D^{X}\phi_{i})(x) = \lambda_{i}\phi_{i}(x) = \sqrt{\lambda_i}\alpha_{i}(x).
	\end{equation}
	
\end{definition}

Thus, $\dist_{x}$ has the following eigenfunction expansion\footnote{Note that the choice of $\sqrt{\lambda_i}$ as scaling factor in $\sqrt{\lambda_i}\phi_i$ allows this expression of $\dist_x$. Additionally, it will prove essential in the results of Section \ref{sec:stabinv}}, which converges in $L^2$:

\[\dist_{x} = \sum_{i=1}^{\infty}\langle \dist_{x}, \phi_{i}\rangle_{L^2}\phi_i = \sum_{i=1}^{\infty} \sqrt{\lambda_i} \alpha_{i}(x)\phi_{i} =  \sum_{i=1}^{\infty}\alpha_{i}(x)\alpha_{i}.\]

We now demonstrate that the eigenfunctions $\phi_i$ corresponding to nonzero eigenvalues are Lipschitz.


\begin{lemma}
	\label{lem:eiglip}
For every $i \in \mathbb{N}_{>0}$, The function $\lambda_{i} \phi_{i}$ is $\sqrt{\operatorname{Vol(X)}}$-Lipschitz. Hence, if  $\lambda_i \neq 0$, $\phi_i$ is $(\sqrt{\operatorname{Vol(X)}}/|\lambda_{i}|)$-Lipschitz. Note that our assertion that $X$ is compact implies that it has finite volume, so these Lipschitz constants are indeed finite.
\end{lemma}
\begin{proof}
	 Let $x,y \in X$ and $\varepsilon = \dist_{X}(x,y)$. By the fact that $\lambda_{i} \phi_{i} = D^{X} \phi_i$ and the Cauchy-Schwarz inequality, we have
	\begin{align*}
	|\lambda_{i} \phi_{i}(x) - \lambda_{i} \phi_{i}(y)|^{2}& = \left|(D^{X}\phi_i)(x) - (D^{X}\phi_i)(y)   \right|^{2}\\
	 & = \left| \int_{X}(\dist_{X}(x,z) - \dist_{X}(y,z))\phi_{i}(z) d\mu_{X}(z)\right|^{2}\\
	& \leq  \int_{X}(\underbrace{\dist_{X}(x,z) - \dist_{X}(y,z)}_{\leq \dist_{X}(x,y) = \varepsilon})^{2} d\mu_{X}(z) \cdot \underbrace{\int_{X} \phi_{i}^{2}(z) d\mu_{X}(z)}_{=1}\\
	& \leq \varepsilon^2 \operatorname{Vol}(X).
	\end{align*}
	Thus, $|\lambda_{i} \phi_{i}(x) - \lambda_{i} \phi_{i}(y)| \leq \varepsilon \sqrt{\operatorname{Vol(X)}}$, so $\lambda_{i} \phi_{i}$ is $\sqrt{\operatorname{Vol(X)}}$-Lipschitz.
\end{proof}

We now consider the functions $\alpha_{i}$ as coordinates on the space $X$.

\begin{definition}
	\label{def:strictlypositive}
	For $k \geq 1$, we define $\Phi_{k}: X \to \C^{k}$ to be the map sending a point $x \in X$ to $(\alpha_{1}(x), \cdots, \alpha_{k}(x)) \in \C^k$. Setting $k = \infty$ gives us a map $\Phi: X \to \C^{\infty}$. We adopt the convention of dropping eigenfunctions in the zero-eigenspace, replacing them with the zero function if necessary.
	
\end{definition}

\begin{definition}
For a metric measure space $(X,\dist_{X},\mu_X)$, we define the \emph{distance kernel embedding} $DKE(X)$ to be the image of the map $\Phi : X \to \C^{\infty}$, and the truncated distanced kernel embedding $DKE_{k}(X)$ to be the image of the map $\Phi_{k}: X \to \C^{k}$.
\end{definition}

%

Recall the following measure-theoretic definition:

\begin{definition}
For a topological space $X$ equipped with its Borel $\sigma$-algebra, we call a measure $\mu_X$ \emph{strictly positive} if the measure of any nonempty open set is strictly positive.
\end{definition}

The following lemma demonstrates, under mild conditions on the measure, that the DKE is pointwise injective.
\begin{lemma}
	Let $(X,\dist_{X},\mu_X)$ be a compact, strictly positive metric measure space. Then the map $\Phi : X \to \C^{\infty}$ is injective.
	\label{lem:psiembed}
\end{lemma}
\begin{proof}
	Suppose that there are $x \neq y \in X$ such that $\Phi(x) = \Phi(y)$. This implies that $\alpha_{i}(x) = \alpha_{i}(y)$ and, in turn, $\lambda_{i}\phi_{i}(x) = \lambda_{i}\phi_{i}(y)$ for all $i$. Let $\dist_{x}$ and $\dist_{y}$ be the distance functions associated to $x$ and $y$ respectively. We know from Equation \ref{eqn:coeffs} that
	\[\label{eq:sproddistfunction}
		\langle \dist_{x}, \phi_i \rangle =  \lambda_{i}\phi_{i}(x). \]

	Thus, using the $L^2$-convergence of our eigenfunction expansion,
	\[ \|\dist_{x} - \sum_{i=1}^{n} \lambda_{i}\phi_{i}(x) \phi_{i} \|_{L^2} = \|\dist_{x} - \sum_{i=1}^{n} \langle \dist_{x}, \phi_i \rangle \phi_i \|_{L^2} \xrightarrow[]{n \to \infty} 0.  \]
	Similarly, 
	\[ \|\dist_{y} - \sum_{i=1}^{n} \lambda_{i}\phi_{i}(y)\phi_i \|_{L^2} = \|\dist_{y} - \sum_{i=1}^{n} \langle \dist_{y}, \phi_i \rangle \phi_i \|_{L^2} \xrightarrow[]{n \to \infty} 0.  \]
	
	Since
	\[\sum_{i=1}^{n} \lambda_{i}\phi_{i}(x)\phi_i  = \sum_{i=1}^{n} \lambda_{i}\phi_{i}(y)\phi_i,  \]
	we may apply the triangle inequality and take limits to conclude that $\|\dist_{x} - \dist_{y}\|_{L^2} = 0$. Let now $r = \dist_{X}(x,y)/3 > 0$, and let $U$ be the open neighborhood of radius $r$ around $x$. The function $|\dist_{x} - \dist_{y}|$ is bounded below by $r$ on $U$, and since $U$ is not empty (it contains $x$), it has strictly positive measure. This then implies $ \|\dist_{x} - \dist_{y}\|_{L^2} > 0$, a contradiction. Thus, $\Phi(x) \neq \Phi(y)$ for $x \neq y$.
\end{proof}	

\begin{cor}
	\label{cor:DKThom}
	By Lemma \ref{lem:eiglip}, every component of the map $\Phi$ is continuous. Any metric on $\mathbb{C}^{\infty}$ gives it a Hausdorff topological structure. Thus, for any such choice of metric, $\Phi$ is a continuous injection from a compact space to a Hausdorff space, and hence is a homeomorphism.
\end{cor}

%
%
%
%
\begin{remark}
	Bates \cite{bates2014embedding} demonstrated that finitely many \emph{Laplacian} eigenfunctions are needed to (not necessarily isometrically) embed a Riemannian manifold in Euclidean space, and that the maximal embedding dimension depends on the dimension, injectivity radius, Ricci curvature, and volume of the manifold. By contrast, Lemma \ref{lem:psiembed} provides an embedding with infinitely many eigenfunctions, a much weaker result. However, it holds in greater generality, applying to any compact, strictly positive metric measure space.
\end{remark}

As opposed to Laplacian eigenfunctions, it is not clear that finitely many distance kernel eigenfunctions suffice for an injective embedding. However, one can, in certain settings, get a \emph{coarse injectivity}: if $\Phi_{k}(x) = \Phi_{k}(y)$ for $x,y \in X$, then $\dist_{X}(x,y) \leq \varepsilon(k)$, where $\lim_{k} \varepsilon(k) = 0$. Let us first show this in the setting of Riemannian manifolds, equipped with the volume measure, for which we need the following lower bound on the volume of balls in $X$.

\begin{prop}[\cite{croke1980some}, Prop. 14]
	For every dimension $n$, there exists a constant $C_n$ such that if $M$ is an $n$-dimensional Riemannian manifold with injectivity radius $R$, and $r \leq \frac{1}{2}R$, then $\operatorname{Vol}B(x,r) \geq C_{n}r^n$ for all $x \in M$.
\end{prop}

\begin{lemma}
	\label{lem:coarseinj}
	There exists a universal function $T(r,n): \mathbb{R}_{>0} \times \mathbb{N} \to \mathbb{R}$ with the following property.
	Let $X$ be a compact $n$-dimensional Riemannian manifold with positive injectivity radius $R>0$, and let $r \leq R/2$. Then if $x,y \in X$ are points with associated distance functions $\dist_{x}$ and $\dist_{y}$, and if $\| \dist_{x} - \dist_{y}\|_{L^2} < T(r,n)$, then $\dist_{X}(x,y) \leq 3r$.

\end{lemma}
\begin{proof}
	Suppose that $\dist_{X}(x,y) > 3r$. Then the balls of radius $r$ about $x$ and $y$ do not overlap, and for all $z \in B = B(x,r) \sqcup B(y, r)$, $|\dist_{X}(x,z) - \dist_{X}(y,z)| \geq r$. We can deduce that:
	\[\| \dist_x - \dist_y \|_{L^2}^{2} \geq \int_{B} (\dist_x - \dist_y)^2 d\mu_{X} \geq 2C_{n}r^{n+2}, \]
	where the factor of $2$ appears because $B$ is the disjoint union of \emph{two} balls of radius $r$. Thus,
	\[\| \dist_x - \dist_y \|_{L^2} \geq \sqrt{2 C_{n}}r^{(n/2) + 1}. \]
	
	Setting $T(n,r) = \sqrt{2 C_{n}} r^{(n/2) + 1}$ completes the proof.
\end{proof}

\begin{cor}
	\label{cor:Rcoarseinj}
	There exists a function $N(r,n): \mathbb{R}_{>0} \times \mathbb{N} \to \mathbb{N}$ with the following property.
	Let $X$ be a complete $n$-dimensional Riemannian manifold with positive injectivity radius $R$. For every $r \leq R/2$ and $k \geq N(r,n)$, if  $\Phi_{k}(x) =\Phi_{k}(y)$  then $\dist_{X}(x,y) \leq 3r$.
\end{cor}

\begin{proof}
	Let $T(r,n)$ be as in Lemma \ref{lem:coarseinj}. 
	From the $L^2$ convergence of the eigenfunction expansions of the distance functions $\dist_{x}$ and $\dist_{y}$, there exists $N(r,n)$ large enough such that, for all $k \geq N(r,n)$,
	\begin{equation}
	\label{eqn:convergence}
	\max \left\{ \|\dist_{x} - \sum_{i=1}^{k}\alpha_{i}(x)\alpha_{i}\|_{L^2}, \|\dist_{y} - \sum_{i=1}^{k}\alpha_{i}(y)\alpha_{i}\|_{L^2}  \right\} < \frac{T(r,n)}{2}.  \end{equation}
	If $\Phi_{k}(x) = \Phi_{k}(y)$, then $\alpha_{i}(x) = \alpha_{i}(y)$ for $i =1, \cdots, k$. Substituting this equality into Equation \ref{eqn:convergence}, and applying the triangle inequality, we deduce that $\| \dist_{x} - \dist_{y}\|_{L^2} < T(r,n)$. The result then follows directly from Lemma \ref{lem:coarseinj}.

\end{proof}

To generalize the above result to metric measure spaces, we need a lower bound on the volume of balls:

\begin{definition}
	\label{def:abstandard}
	Let $a,b > 0$ be positive real numbers. A metric measure space $(X,\dist_{X},\mu_X)$ is called \emph{$(a,b)$-standard} if there is a threshold $r > 0$ such that for all $s \leq r$ and all $x \in X$, 
	\[\operatorname{Vol}B(x,s) \geq as^b .\]
	
	The constant $a$ can be interpreted as bounding the curvature of $X$, whereas the constant $b$ is related to the dimension of $X$. 
\end{definition}

The proof of the following Lemma is identical to that of Lemma \ref{lem:coarseinj}.

\begin{lemma}
	There exists a universal function $T(s,a,b): \mathbb{R}_{>0} \times \mathbb{R}_{>0} \times \mathbb{R}_{>0} \to \mathbb{R}$ with the following property.
	Let $(X,\dist_{X},\mu_X)$ be a compact $(a,b)$-standard metric measure space with threshold parameter $r > 0$. For every $s \leq r$, if $x,y \in X$ are points with associated distance functions $\dist_{x}$ and $\dist_{y}$ respectively, and if $\| \dist_{x} - \dist_{y}\|_{L^2} < T(s,a,b)$, then $\dist_{X}(x,y) \leq 3s$.
\end{lemma}
\begin{proof}
	Suppose that $\dist_{X}(x,y) > 3s$. Then the balls of radius $s$ about $x$ and $y$ do not overlap, and for all $z \in B = B(x,r) \sqcup B(y, r)$, $|\dist_{X}(x,z) - \dist_{X}(y,z)| \geq s$. We can deduce that:
\[\| \dist_{x} - \dist_{y} \|_{L^2}^{2} \geq \int_{B} (\dist_{x} - \dist_{y})^2 d\mu_{X} \geq 2 (as^b)  s^{2} = 2as^{b+2}. \]
Thus,
\[\| \dist_{x} - \dist_{y} \|_{L^2} \geq \sqrt{2 a}s^{(b/2) + 1}. \]

Setting $T(s,a,b) = \sqrt{2a} s^{(b/2) + 1}$ completes the proof.
\end{proof}

And we obtain a corresponding corollary.

\begin{cor}
	\label{cor:ABcoarseinj}
	There exists a universal function $N(s,a,b): \mathbb{R}_{>0} \times \mathbb{R}_{>0} \times \mathbb{R}_{>0} \to \mathbb{N}$ with the following property.
	Let $(X,\dist_{X},\mu_X)$ be a compact $(a,b)$-standard metric measure space with threshold parameter $r$. For every $s \leq r$, and any $k \geq N(s,a,b)$, if  $\Phi_{k}(x) =\Phi_{k}(y)$ then $\dist_{X}(x,y) \leq 3s$.
\end{cor}

%

%% file: discretization.tex
In this section we prove the approximability of the distance kernel transform for metric measure spaces with bounded geometry, a property of particular importance for computation (see the experimental section~\ref{sec:experiments}).

Building on the work of Koltchinskii and Gin\'e, we prove that the DKE of a metric measure space is approximated by the \emph{empirical} DKE associated to a random sample (defined below): 

\begin{theorem}\label{thm:approxDKE}
For any compact metric measure space $(X,\dist_X,\mu)$ with $(a,b)$-standard Borel measure, and $X_n$ an i.i.d. sample of X of size $n$,
\[
	\toas{\dist_H^{L^2}(\Phi_k(X),\hat{\Phi}_k(X_n))}{0} \ \text{as} \ n \to +\infty, 
\]
where $\dist_H^{L^2}$ is the Hausdorff distance for the $L^2$ norm in $\C^k$.
\end{theorem}

Koltchinskii and Gin\'e~\cite{KonchilskiiGine2000}, then Koltchinskii~\cite{10.1007/978-3-0348-8829-5_11} have established the convergence of spectra and eigenprojections of random empirical operators approximating a Hilbert-Schmidt operator. We define some notions in order to state Koltchinskii and Gin\'e's results, and connect it to our setting. 

In the following, let $(X,\mu)$ be a probability space, and $h$ a symmetric measurable kernel $h \colon X \times X \to \R$ that is square integrable, has trivial diagonal, and defines a Hilbert-Schmidt integral operator $H \colon L_2(X,\mu) \to L_2(X,\mu)$, i.e.,
\[
	\int |h(x,y)|^2 d(\mu \otimes \mu)(x,y) < +\infty,
	\hspace{0.6cm} \forall x \in X, h(x,x) = 0, \hspace{0.6cm} 
	(Hf)(x) := \int h(x,y) f(y) d\mu(y). 
\] 

Let $X_n := \{x_1, \ldots, x_n\}$ be points of $X$ sampled i.i.d. from $\mu$ and defining a probability space $(X_n,\mu_n)$ with uniform measure $\mu_n(x_i) = 1/n$ for all $i$. Let $\hat{H}_n$ be the associated empirical operator $\hat{H}_n \colon L_2(X_n,\mu_n) \to L_2(X_n,\mu_n)$, i.e.,
\[
	(\hat{H}_n f)(x) = \int h(x,y) f(y) d\mu_n(y) = \frac{1}{n}\sum_{i=1}^n h(x,x_i) f(x_i).
\]

Denote by $\sigma(H) = \{ \lambda_i \colon i \in \N_{>0} \}$ the set of eigenvalues (with general multiplicities) of operator $H$, ordered such that $|\lambda_1| \geq |\lambda_2| \geq \ldots$. Similarly, denote by $\sigma(\hat{H}_n) = \{ \hat{\lambda}_i \colon i \in \N_{>0} \}$ the set of eigenvalues of $\hat{H}_n$. 
For a function $f \colon X \to \R$, we denote by $\tilde{f}$ its restriction $\tilde{f} \colon X_n \to \R$.

We define $\mu$-Glivenko-Cantelli classes of functions, which contain functions that are well behaved in the empirical setting:

\begin{definition}
A \emph{$\mu$-Glivenko-Cantelli class} of functions $\mathcal{F} \subset L_2(X)$ is a set of functions that satisfy:
\[
\toas{\displaystyle\sup_{f \in \mathcal{F}} \left| \int f(x) d\mu(x) - \int f(x)d\mu_n(x) \right|}{0} \ \text{as} \ n \to +\infty. 
\]
\end{definition}

We state the convergence theorem:

\begin{theorem}[Koltchinskii and Gin\'e~\cite{KonchilskiiGine2000}, Koltchinskii~\cite{10.1007/978-3-0348-8829-5_11}]\label{thm:KG}
With the above notations, 
\[
	\toas{\sum_{i \in \N_{>0}} |\lambda_i - \hat{\lambda}_i|^2}{0} \ \text{as} \ n \to +\infty 
\]

Additionally, let $\mathcal{F}$ be a class of measurable functions on $X$ with a square integrable envelope $F \in L_2(X,\mu)$, i.e. $|f(x)| \leq F(x)$ for all $x \in X$ and $f \in \mathcal{F}$, such that, for all eigenfunctions $\phi_i$, $i \in \N_{>0}$, of operator $H$, not in the kernel of $H$, the class of functions
\[
	\mathcal{F}\phi_i := \{f \phi_i \colon f \in \mathcal{F} \} \ \text{is $\mu$-Glivenko-Cantelli.}
\]

Then, for $\lambda$ an eigenvalue of $H$ of multiplicity $m$ at distance at least $2 \varepsilon > 0$ from other eigenvalues $\sigma(H) - \{\lambda\}$ and $0$, we have:
\[
	\displaystyle\displaystyle\sup_{f,g \in \mathcal{F}} \left| \langle P_\lambda^\varepsilon(\hat{H}_n) \tilde{f}, \tilde{g} \rangle_{L_2(X_n,\mu_n)} - \langle P_\lambda(H) f, g \rangle_{L_2(X,\mu)} \right|\toas{}{0} \ \text{as} \ n \to +\infty, 
\] 
where $P_\lambda(H) \colon L_2(X,\mu) \to L_2(X,\mu)$ is the projection on the $m$-dimensional space spanned by the eigenfunctions of $H$ of eigenvalue $\lambda$, and $P_\lambda^\varepsilon(\hat{H}_n) \colon L_2(X_n,\mu_n) \to L_2(X_n,\mu_n)$ is the projection on the space spanned by all eigenfunctions of $\hat{H}_n$ of eigenvalues in the interval $[\lambda - \varepsilon ; \lambda + \varepsilon]$. 
\end{theorem}

\medskip

We connect this theorem to our setting. Let $(X,\dist_X,\mu)$ be a compact metric measure space, with $(a,b)$-standard Borel measure $\mu$. We denote its diameter by $\diam X < +\infty$, and we assume, up to rescaling, that $\mu(X) = 1$. 

Let $\Phi_k \colon X \to \C^k$ be the distance kernel embedding truncated at dimension $k$, for $k$ fixed, induced by the distance kernel operator $D^X$ (defined in Section~\ref{sec:distkernelembed}). Denote the eigenvalues and eigenfunctions of $D^X$ by $\{(\lambda_i, \phi_i)\}_{i \in \N_{>0}}$, such that $|\lambda_1| \geq |\lambda_2| \geq \ldots$ following Convention~\ref{con:sign}. We assume all non-zero eigenvalues of $D^X$ have multiplicity one, and $D^X$ has at least $k$ non-zero eigenvalues, i.e., $|\lambda_k| > 0$. 



Any finite subsample $X_n = \{x_1, \ldots , x_n\} \subset X$ induces a finite metric measure space $(X_n, \dist_X, \mu_n)$, where $\dist_X$ is inherited from $X$ by restriction, and $\mu_n$ is the uniform measure $\mu_n(x_i) = 1 / n$. 
We denote by $\hat{\Phi}_k \colon X_n \to \C^k$ the truncated distance kernel transform induced by the distance kernel operator $D^{X_n}$ defined on the discrete metric measure space with uniform measure. Denote by $\{(\hat{\lambda}_i, \hat{\phi}_i)\}_{i \in \N_{>0}}$ the spectrum of $D^{X_n}$, ordered by decreasing eigenvalue moduli.

Note that $X_n$ can be embedded in $\C^k$ using either the empirical embedding $\hat{\Phi}_k$, or using the embedding $\Phi_k$ by restriction to $X_n \subset X$. 

\medskip

We first prove some lemmata. Define the class of functions: 
\begin{equation}\label{eq:familyFGC}
	\mathcal{F} := \{\dist_X(x, \cdot) \colon \forall x \in X\} \cup \{\phi_i \colon i = 1 \ldots k \},
\end{equation} 
from $X$ to $\R$.

\begin{lemma}\label{lem:Fsqintegrable}
The family $\mathcal{F}$ has a square integrable envelope.
\end{lemma}

\begin{proof}
Naturally, the functions $\dist_X(x, \cdot)$ are bounded by $\diam(X) < +\infty$ for all $x \in X$. Any eigenfunction $\phi_i$, for $1 \leq i \leq k$, has $L_2$-norm $1$, and consequently admits a point $x_0$ for which $|\phi_i(x_0)| \leq 1$. Since $\phi_i$ is $1/|\lambda_i|$-Lipschitz for $\lambda_i \neq 0$ (Lemma~\ref{lem:eiglip}), the function is bounded in absolute value by $1 + \diam(X)/|\lambda_i|$. Finally, the constant function $x \mapsto \max \{ \diam(X), 1 + \diam(X)/|\lambda_k| \}$ is square integrable and is an envelope for the family $\mathcal{F}$. 

\end{proof}

\begin{lemma}\label{lem:cvHdistance}
The Hausdorff distance $\toas{\dist_{H}(X, X_n)}{0}$ as $n \to +\infty$.
\end{lemma}

\begin{proof}
Fix $\varepsilon > 0$ arbitrarily small. Since $X$ is compact, it can be covered by finitely many balls of radius $\varepsilon$. Since $\mu$ is $(a,b)$-standard, the measure of these balls is uniformly bounded below. Consequently, the probability that any such ball remains empty when picking larger i.i.d. samples goes to zero. Hence, for any $\varepsilon$, $X_n$ is almost surely $\varepsilon$-dense in $X$ when $n \to +\infty$.

\end{proof}

\begin{lemma}\label{lem:W1convergence}
The $1$-Wasserstein distance $\toas{W_1(\mu, \mu_n)}{0}$ as $n \to +\infty$.
\end{lemma}

\begin{proof}
This is standard in statistics and follows from the fact that on Polish spaces the empirical measure weakly converges to $\mu$ almost surely~\cite{10.2307/25048365} and $W_1$ metrizes weak convergence~\cite[Theorem 6.9]{VillaniOldandNew}.

\end{proof}

\begin{lemma}\label{lem:lip}
If $f$ is $c_1$-Lipschitz and $g$ is $c_2$-Lipschitz, then $fg$ is $(||f||_\infty c_2 + ||g||_\infty c_1)$-Lipschitz.
\end{lemma}

\begin{proof}
By an elementary computation, for any $x,y$ in the domain:
\[\begin{array}{ccr}
|f(x)g(x)-f(y)g(y)| &\leq& |f(x)g(x)-f(y)g(x)| + |f(y)g(x)-f(y)g(y)| \\
                  &    & \leq ||g||_\infty |f(x)-f(y)| + ||f||_\infty |g(x)-g(y)|\\
                  &    & \leq ||g||_\infty c_1 \dist_X(x,y) + ||f||_\infty c_2 \dist_X(x,y).\\
\end{array}
\]

\end{proof}

\begin{lemma}\label{lem-muGC}
For any eigenfunction $\phi_j$ of $H$, such that $H \phi_j \neq 0$, the family $\mathcal{F}\phi_j$ is $\mu$-Glivenko-Cantelli.
\end{lemma}

\begin{proof}
Fix an arbitrary $j$ such that $\phi_j$ of eigenvalue $\lambda_j$ is not in the kernel of $H$, i.e., $\lambda_j \neq 0$. Since the functions $\dist_X(x,\cdot)$ and $\phi_i$, for $i \in \{1, \ldots, k\} \cup \{j\}$, are $c$-Lipschitz, for a uniform constant $c \leq \max \{1, 1/|\lambda_k|, 1/|\lambda_j|\} < +\infty$, and uniformly bounded (see proof of Lemma~\ref{lem:Fsqintegrable}), by Lemma~\ref{lem:lip} their pairwise products are $c'$-Lipschitz for a uniform constant $c' < +\infty$. 

Using the formula for $W_1$ from Kantorovich-Rubinstein duality and the convergence of Lemma~\ref{lem:W1convergence}, we have for any $f \in \mathcal{F}$:
\[
\begin{array}{c}
  \left| \displaystyle\int f(x)\phi_j(x) d(\mu - \mu_n)(x) \right| \leq c' \cdot \sup \left\{ \left|\displaystyle\int h d(\mu-\mu_n)\right| \colon h \ \text{is $1$-Lipschitz} \right\} \hspace{3cm} \\
  \hfill \ \ \ \ = \toas{c' \cdot W_1(\mu,\mu_n)}{0} \ \text{as} \ n \to +\infty.\\
 \end{array}
\]

\end{proof}

We finally prove the main result of the section.


\begin{proof}[Proof of Theorem~\ref{thm:approxDKE}] 




Let $X_n$ be an arbitrary i.i.d. subsample of $X$, of size $n$. Recall that $\{(\lambda_i,\phi_i)\}_{i \in \N_{>0}}$ denotes the eigenvalues and eigenfunctions of $D^X$, and $\{(\hat{\lambda}_i,\hat{\phi}_i)\}_{i \in \N_{>0}}$ denotes the eigenvalues and eigenfunctions of the empirical operator $D^{X_n}$. 

Recall that, by hypothesis, the first $k$ eigenvalues of $D^X$ are non-zero, i.e., $|\lambda_1| \geq \ldots \geq |\lambda_k| > 0$, and of multiplicity $1$. 

Pick any $x \in X$, and any $\hat{x} \in X_n$. Since $\phi_i$ is $1/|\lambda_i|$-Lipschitz, the distance between $\Phi_k(x)$ and $\Phi_k(\hat{x})$ is bounded by:
\[
||\Phi_k(x) - \Phi_k(\hat{x})||_\infty = \max_{j=1, \ldots, k} \sqrt{|\lambda_j|}|\phi_j(x)-\phi_j(\hat{x})| \leq \max_{j=1, \ldots, k} \frac{\dist_X(x,\hat{x})}{\sqrt{|\lambda_j|}} = \frac{\dist_X(x,\hat{x})}{\sqrt{|\lambda_k|}}.
\]

Consequently, 
\begin{equation}\label{eq:proof1}
	\dist_H^{L^2}(\Phi_k(X), \Phi_k(X_n)) \leq \toas{\displaystyle\sqrt{\frac{k}{|\lambda_k|}} \ \dist_H(X,X_n)}{0} \ \text{as} \ n \to +\infty,
\end{equation}
by virtue of Lemma~\ref{lem:cvHdistance}.

Define $\delta := \frac{1}{2} \min_{1 \leq i,j \leq k, i\neq j}\{|\lambda_i-\lambda_j|\} > 0$. 
According to Theorem~\ref{thm:KG} on the convergence of empirical eigenvalues,we have $\hat{\lambda}_i \to \lambda_i$ almost surely as $n \to +\infty$. Therefore, the only eigenvalue of the empirical distance kernel operator $D^{X_n}$ that falls in the interval $[\lambda_i-\delta, \lambda_i + \delta]$ is the eigenvalue $\hat{\lambda_i}$, almost surely as $n \to +\infty$. Consequently, as $n \to +\infty$, we have for any eigenvalue $\lambda_i$, $1 \leq i \leq k$, the projections onto eigenspaces: 
\[
	P_{\lambda_i}(D^X) \colon L^2(X,\mu) \to L^2(X,\mu), \ \ \ \text{and} \ \ \ 
	P_{\lambda_i}^{\delta}(D^{X_n}) \colon L^2(X_n,\mu_n) \to L^2(X_n,\mu_n), 
\]
that are almost surely onto $1$-dimensional spaces, spanned respectively by $\phi_i$ and $\hat{\phi}_i$. 

Hence, for any square integrable functions $f,g \colon X \to \R$, we have the following equality: 
\begin{equation}\label{eq:rewritescalarprod}
\begin{gathered}
\left| \langle P_{\lambda_i}^{\delta}(D^{X_n}) \tilde{f}, \tilde{g} \rangle_{L_2(X_n, \mu_n)} - \langle P_\lambda(D^X) f, g \rangle_{L_2(X,\mu)} \right| \hspace{150pt}\\
\hfill  \hspace{50pt} \overset{\text{a.s.}}{=} 
\left| \langle \hat{\phi}_i , \tilde{f} \rangle_{L_2(X_n,\mu_n)} \langle \hat{\phi}_i, \tilde{g} \rangle_{L_2(X_n,\mu_n)} - \langle \phi_i , f \rangle_{L_2(X,\mu)} \langle \phi_i, g \rangle_{L_2(X,\mu)}  \right| \ \text{as} \ n \to +\infty,\\
\end{gathered}
\end{equation}
where $\tilde{f}, \tilde{g}$ are the restrictions of functions $f,g \colon X \to \R$ to $X_n$.

Additionally, for any two functions $f,g$ in the family $\mathcal{F}$ defined in~(\ref{eq:familyFGC}), Lemma~\ref{lem-muGC} and Theorem~\ref{thm:KG} ensure that:
\begin{equation}\label{eq:convergencescalarprod}
	\toas{\left| \langle P_{\lambda_i}^{\delta}(D^{X_n}) \tilde{f}, \tilde{g} \rangle_{L_2(X_n, \mu_n)} - \langle P_{\lambda_i}(D^X) f, g \rangle_{L_2(X,\mu)} \right|}{0} \ \text{as} \ n \to +\infty.
\end{equation}

Substituting $\phi_i$ for $f$ ($1 \leq i \leq k$) and $\dist_X(\hat{x}, \cdot)$ for $g$ (for an arbitrary point $\hat{x} \in X_n$) in~\eqref{eq:convergencescalarprod}, with $\phi_i,\dist_X(\hat{x}, \cdot) \in \mathcal{F}$, we get:
\begin{equation}\label{eq:KGeq1}
\begin{array}{l}
\left| \langle P_{\lambda_i}^{\delta}(D^{X_n}) \tilde{\phi_i}, \ \widetilde{ \dist_X( \hat{x}, \cdot)} \ \rangle_{L_2(X_n, \mu_n)} - \langle P_{\lambda_i}(D^X) \phi_i, \dist_X(\hat{x}, \cdot) \rangle_{L_2(X,\mu)} \right| \\
\ \ \ \ \ \ \ \ \ \ \ \ \ \ \ 
\overset{\text{a.s.}}{=} \left| \langle \hat{\phi}_i, \tilde{\phi}_i \rangle_{L^2(X_n,\mu_n)} \cdot \langle \hat{\phi}_i, \ \widetilde{\dist_X(\hat{x}, \cdot)} \ \rangle_{L_2(X_n, \mu_n)} - \langle \phi_i, \dist_X(\hat{x}, \cdot) \rangle_{L_2(X,\mu)} \right| \\
\ \ \ \ \ \ \ \ \ \ \ \ \ \ \ \ \ \ \ \ \ \ \ \ \ \ \ \ \ \ \ 
= \toas{\left| \langle \hat{\phi}_i, \tilde{\phi}_i \rangle_{L^2(X_n,\mu_n)} \cdot \hat{\lambda}_i \hat{\phi}_i(\hat{x}) - \lambda_i \phi_i(\hat{x}) \right|}{0} \ \text{as} \ n \to +\infty.\\
\end{array}
\end{equation}
where the first equality follows from~\eqref{eq:rewritescalarprod} and the second equality follows from~\eqref{eqn:coeffs}.

Similarly, for $f = g = \phi_i$, we get:
\begin{equation}\label{eq:KGeq2}
	\toas{\left| \langle \hat{\phi}_i, \tilde{\phi}_i \rangle^2_{L^2(X_n,\mu_n)} - 1 \right|}{0} \ \text{as} \ n \to +\infty.
\end{equation}

We resolve the sign ambiguity of $\langle \hat{\phi}_i, \tilde{\phi}_i \rangle^2_{L^2(X_n,\mu_n)}$. Recall that, by Convention~\ref{con:sign}, the eigenfunctions satisfy:
\begin{equation}\label{eq:convsignambiguity}
\langle \phi_i, |\phi_i| \rangle_{L^2(X,\mu)} > 0, \ \ \text{and} \ \
 \langle \hat{\phi}_i, |\hat{\phi}_i| \rangle_{L^2(X_n,\mu_n)} > 0. 
\end{equation}
Now, suppose by contradiction that $\langle \hat{\phi}_i, \tilde{\phi}_i \rangle_{L^2(X_n,\mu_n)}$ is a.s. negative, in which case it converges a.s. to $-1$ (Equation~(\ref{eq:KGeq2})). Taking the limit in Equation~(\ref{eq:KGeq1}), together with the a.s. convergence of eigenvalues $\hat{\lambda}_i \to \lambda_i$, we get that $\hat{\phi}_i(\hat{x})$ converges a.s. to $-\phi_i(\hat{x})$ on any point $\hat{x}$. In turn, with the weak convergence of measures $\mu_n \to \mu$, this implies:
\[
	\underbrace{\langle \hat{\phi}_i, |\hat{\phi}_i| \rangle_{L^2(X_n,\mu_n)}}_{>0} \toas{}{}  \underbrace{\langle -\phi_i, |\phi_i| \rangle_{L^2(X,\mu)}}_{<0}  \ \text{as} \ n \to +\infty,
\]
in contradiction with the sign convention of Equation~(\ref{eq:convsignambiguity}).

In conclusion, we get that for any index $i$, $1 \leq i \leq k$, and for any point $\hat{x}$ of the sample, both:
\[
	\toas{\hat{\lambda_i}}{\lambda_i} \ \text{and} \ \toas{\hat{\phi}_i(\hat{x})}{\phi_i(\hat{x})}  \ \text{as} \ n \to +\infty,
\]
therefore $\toas{||\Phi_k(\hat{x})-\hat{\Phi}_k(\hat{x})||_{\infty}}{0} \ \text{as} \ n \to +\infty.$





%


This gives:
\begin{equation}\label{eq:proof2}
	\toas{\dist_H^{L^2}(\Phi_k(X_n),\hat{\Phi}_k(X_n))}{0} \ \text{as} \ n \to +\infty.
\end{equation}

We conclude the proof by applying the triangle inequality, and Equations~(\ref{eq:proof1}) and~(\ref{eq:proof2}), 
\[
  \dist_H^{L^2}(\Phi_k(X),\hat{\Phi}_k(X_n)) \leq 
  \toas{\dist_H^{L^2}(\Phi_k(X),\Phi_k(X_n)) + \dist_H^{L^2}(\Phi_k(X_n),\hat{\Phi}_k(X_n))}{0} \ \text{as} \ n \to +\infty.
\]
%
%

\end{proof}

%% file: stability.tex
In this section, we focus on compact Riemannian manifolds, with the genericity assumption that the non-trivial eigenvalues $\{\lambda_i\}_{i \in \N_{>0}}$ of the distance kernel operator satisfy $|\lambda_i| \neq |\lambda_j|$ whenever $i \neq j$. Consequently, the non-trivial eigenvalues can be labeled such that $|\lambda_1| > \ldots > |\lambda_k| > \ldots$. In this framework, the DKE is geometrically stable:



\begin{theorem}\label{thm:stability}
Let $(X,\dist_X,\mu)$ and $(Y,\dist_Y,\eta)$ be compact finite-dimensional Riemannian manifolds equipped with their volume measures, such that $\mu(X) = \eta(Y)$. Let $|\lambda_1| > \ldots > |\lambda_k| > 0$ and $|\nu_1| > \ldots > |\nu_k| > 0$ be the $k$ largest eigenvalues in absolute value of operators $D^X$ and $D^Y$ respectively, all non-trivial, of multiplicity one, and with distinct absolute values. Let $\Phi_k\colon X \to \mathbb{C}^k$ and $\Psi_k\colon Y\to \mathbb{C}^k$ be the DKE for $D^X$ and $D^Y$ respectively. 

\medskip

Define the separation $\Delta_k^{X,Y}$ of the spectra of operators $D^X$ and $D^Y$ by:
\[
	\Delta_k^{X,Y} = \min \left\{ |\lambda_i^2-\nu_j^2| \ : \ 1 \leq i,j \leq k, i \neq j \right\},
\]
and let:
\[
	|\tau_1| := \min \{|\lambda_1|,|\nu_1|\}, \ \ \text{and} \ \ |\tau_k| := \max \{|\lambda_k|,|\nu_k|\}.
\]

\medskip

Then, 
\[
	\dist_H^{L^2}(\Phi_k(X),\Psi_k(Y)) \leq \displaystyle\sqrt{k}\left( 4\sqrt{2} \frac{(\varepsilon+|\tau_1|)}{\Delta^{X,Y}_k}\sqrt{|\tau_1|}\right) \cdot \varepsilon + 
	\displaystyle\sqrt{k}\left( 2\sqrt{2} \sqrt{\frac{(\varepsilon + |\tau_1|)}{|\tau_k|}}\right) \cdot \sqrt{\varepsilon}, 
\]
where $\varepsilon := \dist_{G\bar{P}}(X,Y)$ stands for the modified Gromov-Prokhorov distance between $X$ and $Y$ (see Definitions~\ref{def:modifiedPdistance} and~\ref{def:modifiedGHPdistance} below).
\end{theorem}






\begin{remark}
We restrict the stability result to Riemannian manifolds, in order to be able to compare measures with a modified version of the Prokhorov distance that admits a simple interpretation on discrete spaces (Lemma~\ref{lem:matching}). Contrarily to the Prokhorov distance, this distance is not suitable to study general metric measure spaces. However, in the case of metric spaces with $(a,b)$-standard measure and minimal geodesics between any pair of points---and in particular compact Riemannian manifolds with their volume measure---the two distances have similar properties (Lemma~\ref{lem:mprok}).  
\end{remark}


%

Consider the distance operators on, respectively, $L^2(X)$ and $L^2(Y)$,
\[
	(D^{X}f)(x) = \int_X f(z) \dist_X(x,z) \ d\mu(z) \ \ \ \ \ \ (D^{Y}f)(y) = \int_Y f(z) \dist_Y(y,z) \ d\eta(z).
\] 

We recall some definitions from measure theory.

\begin{definition}
The {\em Prokhorov distance} $\dist_P(\eta, \mu)$ between two Borel measures $\eta, \mu$ on a common metric space $(Z,\dist_Z)$ with Borel sigma algebra $\mathcal{B}(Z)$ is defined by:
\[
\dist_P(\eta, \mu) := \inf \{\varepsilon \geq 0 : \eta(A) \leq \mu(A^\varepsilon)+\varepsilon \ \text{and} \ \mu(A) \leq \eta(A^\varepsilon) + \varepsilon, \ \text{for all} \ A \in \mathcal{B}(Z) \},
\]
where $A^\varepsilon$ denotes the $\varepsilon$-neighborhood of $A$ in $Z$, i.e.,
\[
	A^\varepsilon := \{z \in Z \colon \inf_{a \in A} \dist_Z(a,z) \leq \varepsilon \}.
\]
\end{definition}

Here we consider the following modified version of the Prokhorov distance:

\begin{definition}\label{def:modifiedPdistance}
The modified Prokhorov distance $\dist_{\bar{P}}(\eta, \mu)$ between two measures Borel measures $\eta, \mu$ on a common metric space $(Z,\dist_Z)$ with Borel sigma algebra $\mathcal{B}(Z)$, such that $\mu(Z)=\eta(Z)$, is defined by:
\[
\dist_{\bar{P}}(\eta, \mu) := \inf \{\varepsilon \geq 0 : \eta(A) \leq \mu(A^\varepsilon) \ \text{and} \ \mu(A) \leq \eta(A^\varepsilon), \ \text{for all} \ A \in \mathcal{B}(Z) \}.
\]
\end{definition}

This distance naturally satisfies the triangle inequality. 
On compact Riemannian manifolds, notions of convergence for Prokhorov and modified Prokhorov distances coincide:

\begin{lemma}[Burago et al~{\cite[Lemma 8.2]{burago2013laplacebeltrami}}]
	\label{lem:mprok}
	Let $M$ be a $d$-dimensional Riemannian manifold, and $\mu$ a Borel measure on $M$ such that $\mu(M)=\vol_M(M)$. Then,
	\[
	\dist_P(\mu,\vol_M) \leq \dist_{\bar{P}}(\mu,\vol_M) \leq C_M \dist_P(\mu,\vol_M)^{\frac{1}{d}},
	\]
	for a constant $C_M$ depending only on the geometry of $M$.
\end{lemma}


Our stability result is stated in terms of the modified Gromov-Prokhorov distance, capturing both metric and measure similarity:

\begin{definition}\label{def:modifiedGHPdistance}
The {\em modified Gromov-Prokhorov distance} between two metric measure spaces $(X, \dist_X, \mu)$ and $(Y,\dist_Y,\eta)$ with $\mu(X)=\eta(Y)$ is:
\[
\dist_{G\bar{P}}(X,Y) := \inf_{\iota,\kappa} \{\dist_{\bar{P}}(\iota\#\mu, \kappa\#\eta)\},
\]
where the infimum is taken over isometries $\iota\colon X \to \iota(X) \subseteq Z$ and $\kappa\colon Y \to \kappa(Y) \subseteq Z$ into a common metric space $(Z,\dist_Z)$.
\end{definition}

Naturally, the modified Gromov-Prokhorov distance is larger than the Gromov-Prokhorov distance (substituting $\dist_P$ for $\dist_{\bar{P}}$ in the definition above). It is also larger than the Gromov-Hausdorff distance between two Riemannian manifolds equipped with their volume measures. Indeed, for two isometries $\iota$ and $\kappa$, the measure $\iota\#\mu$ of an arbitrary small ball centered at a point $x \in \iota(X)$ is non-zero, which implies that the volume $\kappa\#\eta$ of this ball inflated by $\dist_{\bar{P}}(\iota\#\mu,\kappa\#\eta)$ is non-zero, and hence contains a point in $\kappa(Y)$. 

We prove that isometries do not affect the spectral theory of the distance operators. Specifically:

\begin{prop}\label{prop:samesame}
	Let $\iota \colon X \to Z$ be an isometry of a compact metric measure space $(X,\dist_X,\mu)$ onto its image $\iota(X) := X' \subset Z$. Then, there exists a bijection $\iota^* \colon L^2(X) \to L^2(X')$ such that the following diagram commutes:
	\[
	\xymatrix {
		L^2(X) \ar@{->}[r]^{\iota^*} \ar@{->}[d]_{D^X} & L^2(X') \ar@{->}[d]^{D^{X'}}\\
		L^2(X) \ar@{->}[r]^{\iota^*} & L^2(X')\\}
	\]
	where $D^{X'}$ is the distance kernel operator associated to the space $(X',\dist_Z,\iota\#\mu)$.

	In particular, $\phi$ is an eigenfunction for $D^X$ of eigenvalue $\lambda$ iff $\iota^* \phi$ is an eigenfunction for $D^{X'}$ of eigenvalue $\lambda$, and, denoting the DKE for $X$ and $X'$ by $\Phi_k \colon X \to \C^k$ and $\Phi'_k \colon X' \to \C^k$ respectively, we have:
	\[
	  \Phi_k(X) = \Phi'_k(X').
	 \]
\end{prop}

\begin{proof}
	Define $\iota^* \colon f \mapsto f \circ \iota^{-1}$. It is naturally a bijection. By a change of variable with the pushforward measure, and the fact that $\iota$ is an isometry, we get:
	\begin{align*}
	(D^{X'} (\iota^* f))(x') &= \int_{X'} f(\iota^{-1}(z)) \dist_Z(x',z) d \iota\# \mu(z)\\
	&= \int_X f(y) \dist_Z(x',\iota(y)) d \mu(y)\\
	&= \int_X f(y) \dist_X(\iota^{-1}(x'),y) d\mu(y) = (D^X f)(\iota^{-1}(x')) = (\iota^*(D^X f))(x')\\ 
	\end{align*}
	The spectral consequence follows. Now, let $|\lambda_1| > \ldots > |\lambda_k| > 0$ be the eigenvalues of largest absolute value for $D^X$ and $D^{X'}$. Let $\phi_1, \ldots, \phi_k$ be the corresponding eigenfunctions for $D^X$ used to define the DKE of $X$, i.e., such that $\langle \phi_i,|\phi_i| \rangle > 0$. Let $\phi_1', \ldots, \phi_k'$ be the eigenfunctions for $D^{X'}$ such that $\phi'_i = \phi \circ \iota^{-1}$. By the same change of variable as above, we have $\langle \phi'_i,|\phi'_i| \rangle = \langle \phi_i,|\phi_i| \rangle > 0$, therefore the $\phi_i'$ are used to define the DKE of $X'$.

	Then, for any $x \in X$, we have: 
	\[
	   \sqrt{\lambda_i} \phi_i(x) = \sqrt{\lambda_i} \phi_i \circ \iota^{-1}( \iota(x)), \forall i \in \{1, \ldots , k\}, \ \ \text{hence} \ \ \Phi_k(x) = \Phi'_k(\iota(x)).
	\]

	The map $\iota$ being a bijection, we conclude that $\Phi_k(X) = \Phi'_k(X')$.
\end{proof}




We first focus on finite metric subspaces of a common embedding space $(Z,\dist_Z)$. Recall Weyl's and Davis-Kahan's perturbation theorems for symmetric real matrices with multiplicity $1$ eigenvalues, as formulated for example in~\cite{DBLP:conf/alt/EldridgeBW18}:

\begin{theorem}[Weyl~\cite{Weyl1912}, Davis-Kahan~\cite{davis69}]
	\label{thm:weyl}
	Let $M$ and $M+H$ be symmetric real matrices with multiplicity $1$ eigenvalues, and let $\lambda_1, \ldots, \lambda_n$ be the spectrum of $M$, such that $\lambda_1 > \ldots > \lambda_n$, and let $\nu_1, \ldots ,\nu_n$ be the spectrum of $M+H$, such that $\nu_1 > \ldots > \nu_n$, i.e., eigenvalues are ordered by decreasing real values. 

\smallskip

	For $i \in \{1, \ldots ,n\}$ let $u_i$ and $v_i$ be eigenvectors of respectively $M$ and $M+H$, such that $M u_i = \lambda_i u_i$ and $(M+H) v_i = \nu_i v_i$. Finally, let $\theta_i$ be the angle formed by the lines spanned by $u_i$ and $v_i$. 

\smallskip

	Define $\varepsilon := ||H||$ the spectral norm of $H$ (i.e., the maximal absolute value of an eigenvalue). Then, for all $i \in \{1, \ldots , n\}$:
	
\smallskip

	\begin{enumerate}[(i)]
		\itemsep0.5em   
		\item $|\lambda_i - \nu_i| \leq \varepsilon$,
		\item $\sin \theta_i \leq \varepsilon / \Delta'_i$, with $\Delta'_i = \min \{ |\lambda_i - \nu_j| : j=1 \ldots n, \ j \neq i \}$.
	\end{enumerate}
\end{theorem}

In its usual form, Weyl's and Davis-Kahan's theorems use a different ordering of eigenvalues, compared to Convention~\ref{con:sign} for the DKE. We therefore adapt their result to our setting:

\begin{theorem}[Weyl \& Davis-Kahan, revisited]
	\label{thm:weylrevisited}
	Let $M$ and $M+H$ be symmetric real matrices with eigenvalues of distinct absolute value, and let $\lambda_1, \ldots , \lambda_n$ be the spectrum of $M$, such that $|\lambda_1| > \ldots > |\lambda_n|$, and let $\nu_1, \ldots ,\nu_n$ be the spectrum of $M+H$, such that $|\nu_1| > \ldots > |\nu_n|$, i.e., eigenvalues are ordered by {\em decreasing absolute values}. 

\smallskip

	For $i \in \{1, \ldots ,n\}$ let $u_i$ and $v_i$ be eigenvectors of respectively $M$ and $M+H$, such that $M u_i = \lambda_i u_i$ and $(M+H) v_i = \nu_i v_i$. Finally, let $\theta_i$ be the angle formed by the lines spanned by $u_i$ and $v_i$. 

\smallskip

	Define $\varepsilon := ||H||$, the spectral norm of $H$. Then, for all $i \in \{1, \ldots , n\}$: 

\smallskip

	\begin{enumerate}[(i)]
	\itemsep0.5em  
		\item $|\lambda_i^2 - \nu_i^2| \leq \varepsilon (\varepsilon + 2 |\lambda_1|)$,
		\item $\sin \theta_i \leq \varepsilon (\varepsilon + 2 |\lambda_1|) / \Delta_i$, where $\Delta_i = \min \{ |\lambda_i^2 - \nu_j^2| : j=1 \ldots n, \ j \neq i \}$.
	\end{enumerate}
\end{theorem}

\begin{proof}
Note that an eigenvector $u$ of a symmetric real matrix $A$, such that $Au = \lambda u$, is an eigenvector of $A^2$ of eigenvalue $\lambda^2$. Consequently, $M^2$ is a real symmetric matrix of spectrum $\lambda_1^2 > \ldots > \lambda_n^2$, all of multiplicity one, and eigenvectors $u_1, \ldots, u_n$, and similarly for $(M+H)^2$ with spectrum $\nu_1^2 > \ldots > \nu_n^2$ and eigenvectors $v_1, \ldots, v_n$.

Note also that $(M+H)^2 = M^2 + H^2 + MH + HM$, and $||H^2 + MH + HM|| \leq ||H^2|| + ||MH|| + ||HM|| = ||H||^2 + 2||M|| ||H|| = \varepsilon (\varepsilon + 2|\lambda_1|)$. Besides, $H^2$ and $MH + HM$ are real symmetric, and so is their sum. Applying Theorem~\ref{thm:weyl} to $M^2$, with perturbation $(H^2+MH+HM)$, 
concludes the proof. 
\end{proof}

We use the following notions:

\begin{definition}
Let $X_n = \{x_1, \ldots, x_n\}$ be a discrete space. For a constant $c > 0$, we call $\mu_n$ the \emph{$c$-uniform measure} on $X_n$ the measure $\mu_n(x_i) = c / n$ for all $x_i \in X_n$. It is the empirical measure normalized such that the measure of the total space $X_n$ is $c$. 
\end{definition}

\begin{definition}
For a metric space $(Z, \dist_Z)$, define the {\em bottleneck distance} between two finite samples $X_n, Y_n \subset Z$ of same size $n$ to be:
\[
	\dist_B(X_n,Y_n) = \min_{\text{bij.} \ \pi \colon X_n \to Y_n \ \ \ } \max_{x \in X_n} \ \dist_Z(x, \pi(x)). 
\]
\end{definition}

We prove the following simple statements about finite metric spaces with uniform measure. 

\begin{lemma}\label{lem:matching}
	Let $X_n, Y_n \subset Z$ be samples of same size. For $c > 0$, consider $\mu_n$ and $\eta_n$ to be the $c$-uniform measures on $X_n$ and $Y_n$ respectively, with $\mu_n(X_n) = \eta_n(Y_n) = c$. Then:
	\[
		\dist_B(X_n,Y_n) = \dist_{\bar{P}}(\mu_n, \eta_n).
	\]
\end{lemma}

\begin{proof}
	Let $\varepsilon := \dist_{\bar{P}}(\mu_n, \eta_n)$, and consider the bipartite graph on $X_n \sqcup Y_n$ where $(x,y)$ is an edge iff $\dist_Z(x,y) \leq \varepsilon$. By Hall's marriage theorem for bipartite graphs, the graph admits a bipartite matching of size $n$ iff every subset $A \subset X_n$ has at least $|A|$ neighbors in $Y_n$, and vice versa. This condition is equivalent to $\mu_n(A) \leq \eta_n(A^\varepsilon)$, and $\eta_n(A) \leq \mu_n(A^\varepsilon)$, from the modified Prokhorov metric applied to the $c$-uniform measures $\mu_n$ and $\eta_n$.
\end{proof}

Additionally, the modified Prokhorov distance of samples converges:

\begin{lemma}\label{lem:convmodifPmetric}
Let $(X,\mu)$ be a compact Riemannian manifold equipped with its volume measure, and $(X_n, \mu_n)$ an i.i.d. sample of size $n$ with $\mu(X)$-uniform measure, such that $\mu_n(X_n) = \mu(X)$. Let $\iota \colon X \to \iota(X) \subseteq Z$ be an isometry onto a metric space $(Z,\dist_Z)$. We have:
\[
	\toas{\dist_{\bar{P}}(\iota\#\mu,\iota\#\mu_n)}{0} \ \ \text{as} \ \ n \to +\infty.
\]
\end{lemma}

\begin{proof}
On Polish spaces (such as $(X,\mu)$), $\mu_n$ weakly converges to $\mu$ almost surely~\cite{10.2307/25048365} and the Prokhorov distance metrizes weak convergence~\cite[Chapter 6]{VillaniOldandNew}. Consequently:
\[
	\toas{\dist_{P}(\mu,\mu_n)}{0} \ \ \text{as} \ \ n \to +\infty.
\]	 
The convergence of the modified Prokhorov distance in the Riemannian manifold case follows from the bounds given in Lemma~\ref{lem:mprok}. Hence,
\[
	\toas{\dist_{\bar{P}}(\mu,\mu_n)}{0} \ \ \text{as} \ \ n \to +\infty.
\]
Now, we prove that $\dist_{\bar{P}}(\mu,\mu_n) = \dist_{\bar{P}}(\iota\#\mu,\iota\#\mu_n)$. It is sufficient to prove that for any measurable set $A \subseteq X$ and $\varepsilon > 0$, we have:
\[
	\iota\#\mu(\iota(A^\varepsilon)) = \iota\#\mu( \iota(A)^\varepsilon ) \ \ \text{and} \ \ \iota\#\mu_n(\iota(A^\varepsilon)) = \iota\#\mu_n( \iota(A)^\varepsilon ).
\]
%
These equalities follow from the facts that $\iota(A^\varepsilon) = \iota(A)^\varepsilon \cap \iota(X)$ (because $\iota$ is an isometry), and that $\iota\#\mu (\iota(A)^\varepsilon \cap \iota(X)) = \iota\#\mu (\iota(A)^\varepsilon)$ and $\iota\#\mu_n (\iota(A)^\varepsilon \cap \iota(X)) = \iota\#\mu_n (\iota(A)^\varepsilon)$ (because the supports of $\iota \# \mu$ and $\iota \# \mu_n$ are contained in $\iota(X))$.
\end{proof}

\medskip

The key step in the proof of Theorem~\ref{thm:stability} is the following stability result for discrete metric measure spaces:

\begin{theorem}\label{lem:discr_conv}
	Let $(X_n, \dist_Z, \mu_n)$ and $(Y_n, \dist_Z, \eta_n)$ be discrete metric measure spaces with $c$-uniform measures associated to samples $X_n$ and $Y_n$ of same size $n$, with $\mu_n(X_n) = \eta_n(Y_n) = c > 0$, in a common metric space $(Z,\dist_Z)$.

	Let $|\hat{\lambda}_1| > \ldots > |\hat{\lambda}_k| > 0$ and $|\hat{\nu}_1| > \ldots > |\hat{\nu}_k| > 0$ be the $k$ largest eigenvalues in absolute value of operators $D^{X_n}$ and $D^{Y_n}$ respectively, all non-trivial with distinct absolute values. Let $\hat{\Phi}_k\colon X_n \to \mathbb{C}^k$ and $\hat{\Psi}_k\colon Y_n\to \mathbb{C}^k$ be the DKEs associated with $D^{X_n}$ and $D^{Y_n}$ respectively.


	\medskip

	Define the separation $\Delta_k^{X_n,Y_n}$ of the spectra of operators $D^{X_n}$ and $D^{Y_n}$ by:
	\[
		\Delta^{X_n,Y_n}_k := \min\{ |\hat{\lambda}^2_i - \hat{\nu}^2_j| : 1\leq i,j \leq k \ , \ i \neq j \}, 
	\] 
	and let: 
	\[
		|\hat{\tau}_1| := \min \{|\hat{\lambda}_1|,|\hat{\nu}_1|\}, \ \ \text{and} \ \ |\hat{\tau}_k| := \max \{|\hat{\lambda}_k|,|\hat{\nu}_k|\}.
	\]

	\medskip

	\noindent
	Then, 
\[
	\dist_H^{L^2}(\hat{\Phi}_k(X_n),\hat{\Psi}_k(Y_n)) \leq \displaystyle\sqrt{k}\left( 4\sqrt{2} \frac{(\varepsilon+|\hat{\tau}_1|)}{\Delta^{X_n,Y_n}_k}\sqrt{|\hat{\tau}_1|}\right) \cdot \varepsilon + 
	\displaystyle\sqrt{k}\left( 2\sqrt{2} \sqrt{\frac{(\varepsilon + |\hat{\tau}_1|)}{|\hat{\tau}_k|}}\right) \cdot \sqrt{\varepsilon}, 
\]
where $\varepsilon := \dist_{\bar{P}}(X_n,Y_n)$ stands for the modified Prokhorov distance between $X_n$ and $Y_n$.


\end{theorem}

\begin{proof}
	Write $\varepsilon := \dist_{\bar{P}}(X_n,Y_n)$. By equivalence of modified Prokhorov distance and bottleneck distance (Lemma~\ref{lem:matching}), there exists a labeling $X_n = \{x_1, \ldots, x_n\}$ and $Y_n = \{y_1, \ldots, y_n\}$ such that $\dist_Z(x_i,y_i) \leq \varepsilon$ for all $i = 1 \ldots n$. Let $D^{X_n} := \frac{1}{n}\bar{D}^{X_n}$ and $D^{Y_n} := \frac{1}{n}\bar{D}^{Y_n}$ be the distance matrices $\bar{D}^{X_n}$ and $\bar{D}^{Y_n}$ of $X_n$ and $Y_n$ respectively, with points ordered as above, scaled by a factor $\frac{1}{n}$. They represent the distance kernel operators, considering $L^2(X_n) = \R^n = L^2(Y_n)$. The difference matrix $H = D^{Y_n} - D^{X_n}$ is symmetric. It has entries $h_{ij} = \frac{1}{n}(d(y_i,y_j) - d(x_i,x_j))$, which can be bounded above in absolute value by $2\varepsilon/n$ using the triangle inequality.
	
	The spectral norm of $H$ is consequently at most $2\varepsilon$. Indeed, let $u$ be a non-zero eigenvector of $H$, with $Hu = \gamma u$. Let $|u_i|$ be the largest entry of $u$ in absolute value, i.e., $|u_i| = ||u||_\infty > 0$. We have:
	\begin{equation}
	|\gamma| = \frac{||H u||_\infty}{||u||_\infty} = \max_k \left|\sum_{j=1}^n h_{kj} \frac{u_j}{u_i}\right| \leq \max_k \sum_{j=1}^n |h_{kj}| \left|\frac{u_j}{u_i}\right| \leq 2\varepsilon.
	\end{equation}

Hence:
\begin{equation}\label{eq:twoepsilonboundmatrixnorm}
	||D^{Y_n} - D^{X_n}|| = || H || \leq 2 \varepsilon.
\end{equation}

	Denote by $\hat{\phi}_j$ a unit norm eigenvector of eigenvalue $\hat{\lambda}_j$ of $D^{X_n}$, and by $\hat{\psi}_j$ a unit norm eigenvector of eigenvalue $\hat{\nu}_j$ of $D^{Y_n}$. 
	Fix an index $i$, $1 \leq i \leq n$, and consider the points $x_i$ and $y_i$. By definition, the $j^{\text{th}}$ coordinate, $1 \leq j \leq k$, of $\hat{\Phi}_k(x_i)$ is equal to the $i^{\text{th}}$ coordinate of the eigenvector $\hat{\phi}_j$ times $\sqrt{\hat{\lambda}_j}$. Similarly for $\hat{\Psi}_k(y_i)$ and $\hat{\psi}_j$. We bound:
\[
	||\hat{\Phi}_k(x_i) - \hat{\Psi}_k(y_i)||_\infty = \max_{j=1 \ldots k} \left|\left[\hat{\Phi}_k(x_i)\right]_j - \left[\hat{\Psi}_k(y_i)\right]_j\right| = \max_{j=1 \ldots k} \left|\sqrt{\hat{\lambda}_j}\left[\hat{\phi}_j\right]_i - \sqrt{\hat{\nu}_j}\left[\hat{\psi}_j\right]_i \right|.
\]

First, note that for any $j$, denoting by $\theta_j$ the angle formed by the lines spanned by eigenvectors $\hat{\phi}_j$ and $\hat{\psi}_j$, we have:

\begin{equation}\label{eq:boundinfuv}
\begin{array}{rcl}
	||\hat{\phi}_j-\hat{\psi}_j||_\infty^2 \leq ||\hat{\phi}_j-\hat{\psi}_j||_2^2 &=& \langle \hat{\phi}_j,\hat{\phi}_j \rangle + \langle \hat{\psi}_j,\hat{\psi}_j \rangle - 2\langle \hat{\phi}_j,\hat{\psi}_j \rangle \\
	&=& 2(1-\cos \theta_j) \ \leq \ 2(1-\cos^2 \theta_j) \ = \ 2 \sin^2 \theta_j.
\end{array}
\end{equation}


Then:   
\[
\displaystyle	\left|\sqrt{\hat{\lambda}_j}\left[\hat{\phi}_j\right]_i - \sqrt{\hat{\nu}_j}\left[\hat{\psi}_j\right]_i \right| \leq 
	  \left|\sqrt{\hat{\lambda}_j}\right| \cdot \underbrace{\left|\left[\hat{\phi}_j\right]_i - \left[\hat{\psi}_j\right]_i\right|}_{\leq \sqrt{2}|\sin \theta_j| \ \text{by Eq.~(\ref{eq:boundinfuv})}}  \ \ + \ \  \underbrace{\left|\left[\hat{\psi}_j\right]_i\right|}_{\leq ||\hat{\phi}_j||_2 = 1} \cdot \left|\sqrt{\hat{\lambda}_j} - \sqrt{\hat{\nu}_j}\right|.
\]
Exchanging the roles of $\hat{\lambda}_j$ and $\hat{\nu}_j$ we get a similar bound by symmetry, and thus:
\begin{equation}\label{eq:takethemax}
	\displaystyle	\left|\sqrt{\hat{\lambda}_j}\left[\hat{\phi}_j\right]_i - \sqrt{\hat{\nu}_j}\left[\hat{\psi}_j\right]_i \right| \leq \sqrt{2 \min \{|\hat{\lambda}_j|,|\hat{\nu}_j|\}} \cdot |\sin \theta_j | + \underbrace{\left|\sqrt{\hat{\lambda}_j} - \sqrt{\hat{\nu}_j}\right|}_{(\star)}.
\end{equation}

Recall the notations $|\hat{\tau}_1| := \min \{|\hat{\lambda}_1|,|\hat{\nu}_1|\}$, and $|\hat{\tau}_k| := \max \{|\hat{\lambda}_k|,|\hat{\nu}_k|\}$.

We now bound the term $(\star)$. By hypothesis, both matrices $D^{X_n}$ and $D^{Y_n}$ have eigenvalues with distinct absolute values. Considering the bound on $||H||$ from Equation~(\ref{eq:twoepsilonboundmatrixnorm}), we apply Weyl's (revisited) Theorem~\ref{thm:weylrevisited} to matrices $(D^{X_n}, D^{X_n}+H = D^{Y_n})$, and to matrices $(D^{Y_n}, D^{Y_n}-H)$, to get the bounds, for all $1 \leq j \leq k$:
\[
	\displaystyle\left| |\hat{\lambda}_j| - |\hat{\nu}_j| \right| \left| |\hat{\lambda}_j| + |\hat{\nu}_j| \right| = |\hat{\lambda}^2_j - \hat{\nu}^2_j| \leq 4 \varepsilon (\varepsilon + |\hat{\tau}_1|).
\]
Consequently:
\[
	\displaystyle\left| |\hat{\lambda}_j| - |\hat{\nu}_j| \displaystyle\right| \leq \frac{4 \varepsilon (\varepsilon + |\hat{\tau}_1|)}{\displaystyle\left| |\hat{\lambda}_j| + |\hat{\nu}_j| \displaystyle\right|} \leq \frac{4 \varepsilon (\varepsilon + |\hat{\tau}_1|)}{|\hat{\tau}_k|},
\]
and denote this bound by $\varepsilon' := \frac{4 \varepsilon (\varepsilon + |\hat{\tau}_1|)}{|\hat{\tau}_k|}$.

\smallskip

We distinguish between the following two cases. 
First, suppose that at least one of $\hat{\lambda}_j$ or $\hat{\nu}_j$ has absolute value larger than $\varepsilon'$. In this case, $\hat{\lambda}_j$ and $\hat{\nu}_j$ have the same sign, and: 
\[
	(\star) 
	= 
	\left|\sqrt{\hat{\lambda}_j} - \sqrt{\hat{\nu}_j}\right| = \left|\sqrt{|\hat{\lambda}_j|} - \sqrt{|\hat{\nu}_j|}\right| 
	= 	
	\displaystyle\frac{\displaystyle\left| |\hat{\lambda}_j|-|\hat{\nu}_j| \displaystyle\right|}{ \sqrt{|\hat{\lambda}_j|} + \sqrt{|\hat{\nu}_j|} } 
	\leq 
	\frac{\varepsilon'}{\sqrt{\varepsilon'}} \leq \sqrt{\varepsilon'}.
\]

Otherwise, both $\hat{\lambda}_j$ and $\hat{\nu}_j$ have absolute value smaller than $\varepsilon'$, and:
\[
	(\star) = \left|\sqrt{\hat{\lambda}_j} - \sqrt{\hat{\nu}_j}\right| \leq 
	\sqrt{|\hat{\lambda}_j|+|\hat{\nu}_j|} \leq \sqrt{2 \varepsilon'}.
\]

In general, $(\star)$ is consequently bounded above by:
\[
	(\star) \leq \sqrt{2\varepsilon'} = \displaystyle\left( 2\sqrt{2} \sqrt{\frac{\varepsilon + |\hat{\tau}_1|)}{|\hat{\tau}_k|}}\right) \cdot \sqrt{\varepsilon}.
\]



%


\medskip

Now, using the quantity $\Delta^{X_n,Y_n}_k$, we get the following bound on $\sin \theta_j$, for any $j=1 \ldots k$, using Davis-Kahan's (revisited) Theorem~\ref{thm:weylrevisited} with $||D^{X_n} - D^{Y_n}|| \leq 2\varepsilon$:
\[
	\sin \theta_j \leq \displaystyle\frac{4 \varepsilon (\varepsilon+ |\hat{\tau}_1|)}{\Delta^{X_n,Y_n}_k}.
\]

\medskip

In conclusion, 
we get from Equation~(\ref{eq:takethemax}), passing to the max in $j$, that for every pair of points $(x_i,y_i)$,
\[
	||\hat{\Phi}_k(x_i) - \hat{\Psi}_k(y_i)||_\infty \leq 
	\displaystyle\left( 4\sqrt{2} \frac{(\varepsilon+|\hat{\tau}_1|)}{\Delta^{X_n,Y_n}_k} \sqrt{|\hat{\tau}_1|}\right) \cdot \varepsilon + 
	\displaystyle\left( 2\sqrt{2} \sqrt{\frac{(\varepsilon + |\hat{\tau}_1|)}{|\hat{\tau}_k|}}\right) \cdot \sqrt{\varepsilon}.
%
\]

	And finally,
	\begin{align*}
	\dist_H^{L^2}(\hat{\Phi}_k,\hat{\Psi}_k) \leq \max_{i=1\ldots n} \{||\hat{\Phi}_k(x_i) - \hat{\Psi}_k(y_i)||_2\} \leq \sqrt{k} \max_{i=1\ldots n} \{||\hat{\Phi}_k(x_i) - \hat{\Psi}_k(y_i)||_\infty\} \ \ \ \ \ \ \ \ \ \ \ \ \ \ \ \ \ \\
	\hfill \leq 
	\displaystyle\sqrt{k}\left( 4\sqrt{2} \frac{(\varepsilon+|\hat{\tau}_1|)}{\Delta^{X_n,Y_n}_k}\sqrt{|\hat{\tau}_1|}\right) \cdot \varepsilon + 
	\displaystyle\sqrt{k}\left( 2\sqrt{2} \sqrt{\frac{(\varepsilon + |\hat{\tau}_1|)}{|\hat{\tau}_k|}}\right) \cdot \sqrt{\varepsilon}.
	\end{align*}
%
\end{proof}

We conclude the section by proving Theorem~\ref{thm:stability}:

\begin{proof}[Proof of Theorem~\ref{thm:stability}]
Let $(X,\dist_X,\mu)$ and $(Y,\dist_Y,\eta)$ be compact Riemannian manifolds equipped with their volume measures, with $\mu(X)=\eta(Y)$. Denote by $\{\lambda_i\}_{i\in \N_{>0}}$ and $\{\nu_i\}_{i\in \N_{>0}}$ the spectra, ordered by decreasing absolute values, of the distance kernel operators $D^X$ and $D^Y$ respectively, and by $\Phi_k(X)$ and $\Psi_k(Y)$ the corresponding DKEs. 

Let $X_n$ and $Y_n$ be i.i.d. samples of $(X,\mu)$ and $(Y,\eta)$ respectively, both of size $n$. Consider the metric measure spaces $(X_n,\dist_X,\mu_n)$ and $(Y_n,\dist_Y,\eta_n)$ equipped with the $c$-uniform measure, with $c=\mu(X)=\eta(Y)$. 
Let $\{\hat{\lambda}_i\}_{i\in \N_{>0}}$ and $\{\hat{\nu}_i\}_{i\in \N_{>0}}$ be the spectra, ordered by decreasing absolute values, of the associated distance kernel operators $D^{X_n}$ and $D^{Y_n}$ respectively, and let $\hat{\Phi}_k(X_n)$ and $\hat{\Psi}_k(Y_n)$ be the corresponding DKEs. 

\medskip

By the triangle inequality, we have: 
\begin{equation}\label{eq:boundprokdistviaempiralembed}
\begin{array}{l}
 \dist_H^{L^2}(\Phi_k(X),\Psi_k(Y)) \leq \dist_H^{L^2}(\Phi_k(X),\hat{\Phi}_k(X_n))  + \dist_H^{L^2}(\hat{\Phi}_k(X_n),\hat{\Psi}_k(Y_n)) 
	    + \dist_H^{L^2}(\hat{\Psi}_k(Y_n),\Psi_k(Y)) \\
	     \hfill \ \asleq \ \dist_H^{L^2}(\hat{\Phi}_k(X_n),\hat{\Psi}_k(Y_n)) + \varrho_0 \ \ \text{as} \ \ n \to +\infty,\\
\end{array}
\end{equation}
for an arbitrary small $\varrho_0 > 0$, where the last almost sure inequality follows from the almost sure convergence of $\dist_H^{L^2}(\Phi_k(X),\hat{\Phi}_k(X_n))$ and $\dist_H^{L^2}(\hat{\Psi}_k(Y_n),\Psi_k(Y))$ to $0$ as $n$ goes to infinity (Theorem~\ref{thm:approxDKE}, the volume measure in a Riemannian manifold being $(a,b)$-standard).

\medskip

For an arbitrary small $\varrho_1 > 0$, there exist a metric space $(Z,\dist_Z)$ and isometries $\iota\colon X \to \iota(X) \subseteq Z$ and $\kappa\colon Y \to \kappa(Y) \subseteq Z$ such that:
\begin{equation}\label{eq:sandwichmodifprokdistwithGP}
	\dist_{G\bar{P}}(X,Y) \leq \dist_{\bar{P}}(\iota\#\mu,\kappa\#\eta) \leq \dist_{G\bar{P}}(X,Y) + \varrho_1.
\end{equation}
Fix $\varrho_1 >0$, and pick $(Z,\dist_Z), \iota, \kappa$ satisfying the above. 

\medskip

Write for short $X_n' := \iota(X_n)$, and $Y_n' := \kappa(Y_n)$. Consider the metric measure spaces $(X_n',\dist_Z,\iota\#\mu_n)$ and $(Y_n',\dist_Z,\kappa\#\eta_n)$, their associated distance kernel operators $D^{X_n'}$ and $D^{Y_n'}$, and their DKEs $\hat{\Phi}'_k(X_n')$ and $\hat{\Psi}'_k(Y_n')$. 

By virtue of Proposition~\ref{prop:samesame}, $D^{X_n'}$ and $D^{X_n}$ have the same spectrum $\{\hat{\lambda}_i\}_{i\in \N_{>0}}$ and the same DKE $\hat{\Phi}_k(X_n) = \hat{\Phi}'_k(X_n')$. Similarly, $D^{Y_n'}$ and $D^{Y_n}$ have the same spectrum $\{\hat{\nu}_i\}_{i\in \N_{>0}}$ and the same DKE $\hat{\Psi}_k(Y)=\hat{\Psi}'_k(Y')$. In particular,
\begin{equation}\label{eq:equalitydistPsamples}
	\dist_H^{L^2}(\hat{\Phi}_k(X_n),\hat{\Psi}_k(Y_n)) = \dist_H^{L^2}(\hat{\Phi}'_k(X_n'),\hat{\Psi}'_k(Y_n')).
\end{equation}

\medskip

Define $|\hat{\tau}_1| := \min \{|\hat{\lambda}_1|,|\hat{\nu}_1|\}$, $|\hat{\tau}_k| := \max \{|\hat{\lambda}_k|,|\hat{\nu}_k|\}$, and $\Delta^{X_n,Y_n}_k := \min\{ |\hat{\lambda}^2_i - \hat{\nu}^2_j| : 1\leq i,j \leq k \ , \ i \neq j \}$.

Following the hypothesis that the $k$ largest eigenvalues, in absolute value, of $D^X$ and $D^Y$ are non-zero and have distinct absolute values, we have, by the convergence of empirical eigenvalues of Theorem~\ref{thm:KG}:
\[
	|\hat{\lambda}_1| > |\hat{\lambda}_2| > \ldots > |\hat{\lambda}_k| > 0 \ \ \text{and} \ \ |\hat{\nu}_1| > |\hat{\nu}_2| > \ldots > |\hat{\nu}_k| > 0 
	 \ \ \text{almost surely, as} \ \ n \to +\infty,
\] 
and, additionally,
\begin{equation}\label{eq:proofconveigen}
\toas{|\hat{\tau}_1|}{|\tau_1|}, \ \ \toas{|\hat{\tau}_k|}{|\tau_k|} > 0, \ \ \toas{\Delta_k^{X_n,Y_n}}{\Delta_k^{X,Y} > 0} \ \ \text{as} \ \ n \to +\infty.
\end{equation}
In particular, $|\hat{\tau}_k|$ and $\Delta_k^{X_n,Y_n}$ are almost surely strictly positive when $n \to +\infty$.

\medskip

Now, considering the bound of Theorem~\ref{lem:discr_conv}, define the continuous, monotonically increasing function $F \colon \R_{\geq 0}^4 \to \R$:
\[ 
   F\left(\varepsilon,\tau_1,\frac{1}{\tau_k},\frac{1}{\Delta}\right) := \displaystyle\sqrt{k}\left( 4\sqrt{2} \frac{(\varepsilon+\tau_1)}{\Delta}\sqrt{\tau_1}\right) \cdot \varepsilon + 
	\displaystyle\sqrt{k}\left( 2\sqrt{2} \sqrt{\frac{(\varepsilon + \tau_1)}{\tau_k}}\right) \cdot \sqrt{\varepsilon},
\]
such that:
\[
	\dist_H^{L^2}(\hat{\Phi}'_k(X_n'),\hat{\Psi}'_k(Y_n')) \leq F\left(\dist_{\bar{P}}(\iota\#\mu_n,\kappa\#\eta_n), |\hat{\tau}_1|, \frac{1}{|\hat{\tau}_k|}, \frac{1}{\Delta_k^{X_n,Y_n}}\right).
\]

By the triangle inequality, we have:
\begin{equation}\label{eq:proofineqdistances}
\begin{array}{l}
	\dist_{\bar{P}}(\iota\#\mu_n,\kappa\#\eta_n) \leq \dist_{\bar{P}}(\iota\#\mu_n,\iota\#\mu) + \dist_{\bar{P}}(\iota\#\mu,\kappa\#\eta) + \dist_{\bar{P}}(\kappa\#\eta,\kappa\#\eta_n)  \ \ \ \ \ \ \ \ \ \ \ \ \ \ \ \ \  \\
	\hfill \toas{}{} \dist_{\bar{P}}(\iota\#\mu,\kappa\#\eta) \ \ \text{as} \ \ n \to +\infty,\\
\end{array}
\end{equation}
where the last almost sure convergence follows from Lemma~\ref{lem:convmodifPmetric}.


Consequently, for an arbitrary small $\varrho_2 > 0$, we have, by Equations~(\ref{eq:proofconveigen}) and~(\ref{eq:proofineqdistances}) above, and the fact that $F$ is continuous monotonically increasing:
\begin{equation}\label{eq:boundwithvarrho2}
\begin{array}{l}
	\dist_H^{L^2}(\hat{\Phi}_k(X_n),\hat{\Psi}_k(Y_n)) \ \ \ \ \aseq{Eq.(\ref{eq:equalitydistPsamples})} \ \ \ \ \dist_H^{L^2}(\hat{\Phi}'_k(X_n'),\hat{\Psi}'_k(Y_n')) \ \asleq \ F\left(\dist_{\bar{P}}(\iota\#\mu,\kappa\#\eta) , |\tau_1|, \frac{1}{|\tau_k|}, \frac{1}{\Delta_k^{X,Y}}\right) + \varrho_2\\
		\hfill \ \ \text{as} \ \ n \to +\infty.
\end{array}
\end{equation}

	In conclusion, for arbitrary small $\varrho_0, \varrho_1, \varrho_2 > 0$, there exist isometries $\iota,\kappa$ onto a space $(Z,\dist_Z)$, and large enough samples $X_n$ and $Y_n$ of $X$ and $Y$, such that:
	\[
	\begin{array}{l}
		\dist_H^{L^2}(\Phi_k(X),\Psi_k(Y)) \ \ \leqtext{Eq.(\ref{eq:boundprokdistviaempiralembed})} \ \ \dist_H^{L^2}(\hat{\Phi}_k(X_n),\hat{\Psi}_k(Y_n)) + \varrho_0 \ \ \ \ \ \ \ \ \ \ \ \ \ \ \ \ \ \ \ \ \ \ \  \ \ \ \ \ \ \ \ \ \ \ \ \ \ \ \ \ \ \ \ \ \ \  \ \ \ \ \ \ \ \ 
 \\
		\hfill \leqtext{Eq.(\ref{eq:boundwithvarrho2})} \displaystyle F\left(\dist_{\bar{P}}(\iota\#\mu,\kappa\#\eta) , |\tau_1|, \frac{1}{|\tau_k|}, \frac{1}{\Delta_k^{X,Y}}\right) + \varrho_0+\varrho_2 \ \ \ \ \ \ \ \ \ \ \ \ \ \ \  \\
		\hfill \leqtext{Eq.(\ref{eq:sandwichmodifprokdistwithGP})} \displaystyle F\left(\dist_{G\bar{P}}(X,Y) + \varrho_1 , |\tau_1|, \frac{1}{|\tau_k|}, \frac{1}{\Delta_k^{X,Y}}\right) + \varrho_0+\varrho_2,\\
	\end{array}
	\]
	We conclude, by continuity of $F$, that:
	\[
		\dist_H^{L^2}(\Phi_k(X),\Psi_k(Y)) \leq F\left(\dist_{G\bar{P}}(X,Y) , |\tau_1|, \frac{1}{|\tau_k|}, \frac{1}{\Delta_k^{X,Y}}\right).
	\]
\end{proof}

%% file: stabinv.tex
\subsection{Injectivity}

	In this section, we explore the injectivity of the distance kernel embedding on the space of metric measure spaces. Our first result is that, under mild regularity assumptions, two metric measure spaces have the same distance kernel embedding if and only if there is an isometry between them.
	
\begin{obs}
	\label{obs:hom}
	Corollary \ref{cor:DKThom} states that $\Phi$ is a homeomorphism when $(X,\dist_X,\mu_X)$ is a strictly positive metric measure space. Therefore, if $\Phi(X,\dist_X,\mu_X) = \Phi(Y,\dist_{Y},\mu_Y)$ for a pair of strictly positive metric measure spaces, then there is a homeomorphism from $X$ to $Y$ preserving all the eigenfunctions. Thus, in the following lemma, we fix an underlying topological space and allow the metric and measure to vary.
\end{obs}

\begin{lemma}
	\label{lem:PhiInj}
	Fix a compact topological space $T$. Let $\mu$ and $\mu'$ be strictly positive measures for the Borel $\sigma$-algebra on $T$, with $\mu$ absolutely continuous with respect to $\mu'$, and $\dist$ and $\dist'$ metrics on $X$, both inducing the topology $T$, making $X = (T,\dist,\mu)$ and $X' = (T,\dist',\mu')$ metric measure spaces. Let $D$ and $D'$ be the resulting integral operators. If $\Phi(X) = \Phi(X')$, then $\dist = \dist'$. 
	
\end{lemma}
\begin{proof}
	By assuming that both metric measure spaces induce the same topology, we can work with a single $\sigma$-algebra: their common Borel $\sigma$-algebra. This will prove essential in the following proof, where we take various unions and complements of measurable sets for $\mu$ and $\mu'$, respectively.\\

	The equality $\Phi(X) = \Phi(X')$ implies that $D$ and $D'$ have the same scaled eigenfunctions $\alpha_i$. The distance functions $\dist, \dist'$ thus have the same eigenfunction expansion:
	\[(x_1, x_2) \mapsto \sum_{i=1}^{\infty}\alpha_{i}(x_1)\alpha_{i}(x_2). \]
	
	This converges to $\dist$ in $L_2 (\mu \otimes \mu)$ and to $\dist'$ in $L_2 (\mu' \otimes \mu')$ to $\dist'$. Let us denote by $S_n$ the partial sums of this expansion:
		\[S_n = \sum_{i=1}^{n}\alpha_{i}(x_1)\alpha_{i}(x_2). \]
	
	It is a standard result in measure theory that any $L^{2}$-convergent sequence admits a subsequence that converges pointwise a.e.\footnote{See Theorem 2.15(c) in \cite{folland2009guide} for the implication that convergence in measure implies the existence of pointwise a.e. convergent subsequence. Chebyshev's inequality proves that $L^2$ convergence implies convergence in measure.} Thus, one can extract a subsequence $S_{n_{k}}$ that converges to $\dist$ pointwise on $(T \times T) \setminus N_1$, for some set $N_1 \subset T$ such that $(\mu \otimes \mu)(N_1) = 0$. We can then extract a further subsequence $S_{n_{k_j}}$ that converges pointwise to $\dist'$ on $((T \times T) \setminus N_1) \setminus N_2$, where $(\mu' \otimes \mu')(N_2) = 0$. Since $\mu$ is absolutely continuous to $\mu'$, if we set $N = N_1 \cup N_2$ then $(\mu \otimes \mu)(N) = 0$. Since $\mu$ is strictly positive, the set $N$ cannot contain any open sets, hence $N^{c}$ is dense in $T \times T$. We see then that $\dist = \dist'$ on a dense subset of $T \times T$; since these functions are both continuous in the same topology $T$, they are equal everywhere.

\end{proof}


For finite metric measure spaces, Lemma \ref{lem:PhiInj} only requires a finite-dimensional embedding to hold. 

\begin{cor}
	Let $X = (X,\dist_{X},\mu_X)$ and $Y = (Y, \dist_{Y}, \mu_Y)$ be a pair of finite metric measure spaces, with $|X|, |Y| \leq k$. If $\Phi_{k}(X) = \Phi_{k}(Y)$, then $X$ and $Y$ are isometric.
\end{cor}
\begin{proof}
	From Observation \ref{obs:hom}, we see that $|X| = |Y|$, so $X$ and $Y$ are homeomorphic when equipped with the discrete topology, which is the same topology induced by any choice of metric on $X$ or $Y$. The result follows then from Lemma \ref{lem:PhiInj}, by noting that all the ``higher" eigenfunctions vanish and therefore are equal.
\end{proof}

\subsection{Inverse Stability}

These injectivity results tell us that distinct metric spaces have distinct embeddings. However, we would like to assert something stronger: spaces with similar embeddings are also geometrically similar. To that end, we need to introduce some new definitions and technical lemmas.

\begin{definition}
	For vectors $v,w \in \mathbb{C}^k$, define the following bilinear form:
	\[[v,w] = \sum_{i=1}^{k} v_{i}w_{i} \in \mathbb{C}. \]
	This form is symmetric but not a dot product. 
\end{definition}

The utility of the bilinear form $[\cdot, \cdot]$ comes from the following equality:
\begin{equation}
\label{eqn:expansion}
[\Phi_{k}(x),\Phi_{k}(x')] =   \sum_{i=1}^{k}\alpha_{i}(x)\alpha_{i}(x') =  \sum_{i=1}^{k}\left(\sqrt{\lambda_{i}}\phi_{i}(x)\right)\left(\sqrt{\lambda_{i}}\phi_{i}(x')\right)= \sum_{i=1}^{k} \lambda_{i}\phi_{i}(x)\phi_{i}(x'). 
\end{equation}

That is, when applied to the distance kernel embedding it gives the first $k$ terms of the eigenfunction expansion of the distance function $\dist_{X}$.

\begin{lemma}
	The bilinear form $[,]$ satisfies the following Cauchy-Schwarz inequality:
	\[|[v,w]| \leq \| v \|_{2}\| w \|_{2}. \]	
\end{lemma}
\begin{proof}
	Let $v = (v_1, \cdots, v_k) \in \mathbb{C}^k$ and $w = (w_1, \cdots, w_k) \in \mathbb{C}^k$. By definition,
	\[[v,w] = \sum_{i=1}^{k}v_{i}w_{i}. \]
	Using the triangle inequality for complex numbers,
	\[|[v,w]| = \left|\sum_{i=1}^{k}v_{i}w_{i}\right| \leq \sum_{i=1}^{k}|v_{i}||w_{i}| = \langle \tilde{v}, \tilde{w} \rangle, \]
	where $\tilde{v},\tilde{w} \in \mathbb{R}^k$ are obtained from $v$ and $w$ by taking component-wise moduli. Note that
	\[\| v \|_{2}^2 =  \sum_{i=1}^{k}|v_{i}|^2 = \|\tilde{v}\|_{2}^2. \] 
	\[\| w \|_{2}^2 = \sum_{i=1}^{k}|w_{i}|^2 = \|\tilde{w}\|_{2}^2. \] 
	
	Thus $v$ and $\tilde{v}$ have the same magnitude, as do $w$ and $\tilde{w}$. To complete the proof, we apply the ordinary Cauchy-Schwarz inequality to $\tilde{v}$ and $\tilde{w}$,
	\[\langle \tilde{v}, \tilde{w} \rangle \leq \|\tilde{v}
	\|_2 \|\tilde{w}\|_2 = \| v \|_{2}\| w \|_{2}.    \]
	
\end{proof}

The following lemma asserts that pairs of nearby vectors have similar bilinear products.

\begin{lemma}
	\label{lem:haustodist}
	Let $v_{1},v_{2},w_{1},w_{2} \in \mathbb{C}^k$ be vectors such that $\|v_1 - w_1\|_{2} \leq \varepsilon$ and $\|v_2 - w_2\|_{2} \leq \varepsilon$. Then \[|[v_1,v_2] - [w_1,w_2]| \leq \varepsilon \min \left\{\|v_1\|_{2} + \|v_2\|_{2}, \|w_1\|_{2} + \|w_2\|_{2}\right\} + \varepsilon^2.\]
\end{lemma}  
\begin{proof}
By bilinearity,
	\[[w_1,w_2] = [v_1,v_2] + [v_1,(w_2 - v_2)] + [(w_1 -v_1),v_2] + [(w_1 - v_1),(w_2 - v_2)].\]
	Thus,
	\[|[v_1,v_2] - [w_1,w_2]| \leq | [v_1,(w_2 - v_2)]| + |[(w_1 -v_1),v_2]| + |[(w_1 - v_1),(w_2 - v_2)]|.\]
	
	By a symmetric argument, switching $v_1$ and $v_2$ with $w_1$ and $w_2$, one obtains:
	\[|[v_1,v_2] - [w_1,w_2]| \leq | [w_1,(v_2 - w_2)]| + |[(v_1 -w_1),w_2]| + |[(v_1 - w_1),(v_2 - w_2)]|.\]
	The result then follows by applying the Cauchy-Schwarz to each term on the right-hand sides of both inequalities, and taking the minimum of the two sums.
\end{proof}

 Before stating our first inverse stability result, we need the following definition:
\begin{definition}
	For a compact metric measure space $(X, \dist_{X}, \mu_{X})$ and a positive integer $k$, define the error function:
	\[E_{X,k}(x,x') = |[\Phi_{k}(x),\Phi_{k}(x')] - \dist_{X}(x,x')|\]
	By equation (\ref{eqn:expansion}), this measures the extent to which the eigenfunction expansion of $\dist_{X}$, truncated at $k$, approximates the metric.
\end{definition}
\begin{theorem}
	\label{thm:invstabmeas}
	Let $(X,\dist_{X}, \mu_X)$ and $(Y,\dist_{Y}, \mu_Y)$ be compact metric measure spaces, with eigenvalues $\{\lambda_{i}\}$ and $\{\nu_{i}\}$. Take $k \in \mathbb{N}$ to be a positive integer, and let $\varepsilon = \dist_{H}^{L^2}(\Phi_{k}(X),\Phi_{k}(Y))$. Then
	\[d_{GH}(X,Y) \leq 2\varepsilon \min \left\{ \max_{x \in X} \| \Phi_{k}(x) \|_{2}, \max_{y \in Y} \| \Phi_{k}(y) \|_{2}\right\} + \| E_{X,k} \|_{\infty} + \| E_{Y,k} \|_{\infty} + \varepsilon^2. \]
\end{theorem}

\begin{proof}
	Let $C$ be an optimal Hausdorff correspondence between $\Phi_{k}(X)$ and $\Phi_{k}(Y)$. Let $(x,x') \in X \times X$ and $(y,y') \in Y \times Y$ with $(\Phi_{k}(x),\Phi_{k}(y)),(\Phi_{k}(x'),\Phi_{k}(y')) \in C$. Lemma \ref{lem:haustodist}, together with the bounds $\|\Phi_{k}(x)-\Phi_{k}(y)\|_{L^2} \leq \varepsilon$ and $\|\Phi_{k}(x')-\Phi_{k}(y')\|_{L^2} \leq \varepsilon$ , gives
	\[|[\Phi_{k}(x),\Phi_{k}(x')] - [\Phi_{k}(y),\Phi_{k}(y')] | \leq  2\varepsilon \min \left\{ \max_{x \in X} \| \Phi_{k}(x) \|_{2}, \max_{y \in Y} \| \Phi_{k}(y) \|_{2} \right\} + \varepsilon^2.\]
	
	Using the triangle inequality, we can replace $[\Phi_{k}(x),\Phi_{k}(x')]$ with $\dist_{X}(x,x')$ and $[\Phi_{k}(y),\Phi_{k}(y')]$ with $\dist_{Y}(y,y')$, at the cost of adding an additive error of at most $\| E_{X,k} \|_{\infty}$ and $\| E_{Y,k} \|_{\infty}$ respectively, giving the inequality:
	\[ |\dist_{X}(x,x') - \dist_{Y}(y,y')| \leq 2\varepsilon \min \left\{ \max_{x \in X} \| \Phi_{k}(x) \|_{2}, \max_{y \in Y} \| \Phi_{k}(y) \|_{2}\right\} + \| E_{X,k} \|_{\infty} + \| E_{Y,k} \|_{\infty} + \varepsilon^2. \]
	
	The result follows.\end{proof}

\begin{remark}
%
In Section \ref{sec:embeddingconstants}, we provide  analytic estimates on the constants appearing in Theorem \ref{thm:invstabmeas}, in terms of the geometry of the spaces $X$ and $Y$. These estimates are pessimistic and fall short of the positive results obtained in the simulations of Section \ref{sec:experiments}. 

\end{remark}

\begin{remark}
	The theorem is stated in terms of the Gromov-Hausdorff distance rather than its measure-theoretic variants, such as the Gromov-Hausdorff-Prokhorov distance. This is because scaling the measure does not affect the embedding. Thus, we can say that our embedding uses the measure to prioritize certain distance functions on our space, but it does not record the measure itself.
\end{remark}

Theorem~\ref{thm:invstabmeas} applies to compact metric spaces in general. In the following section, we reconsider these results for the finite case, demonstrating how the embedding is computed and concluding with a simplified, more precise form of Theorem~\ref{thm:invstabmeas}.

\subsection*{Finite Metric Measure Spaces}

In the finite case, we have a metric measure space $(X,\dist_{X},\mu_X)$ with $|X| = n$. We assume that $\mu_{X}$ has full support, so that $\mu_{X}(x) > 0$ for all $x \in X$. Define the matrix\footnote{We intentionally use the same symbol $D$ to refer to both the operator and its associated matrix.} $D_{ij} = \dist_{X}(x_i,x_j)\mu_{X}(x_j)$, so that if $f: X \to \mathbb{R}$ is a function, and $v \in \mathbb{R}^n$ is the vector $v_{i} = f(x_i)$, then $(Df)(x_j) = \sum_{i=1}^{n}\mu_{X}(x_i)\dist_{X}(x_i,x_j)f(x_i) =  (Dv)_{j}$. Define the diagonal matrix $Q_{ii} = \mu_{X}(x_i)$, giving rise to the inner product $Q(v,w) = v^T Q w$.

\begin{lemma}
	The operator $D$ is self-adjoint with respect to $Q$.
\end{lemma}
\begin{proof}
	Let $v,w \in \mathbb{R}^n$. Compute:
	\begin{align*}
	\langle Dv, w \rangle_{Q} &= \sum_{i}(Dv)_{i}w_{i}\mu_{X}(x_i)\\
	& = \sum_{i}(\sum_{j}v_j D_{ij} )w_{i}\mu_{X}(x_i)\\
	& = \sum_{i,j}\dist(x_i,x_j)v_j w_i \mu_{X}(x_j)\mu_{X}(x_i)\\
	& = \sum_{j}(\sum_{i} \dist(x_j,x_i)\mu_{X}(x_i)w_i) v_j \mu_{X}(x_j)\\
	& = \sum_{j}(\sum_{i} D_{ji} w_i) v_j \mu_{X}(x_j)\\
	& = \sum_{j} (Dw)_{j}v_{j}\mu_{X}(x_j)\\
	& = \langle v, Dw \rangle_{Q}.
	\end{align*}
\end{proof}

Thus, by the spectral theorem, $D$ has real eigenvalues $\{\lambda_1, \cdots, \lambda_n \}$, ordered by decreasing absolute value, and a basis of $Q$-orthonormal real eigenvectors $\{e_1, \cdots e_n\}$. As earlier, we  maintain the convention that each eigenvalue has multiplicity one.\\

Let $A$ be the $n \times n$ matrix whose $i$th column is $e_i$. $Q$-orthonormality of the eigenbasis means that:
\[A^{T}QA = I.\]

Since $\mu_{X}(x_i)>0$ for all $i$, the diagonal matrix $Q$ is invertible. Moreover, since $A$ is orthonormal for the inner product induced by $Q$, it too is invertible. We can thus deduce:
\[ A^{T}QA = I  \Longrightarrow AA^{T} = Q^{-1}.\]


Denoting the $i$th row of the matrix $A$ by $r_i$, this tells us that:
\begin{equation}
\label{eqn:rownorm}
\| r_i \|_2 = 1/\sqrt{\mu_{X}(x_i)}.
\end{equation}

As before, we define the functions $\alpha_{i}(x_j) = \sqrt{\lambda_{i}}(e_i)_{j}$ with the convention that the square root of a negative eigenvalue is the imaginary number with positive imaginary part, giving rise to the embedding $\Phi = (\alpha_1, \cdots, \alpha_n)$.  If $V$ is the diagonal matrix whose $(ii)$th entry is $\sqrt{\lambda_i}$, then $\Phi$ maps $x_i$ to the $i$th row of $AV$. We now show how to recover the geometry of $X$ from its embedding. For $x_{i} \in X$, let $\dist_{i} : X \to \mathbb{R}$ be the ``distance to $x_i$" function, thought of as a vector in $\mathbb{R}^n \subset \mathbb{C}^n$. Observe that
\begin{align*}
\dist_{i} &= \sum_{l=1}^{n}\langle e_l, \dist_i \rangle_{Q} e_l\\
& = \sum_{l=1}^{n} \left(\sum_{j=1}^{n}(e_{l})_{j}Q_{jj}(d_{i})_{j} \right) e_l\\
 & = \sum_{l=1}^{n}\left(\sum_{j=1}^{n}(e_{l})_{j}\mu_{X}(x_j)(d_{i})_{j} \right) e_l\\
 & = \sum_{l=1}^{n} (De_l)_{i} e_{l}\\
 & = \sum_{l=1}^{n} \lambda_{l}(e_{l})_{i}e_{l}.
\end{align*}

Hence, using the same bilinear form $[\cdot, \cdot]$ as earlier, 
\[\dist(x_i,x_j) = (\dist_{i})_{j} = \sum_{l=1}^{n}\lambda_{l}(e_l)_{i}(e_{l})_{j} = \sum_{l=1}^{n} (\sqrt{\lambda_l} e_{l})_{i} (\sqrt{\lambda_l} e_{l})_{j} = [ \Phi(x_i), \Phi(x_j) ].\]

\begin{example}
	Let $X$ consist of two points,$x_1$ and $x_2$, with $\dist(x_1,x_2)  = 1$. Let $\mu (x_1) = 1$ and $\mu (x_2) = 4$. Our matrix $D$ is then
	\[\begin{pmatrix}
	0 & 4\\
	1 & 0
	\end{pmatrix}.\]  
	The eigenvalues of this matrix are $\pm 2$. Let 
	\[e_1 = \begin{pmatrix}
	\frac{1}{\sqrt{2}}\\
	\frac{1}{2\sqrt{2}}
	\end{pmatrix} \,\,\,\,\,\,\,\, e_2 = \begin{pmatrix}
	\frac{1}{\sqrt{2}}\\
	\frac{-1}{2\sqrt{2}}
	\end{pmatrix}.  \]
	
	These are eigenvectors with eigenvalue $+2$ and $-2$ respectively. Define the inner product matrix
	\[Q  = \begin{pmatrix}
	1 & 0\\
	0 & 4
	\end{pmatrix}.\]
	Observe that
	\[e_1^T Q e_1 = e_2^T Q e_2 =  1 \cdot \frac{1}{2} + 4 \cdot \frac{1}{8}  = 1 \]
	\[e_1^T Q e_2 = e_2^T Q e_1 = 1 \cdot \frac{1}{2} - 4 \cdot \frac{1}{8} = 0, \]
	
	so that $\{e_1,e_2\}$ is $Q$-orthonormal. We then have that
	\[\Phi(x_1) = \langle \frac{\sqrt{2}}{\sqrt{2}}, \frac{\sqrt{-2}}{\sqrt{2}}  \rangle = \langle 1, \sqrt{-1} \rangle    \]
	\[\Phi(x_2) = \langle \frac{\sqrt{2}}{2\sqrt{2}}, \frac{-\sqrt{-2}}{2\sqrt{2}}  \rangle = \langle \frac{1}{2}, \frac{-\sqrt{-1}}{2} \rangle,    \]
	
	and
	\[\dist_{X}(x_1,x_2) = [\Phi(x_1),\Phi(x_2)] = \frac{1}{2} + \frac{1}{2} = 1 \]
	\[\dist_{X}(x_1,x_1) = [\Phi(x_1),\Phi(x_1)] =  1 - 1 = 0 \]
	\[\dist_{X}(x_2,x_2) = [\Phi(x_2),\Phi(x_2)] = \frac{1}{4} - \frac{1}{4} = 0. \]
	
\end{example}

For finite metric spaces, we have the following explicit bound on the norms of our embedding vectors:

\begin{prop}
	\label{prop:embedboundtopeig}
	We have $\| \Phi_{k}(x)\|_{2} \leq \sqrt{|\lambda_1|}/\mu_{X}(x)$ for all $x \in X$.
\end{prop}
\begin{proof}
	Denote by $s_i$ the complex vector whose $j$th component $(s_i)_j$ is equal to $\sqrt{\lambda_{j}}(r_i)_j$. That is, $s_i$ is obtained by rescaling the entries of the row vector $r_i$ by the square roots of the eigenvalues corresponding to the columns of those entries. For $m \in \{1, \cdots ,n\}$, let $s_{i,m}$ be the vector obtained by keeping only the first $m$ coordinates of $s_{i}$, and define $r_{i,m}$ similarly. By definition, $s_{i,k} = \Phi_{k}(x_i)$.\\
	
	By our ordering of eigenvalues, $|\lambda_{1}| \geq |\lambda_{j}|$ for all $j$. Thus, each component of $s_{i,k}$ has modulus no larger than the corresponding component of $\sqrt{\lambda_{1}}r_{i,k}$. From this we deduce that $\|\Phi_{k}(x_i)\|_{2} = \|s_{i,k}\|_{2} \leq \|\sqrt{\lambda_{1}} r_{i,k}\|_{2} =  |\sqrt{\lambda_{1}}|\|r_{i,k}\|_{2} = \sqrt{|\lambda_{1}|}\|r_{i,k}\|_{2}$. Since $r_{i,k}$ is obtained by truncating $r_i$, we have $\|r_{i,k}\|_{2} \leq \|r_{i}\|_{2}$. We thus deduce  that $\|\Phi_{k}(x_i)\|_{2}\leq \sqrt{|\lambda_{1}|}\|r_{i}\|_{2}$.  Combining this inequality with Equation~(\ref{eqn:rownorm}) completes the proof.
	
\end{proof}

We also have the following bound on the error function:

\begin{prop}
	\label{prop:truncbound}
Applying the bilinear form $[\cdot,\cdot]$ to the truncated embedding $\Phi_{k}$ gives an approximation of the distance function $\dist$ whose error is bounded by the largest (in absolute value) omitted eigenvalue and the distribution of $\mu_{X}$:
\[E_{X,k}(x_i,x_j) = |\dist(x_i,x_j) - [ \Phi_{k}(x_i), \Phi_{k}(x_j) ]| \leq  \frac{|\lambda_{k+1}|}{\sqrt{\mu_{X}(x_i)\mu_{X}(x_j)}}.\]
\end{prop}
\begin{proof}
Define the vectors $s_i$ as in the proof of Proposition \ref{prop:embedboundtopeig}. For $m \in \{1,2, \cdots ,n, n+1\}$, let $s_{i}^{m}$ be the vector obtained by throwing away the first $(m-1)$ coordinates of $s_{i}$, so that $s_{i}^{1} = s_{i}$ and $s_{i}^{(n+1)}$ is empty. Define $r_{i}^{m}$ similarly. Observe that:
\[|\dist(x_i,x_j) - [ \Phi_{k}(x_i), \Phi_{k}(x_j) ]| = \left|\sum_{l=k+1}^{n} \lambda_{l}(e_l)_{i}(e_l)_{j} \right| = |\langle s_{i}^{(k+1)}, s_{j}^{(k+1)} \rangle|,\]
where the inner product on the right-hand side is the ordinary, Euclidean one. The Cauchy-Schwarz inequality tells us that:
\[|\langle s_{i}^{(k+1)}, s_{j}^{(k+1)} \rangle| \leq \|s_{i}^{(k+1)}\|_{2}\|s_{j}^{(k+1)}\|_{2}.\]

We now bound the quantities on the right-hand side of the above inequality. By our ordering of eigenvalues, $|\lambda_{k+1}| \geq |\lambda_{l}|$ for all $l \geq k+1$. Thus every component of the vector $\sqrt{\lambda_{k+1}}r_{i}^{(k+1)}$ has a larger modulus than the corresponding component of $s_{i}^{(k+1)}$. We deduce that $\|s_{i}^{(k+1)}\|_{2} \leq \|\sqrt{\lambda_{k+1}}r_{i}^{(k+1)}\|_{2}= |\sqrt{\lambda_{k+1}}|\|r_{i}^{(k+1)}\|_{2} = \sqrt{|\lambda_{k+1}|} \|r_{i}^{(k+1)}\|_{2}$. We can similarly deduce that $\|s_{j}^{(k+1)}\|_{2} \leq \sqrt{|\lambda_{k+1}|} \|r_{j}^{(k+1)}\|_{2}$. Since $r_{i}^{(k+1)}$ and $r_{j}^{(k+1)}$ are truncations of $r_i$ and $r_j$, respectively, we deduce that:
\[\|s_{i}^{(k+1)}\|_{2}\|s_{j}^{(k+1)}\|_{2} \leq (\sqrt{|\lambda_{k+1}|})^{2}\|r_{i}^{(k+1)}\|_{2}\|r_{j}^{(k+1)}\|_{2} \leq |\lambda_{k+1}| \|r_{i}\|_{2}\|r_{j}\|_{2}. \]

Recalling from Equation~(\ref{eqn:rownorm}) that $\|r_{i}\|_{2} = 1/\sqrt{\mu_{X}(x_i)}$ and $\|r_{j}\|_{2} = 1/\sqrt{\mu_{X}(x_j)}$, the proof follows.
\end{proof}

We now provide a finite analogue of Theorem \ref{thm:invstabmeas}:
%
\begin{theorem}
	\label{thm:invstab}
	Let $(X,\dist_{X}, \mu_X)$ and $(Y,\dist_{Y}, \mu_Y)$ be finite metric measure spaces, with eigenvalues $\{\lambda_{i}\}$ and $\{\nu_{i}\}$, and let $\theta = \min \{\min_{x \in X}\mu_{X}(x),\min_{y \in Y}\mu_{X}(y)\}$. Take $k \leq |X|,|Y|$, and suppose that $\dist_{H}^{L^{2}}(\Phi_{k}(X),\Phi_{k}(Y)) \leq \varepsilon$. Then,
	\[\dist_{GH}(X,Y) \leq 2\varepsilon \frac{\min (\sqrt{|\lambda_{1}|},\sqrt{|\nu_{1}|})}{\theta} + \varepsilon^2 + \frac{|\lambda_{k+1}| + |\nu_{k+1}|}{\theta}.       \]	 
\end{theorem}
\begin{proof}
	The result holds trivially for $\theta = 0$, so let us assume $\theta > 0$.\\
	
	Suppose that the Hausdorff distance between $\Phi_{k}(X)$ and $\Phi_{k}(Y)$ is realized by a correspondence $C \subset \Phi_{k}(X) \times \Phi_{k}(Y)$. This induces a correspondence on $X \times Y$.  Suppose that $x$ is paired with $y$ and $x'$ is paired with $y'$. If we write $v_{1} = \Phi_{k}(x), v_{2} = \Phi_{k}(x')$, $w_{1} = \Phi_{k}(y)$, and $w_{2} = \Phi_{k}(y')$, this means that $\|v_1 - w_1\|_{2} \leq \varepsilon$ and $\|v_2 - w_2\|_{2} \leq \varepsilon$.  Then, Proposition \ref{prop:truncbound} tells us that $|\dist_{X}(x,x') - [v_1,v_2]| \leq |\lambda_{k+1}|/\theta$ and $|\dist_{Y}(y,y') - [w_1,w_2]| \leq |\nu_{k+1}|/\theta$. By Proposition \ref{prop:embedboundtopeig}, the $L^2$ norms of $v_{1}$ and $v_{2}$ are at most $\sqrt{|\lambda_{1}|}/\theta$, and similarly the $L^2$ norms of $w_{1}$ and $w_{2}$ are at most $\sqrt{|\nu_{1}|}/\theta$.\\
	 
	 Combining these various bounds with Lemma \ref{lem:haustodist}, and applying the triangle inequality, we deduce:
	 \begin{align*}
	 |\dist_{X}(x,x') - \dist_{Y}(y,y')|
	 & \leq |\dist_{X}(x,x') - [v_1,v_2]| + |[v_1,v_2] - [w_1,w_2]| + |[w_1,w_2] - d_{Y}(y,y')|\\
	 & \leq \frac{|\lambda_{k+1}|}{\theta} + \varepsilon\min \left\{\|v_1\|_{2} + \|v_2\|_{2}, \|w_1\|_{2} + \|w_2\|_{2}\right\} + \varepsilon^2 + \frac{|\nu_{k+1}|}{\theta}\\
	 & \leq \varepsilon \min \left\{2\frac{\sqrt{|\lambda_{1}|}}{\theta},2 \frac{\sqrt{|\nu_{1}|}}{\theta}\right\} + \varepsilon^2 + \frac{|\lambda_{k+1}| + |\nu_{k+1}|}{\theta}\\
	 & =  2\varepsilon \frac{\min (\sqrt{|\lambda_{1}|},\sqrt{|\nu_{1}|})}{\theta} + \varepsilon^2 + \frac{|\lambda_{k+1}| + |\nu_{k+1}|}{\theta}.   
	 \end{align*}
	 
\end{proof}

	When $k \geq |X|$, $\mu_{k+1} = 0$, and similarly for $Y$ and $\nu_{k+1}$. Thus, the final additive error term in Theorem \ref{thm:invstab} disappears in this setting. Note that in the case of uniform unit atomic measures, i.e. $\mu_X(x_i) = 1 = \mu_{Y}(y_j) \,\, \forall i,j$, we have $\theta = 1$.
	

%

%% file: embeddingconstants.tex
The bounds in the main theorem of Section \ref{sec:stabinv}, Theorem \ref{thm:invstabmeas}, depend on the magnitude of the quantities $\|E_{X,k}\|_{\infty}$ and  $\max_{x \in X} \|\Phi_{k}(x)\|_{2}$. The goal of this section is to provide worst-case estimates of these quantities, using standard measure-theoretic techniques. We begin by bounding the $L^{\infty}$ norm of our eigenfunctions.

\begin{lemma}
	\label{lem:phibound}
	Suppose that $(X,\dist_X,\mu_X)$ is a compact metric measure space that is $(a,b)$-standard with threshold $r$. Let $\phi_i$ be an eigenfunction of the distance kernel operator $D^{X}$ with eigenvalue $\lambda_i$. Then we have the following $L^{\infty}$ bound on $\phi_i$:
	\[\|\phi_i\|_{\infty} \leq \frac{1}{\sqrt{a}r^{b/2}} + \frac{r}{|\lambda_i|}\sqrt{\vol(X)}.   \]
	
\end{lemma}
\begin{proof}
	The result holds trivially when $\lambda_{i} = 0$, so let us assume $\lambda_{i} > 0$.	By Lemma \ref{lem:eiglip}, $\phi_i$ is $\frac{\sqrt{\vol(X)}}{|\lambda_i|}$-Lipschitz. Let us assume that
	\[\|\phi_{i}\|_{\infty} \geq \frac{r}{|\lambda_i|}\sqrt{\vol(X)}, \]
	as otherwise, the result would be immediate. Let $x \in X$ be a point at which $|\phi_i|$ attains the value $\|\phi_{i}\|_{\infty}$. Consider the ball $B = B(x,r)$ centered at $x$. By Lipschitzness of $\phi_i$, we know that
	\[|\phi_{i}(y)|\geq \left(\|\phi_{i}\|_{\infty}  - \frac{r}{|\lambda_i|}\sqrt{\vol(X)}\right), \,\,\,\,\, \forall y \in B. \]
		This tells us that the $L^2$ norm of $\phi_i$ restricted to the ball $B$ is bounded below by:
	\[\|\phi_{i}1_{B}\|_{2} \geq \left(\|\phi_{i}\|_{\infty}  - \frac{r}{|\lambda_i|}\sqrt{\vol(X)}\right)\sqrt{\mu_{X}(B)}.  \]
	By $(a,b)$-standardness, 
	\[\mu_{X}(B) \geq ar^{b},\]
	so that 
	\[\sqrt{\mu_{X}(B)} \geq \sqrt{a}r^{b/2}.  \]
	Since our eigenfunctions are normalized, we have
	\[1 = \|\phi_{i}\|_{2} \geq \|\phi_{i}1_{B}\|_{2} \geq \left(\|\phi_{i}\|_{\infty}  - \frac{r}{|\lambda_i|}\sqrt{\vol(X)}\right)\sqrt{a}r^{b/2},  \]
	hence the result.
\end{proof}

This gives us a bound on the $L^2$ norm of our embedding vectors.

\begin{lemma}
	\label{lem:embedbound}
	Suppose that $(X,\dist_X,\mu_X)$ is a compact metric measure space that is $(a,b)$-standard with threshold $r$. Then, for any $x \in X$, 
	\[\|\Phi_{k}(x)\|_{2} \leq  \sqrt{\sum_{i=1}^{k}  \left(\frac{\sqrt{|\lambda_i|}}{\sqrt{a}r^{b/2}} + \frac{r}{\sqrt{|\lambda_i|}}\sqrt{\vol(X)} \right)^2 }. \]	
\end{lemma}
\begin{proof}
	This follows from Lemma \ref{lem:phibound} via the following computation:
	
	\begin{align*}
	\|\Phi_{k}(x)\|_{2}^{2} &= \sum_{i=1}^{k} (\sqrt{\lambda_{i}}\phi_{i}(x))^{2}\\
	&\leq \sum_{i=1}^{k}|\lambda_{i}|\phi_{i}^{2}(x)\\
	&\leq \sum_{i=1}^{k}|\lambda_{i}|\left(\frac{1}{\sqrt{a}r^{b/2}} + \frac{r}{|\lambda_i|}\sqrt{\vol(X)}\right)^{2}\\
	&= \sum_{i=1}^{k}\left(\frac{\sqrt{|\lambda_{i}|}}{\sqrt{a}r^{b/2}} + \frac{r}{\sqrt{|\lambda_i|}}\sqrt{\vol(X)}\right)^{2}.		
	\end{align*}

\end{proof}

We now turn to studying the error of the eigenfunction approximation to the distance function $\dist_{X}(x,x)'$ on the product space $X \times X$.

\begin{lemma}
	\label{lem:distlip}
	The distance function $\dist_{X}: X \times X \to \mathbb{R}$ is $\sqrt{2}$-Lipschitz.
\end{lemma}
\begin{proof}
	Let $(x_1,x_2), (y_1,y_2) \in X \times X$. By the triangle inequality, we have
	\begin{align*}
	\dist_{X}(x_1,x_2) - \dist_{X}(y_1,y_2) &\leq  	\big(\dist_{X}(x_1,y_1) + \dist_{X}(y_1,y_2) + \dist_{X}(y_2,x_2)\big) - \dist_{X}(y_1,y_2)\\
	& = \dist_{X}(x_1,y_1) + \dist_{X}(y_2,x_2)\\
	& \leq \sqrt{2}\sqrt{\dist_{X}(x_1,y_1)^{2} + \dist_{X}(y_2,x_2)^{2}}\\
	& = \sqrt{2}\dist_{X\times X}((x_1,x_2),(y_1,y_2)).
	\end{align*}
\end{proof}

\begin{lemma}
	\label{lem:phi2lip}
		Suppose that $(X,\dist_X,\mu_X)$ is a compact metric measure space that is $(a,b)$-standard with threshold $r$. Let $\phi_i$ be an eigenfunction of the distance kernel operator $D^{X}$ with eigenvalue $\lambda_i$. Then the product function $(x,y) \mapsto \lambda_i \phi_{i}(x)\phi_{i}(y)$ is Lipschitz over $X \times X$, with Lipschitz constant
		\[K_i = \sqrt{2} \left(\frac{\sqrt{\vol(X)}}{ \sqrt{a}r^{b/2}} + \frac{r \vol(X)}{|\lambda_i|}\right).  \]
\end{lemma}
\begin{proof}
	As before, let $(x_1,x_2),(y_1,y_2) \in X \times X$. We have:
	\begin{align*}
	|\phi_{i}(x_1)\phi_{i}(x_2) - \phi_{i}(y_1)\phi_{i}(y_2)| &= |\phi_{i}(x_1)\phi_{i}(x_2) - \phi_{i}(x_1)\phi_{i}(y_2) + \phi_{i}(x_1)\phi_{i}(y_2) - \phi_{i}(y_1)\phi_{i}(y_2)|\\
	& \leq |\phi_{i}(x_1)\phi_{i}(x_2) - \phi_{i}(x_1)\phi_{i}(y_2)| + |\phi_{i}(x_1)\phi_{i}(y_2) - \phi_{i}(y_1)\phi_{i}(y_2)|\\
	& = |\phi_{i}(x_1)||\phi_{i}(x_2)-\phi_{i}(y_2)| + |\phi_{i}(y_2)||\phi_{i}(x_1) - \phi_{i}(y_1)|.
	\end{align*}

	Let $k_i$ be the Lipschitz constant of $\phi_i$. The last line above is bounded by
	\begin{align*}
	k_{i}\|\phi_{i}\|_{\infty}(\dist_{X}(x_1,y_1) + \dist_{X}(x_2,y_2)) & \leq \sqrt{2} k_{i}\|\phi_{i}\|_{\infty}\sqrt{\dist_{X}(x_1,y_1)^{2} + \dist_{X}(x_2,y_2)^{2}}\\
	(\text{Lemma \ref{lem:eiglip}})& \leq \sqrt{2} \frac{\sqrt{\vol (X)}}{|\lambda_{i}|}\|\phi_{i}\|_{\infty}\dist_{X\times X}((x_1,x_2),(y_1,y_2))\\
		(\text{Lemma \ref{lem:phibound}})& \leq \sqrt{2} \frac{\sqrt{\vol (X)}}{|\lambda_{i}|}\left(\frac{1}{\sqrt{a}r^{b/2}} + \frac{r}{|\lambda_i|}\sqrt{\vol(X)}\right)\dist_{X\times X}((x_1,x_2),(y_1,y_2)).
	\end{align*} 
	The result follows after scaling by $\lambda_{i}$.
\end{proof} 

Using these results, we can control the error $E_{X,k}(x,x') = |[\Phi_{k}(x),\Phi_{k}(x')] - \dist_{X}(x,x')|$. 

\begin{lemma}
	\label{lem:errorbound}
		Suppose that $(X,\dist_X,\mu_X)$ is a compact metric measure space that is $(a,b)$-standard with threshold $r$. Let $K_i$ be as in Lemma \ref{lem:phi2lip}, and let:
		\[K = \sqrt{2} + \sum_{i=1}^{k} K_i.\]
		Define also the sum of squares of eigenvalues:
		\[\Lambda_{k+1} = \sum_{i=k+1}^{\infty} \lambda_{i}^2. \]
		Then, we have the following $L^{\infty}$ bound on $E_{X,k}$:
		\[\|E_{X,k}\|_{\infty} \leq \frac{\sqrt{\Lambda_{k+1}}}{ar^b} + rK.  \]
\end{lemma}
\begin{proof}
	We can rewrite the error $E_{X,k}$ as follows:
	\[|[\Phi_{k}(x),\Phi_{k}(x')] - \dist_{X}(x,x')|  =|\sum_{i=1}^{k}\lambda_{i}\phi_{i}(x)\phi_{i}(x') - \dist_{X}(x,x') |.\]
	
	Each term of the form $\lambda_{i}\phi_{i}(x)\phi_{i}(x')$ is, by Lemma \ref{lem:phi2lip}, $K_{i}$-Lipschitz. By Lemma \ref{lem:distlip}, we know that the function $\dist_{X}(x,x')$ is $\sqrt{2}$-Lipschitz. As the sum of an $A$-Lipschitz function and a $B$-Lipschitz function is $(A+B)$-Lipschitz, and since taking absolute values does not increase Lipschitz constants, we can conclude that $E_{X,k}$ is $K$-Lipschitz. Moreover, we know that $X \times X$ is equipped with an $(a^2,2b)$-standard measure with threshold $r$, as is immediate from the definition of the product measure. Let us assume that
	\[\|E_{X,k}\|_{\infty} \geq rK, \]
	as otherwise, the result would be immediate. Let $(x,x') \in X \times X$ be a point at which $E_{X,k}$ attains its maximum. Consider the ball $B = B((x,x'),r)$ centered at $(x,x')$. By Lipschitzness of $E_{X,k}$, we know that
	\[E_{X,k}(y,y')\geq \|E_{X,k}\|_{\infty} - rK, \,\,\,\,\, \forall (y,y') \in B. \]
	This tells us that the $L^2$ norm of $E_{X,k}$ restricted to the ball $B$ is bounded below by:
	\[\|E_{X,k} 1_B\|_{L^2(X \times X)}\geq (\|E_{X,k}\|_{\infty}-rK)\sqrt{\mu_{X \times X}(B)}. \]
	By $(a^2,ab)$-standardness, 
	\[\mu_{X}(B) \geq a^2 r^{2b},\]
	so 
	\[\sqrt{\mu_{X}(B)} \geq ar^{b}.  \]
	However, by writing $E_{X,k}(x,x') = \sum_{i=k+1}^{\infty}\lambda_{i}\phi_{i}(x)\phi_{i}(x')$, the orthonormality of our eigenfunctions (and hence the orthonormality of the products $\phi_{i}(x)\phi_{i}(x')$ on $X \times X$) allows us to deduce that:
	\begin{align*}
\|E_{X,k}\|_{L^2(X \times X)}^{2} &= \|\sum_{i=k+1}^{\infty}\lambda_i \phi_{i}(x)\phi_{i}(x')\|_{L^2(X \times X)}^{2}\\
& = \sum_{i=k+1}^{\infty} \|\lambda_i \phi_{i}(x)\phi_{i}(x')\|_{L^2(X \times X)}^{2}\\
& = \sum_{i=k+1}^{\infty} \lambda_{i}^{2}\|\phi_i\|_{L^{2}(X)}^{2}\|\phi_i\|_{L^{2}(X)}^{2}\\
& = \sum_{i=k+1}^{\infty} \lambda_{i}^{2}\\
& = \Lambda_{k+1}.		
	\end{align*}

	Thus,
	\[\sqrt{\Lambda_{k+1}} = \|E_{X,k}\|_{L^2(X \times X)} \geq \|E_{X,k} 1_B\|_{L^2(X \times X)} \geq (\|E_{X,k}\|_{\infty}-rK)ar^b,  \]	
	
	which gives the result.
\end{proof}

%% file: intrinsicpht.tex
\subsection*{Background}
The results of this section assume familiarity with certain tools and results in applied topology, namely those pertaining to persistence modules and persistence diagrams. For an introduction to these topics, the reader can consult the articles of Ghrist \cite{ghrist2008barcodes} and Carlsson \cite{carlsson2009topology}. A more formal and comprehensive treatment can be found in the texts by Edelsbrunner and Harer \cite{edelsbrunner2010computational}, Ghrist \cite{ghrist2014elementary}, and Oudot \cite{oudot2015persistence}. 

\subsection{Existence Results}

In this section, we introduce a large family of topological transforms. These topological transforms can be defined for any homological degree, and so we work with graded persistence diagrams: families of persistence diagrams indexed by degree. Thus, the notation $PH(X,f)$ will correspond to the graded persistence diagram of the sublevel set filtration of the function $f$ on the space $X$, and the notation $\mathbf{GrDiag}$ will refer to the space of graded persistence diagrams.\\

Our first results demonstrate that, under certain general hypotheses, persistence diagrams and their associated Betti and Euler curves exist.

\begin{prop}
	\label{prop:pers}
	 	Let $(X,\dist_{X},\mu_X)$ be a compact metric measure space homeomorphic to the geometric realization of a finite simplicial complex. Then, any finite linear combination $f = \sum_{i=1}^{n}c_i \phi_i$ of eigenfunctions of $D^X$ with nonzero eigenvalue has a well-defined sublevel set persistence diagram $PH(X,f)$.
\end{prop}
\begin{proof}
	Lemma \ref{lem:eiglip} states that the eigenfunctions of $D^X$ are Lipschitz, and hence continuous. Since $f$ is a finite linear combination of continuous functions, it, too, is continuous. Lastly, we appeal to Theorem 3.33 in \cite{chazal2016structure}, which asserts the existence of persistence diagrams of continuous functions on geometric realizations of finite simplicial complexes. 
\end{proof}

In addition to persistent homology, we also define the Betti and Euler curves of the pair $(X,f)$.

\begin{definition}
	Let $(X,\dist_{X},\mu_X)$ be a compact metric measure space. For any homological degree $k \geq 0$, and any finite linear combination $f = \sum_{i=1}^{n}c_i \phi_i$ of eigenfunctions of $D^X$ with nonzero eigenvalue, we view our graded persistence diagrams as graded barcodes and (tentatively) define the degree-$k$ Betti curve to be the sum of the indicator functions of the intervals in the degree-$k$ persistent homology of $(X,f)$:
	\[\beta_{k}(X,f) = \sum_{I \in PH_{k}(X,f)} 1_{I}.  \]
	The Euler curve is then (tentatively) defined to be the alternating sum of these Betti curves:
	\[\chi(X,f) = \sum_{k=0}^{\infty} (-1)^{k} \beta_{k}(X,f).   \]
\end{definition} 

Justifying the existence of these curves requires introducing some more ideas from the theory of persistent homology.

\begin{definition}
	For a point $x$ in a persistence diagram, define $\operatorname{per}(x)$ to be the distance from $x$ to the diagonal. For a real-valued function $f: X \to \mathbb{R}$ on a triangulable, compact metric space $X$, exponent $q > 0$, and threshold parameter $t \geq 0$, we compute the sublevel set persistence and define:
	\[\operatorname{Pers}_{q}(f,t) = \sum_{\operatorname{per}(x) > t} \operatorname{per}(x)^q.  \]
	This is the sum, over all homological degrees, of the $q$th powers of the persistence of points in $PH(X,f)$ with persistence more than $t$. When $t=0$, this quantity is called the \emph{ degree-$q$ total persistence of $f$.} Note that the parameter $q$ refers to the exponent in the sum, not the homological degree of the persistence diagram.
\end{definition}

Bounding the total degree-$q$ persistence of $PH(X,f)$ necessitates placing some restrictions on $X$ and $f$. 

\begin{definition}
	A triangulable metric space $X$ has \emph{polynomial combinatorial complexity} if there are constants $C_0$ and $d$ such that, for any radius parameter $r>0$, there exists a triangulation $T$ of $X$ where every triangle has diameter at most $r$, and $T$ has at most $C_0/r^d$ simplices. 
\end{definition}

In~\cite{cohen2010lipschitz}, the authors note that the bilipschitz image of a $d$-dimensional Euclidean simplicial complex is always of polynomial combinatorial complexity. They also prove that Lipschitz functions on such spaces have finite degree-$q$ total persistence, for $q$ sufficiently large.

\begin{theorem}[{\cite[Section 2.3]{cohen2010lipschitz}}]
	\label{thm:boundedtotper}
	Let $f: X \to \mathbb{R}$ be a Lipschitz function on a metric space $X$ of $(C_0,d)$-polynomial combinatorial complexity which is homeomorphic to the geometric realization of a finite simplicial complex. Let $q = d + \delta$ for some constant $\delta > 0$. Then we have the following bound on  degree-$q$ total persistence of the sublevel set filtration:
	\[\operatorname{Pers}_{q}(f,0) \leq C_0 \operatorname{Lip}(f)^{d}\operatorname{Amp}(f)^\delta \cdot \left(1 + \frac{d+ \delta}{\delta}  \right),  \]
	where $\operatorname{Lip}(f)$ is the Lipschitz constant of $f$ and $\operatorname{Amp}(f) = \max f - \min f$.
\end{theorem}

This motivates the following definition:

\begin{definition}[ {\cite[Section 2.3]{cohen2010lipschitz}} ]
	\label{def:boundedtotal}
	A metric space $X$ \emph{implies bounded degree-$q$ total persistence} if there is a constant $C_{X}$ such that $\operatorname{Pers}_{q}(f,0) \leq C_{X}$ for any real-valued Lipschitz function $f$ with $\operatorname{Lip}(f) \leq 1$.\\
	
	Note that if $X$ implies bounded degree-$q$ total persistence, then $\operatorname{Pers}_{q}(f,0) \leq \operatorname{Lip}(f)^{q} C_{X}$ for any Lipschitz function $f$, as scaling a function only changes the resulting sublevel set persistence by scaling the endpoints of the intervals in its barcode.
\end{definition}
%
%
%

We now prove the existence of Betti and Euler curves for eigenfunctions with nonzero eigenvalue.

\begin{prop}
	\label{prop:betticurves}
	Suppose that $X$ is homeomorphic to the geometric realization of a finite simplicial complex that implies bounded degree-$q$ total persistence. Let $p=1/q$. Then for any homological degree $k$, the sum defining $\beta_{k}(X,f)$ converges in $L^p$. Moreover, the sum defining $\chi(X,f)$ is finite, so the Euler curve exists as a function in $L^p$.
\end{prop}
\begin{proof}
	This proof makes use of Lemma \ref{lem:eiglip}, that eigenfunctions of $D^{X}$ with nonzero eigenvalue are Lipschitz. We would like to show that the potentially infinite sum of indicator functions arising in the definition of $\beta_{k}(X,f)$ converges in $L^p$ for $0 < p < \infty$. The Weierstrass M-test guarantees convergence if:
	\[\sum_{I \in PH_{k}(X,f)} \|1_{I}\|_{L^p} < \infty. \]
	
	Note that the length of an interval $I$ in a barcode is twice the persistence of the corresponding point $x$ in the associated persistence diagram. Thus we have:
	\[\| 1_I \|_{L^p} = \operatorname{len}(I)^{1/p} = (2 \operatorname{per}(x))^{1/p} .\]

	We can thus rewrite the above sum:

	\[\sum_{I \in PH_{k}(X,f)} \|I\|_{L^p} = 2^{1/p} \sum_{x \in PH_{k}(X,f)} \operatorname{per}(x)^{1/p}.\]
	
	Since $p = 1/q$ and $f$ is Lipschitz, we can leverage the fact that $X$ implies bounded degree-$q$ total persistence to show:
	\[\sum_{x \in PH_{k}(X,f)} \operatorname{per}(x)^{1/p} = \sum_{x \in PH_{k}(X,f)} \operatorname{per}(x)^{q}  \leq \sum_{j=0}^{\infty} \sum_{x \in PH_{j}(X,f)} \operatorname{per}(x)^{1/p} = \operatorname{Pers}_{q}(f,0) < \infty. \]
	
 Thus, we have shown that $\beta_{k}(X,f)$ exists and is well-defined.\\
	
	To demonstrate that the sum defining $\chi(X,f)$ is finite, we merely note that, as $X$ is homeomorphic to the geometric realization of a finite simplical complex, there is a finite degree $K$ such that no subset of $X$ has nontrivial homology in degree $k \geq K$. Thus we can write:
	\[\chi(X,f) = \sum_{k=0}^{K} (-1)^k \beta_{k}(X,f).  \]  
\end{proof}

We now define the topological transforms of interest, beginning with those previously studied and then introducing our new contributions. We continue to assume all the assumptions needed to guarantee the existence of the objects in question: our eigenspaces have multiplicity zero, our scheme for identifying ``positive" eigenfunctions does not fail, we only take eigenfunctions with nonzero eigenvalue, the metric measure space is homeomorphic to the geometric realization of a finite simplicial complex and, to guarantee the existence of the appropriate Euler curves, the metric space implies bounded degree-$q$ total persistence.\\

%
	

\begin{definition}
	Let $\mathbb{S}^{k}$ be the $k$-dimensional sphere, and $L(\mathbb{R}^{k+1},\mathbb{R})$ the space of linear maps from $\mathbb{R}^{k+1}$ to $\mathbb{R}$. Define the map $\Theta : \mathbb{S}^{k} \to L(\mathbb{R}^{k+1},\mathbb{R})$ which sends $v \in \mathbb{S}^{k}$ to the map $x \mapsto \langle x, v \rangle$.
\end{definition}

\begin{definition}(\cite{curry2018many})
	Let $X \subset \mathbb{R}^d$ be a compact, definable set\footnote{See \cite{curry2018many} \S2 for the definition of definable sets.}. For every $v \in \mathbb{S}^{d-1}$, the sublevel set persistence diagram and Euler curve of the pair $(X,\Theta(v))$ exist. The \emph{Persistent Homology Transform} is the map $PHT(X):\mathbb{S}^{d-1} \to \mathbf{GrDiag}$  defined by:
	\[PHT(X)(v) = PHT(X,\Theta(v)).\]
	If one computes Euler curves instead of persistence diagrams, one obtains the \emph{Euler Characteristic Transform} $ECT(X)$.
\end{definition}

\begin{theorem}(\cite{curry2018many,ghrist2018persistent})
	\label{thm:PHTinj}
The PHT and ECT are injective for all $k$. That is, let $X,Y \subset \mathbb{R}^{d}$ be compact, definable subsets. If $PHT(X) = PHT(Y)$ or $ECT(X) = ECT(Y)$, then $X=Y$ as sets.
\end{theorem}

We now define our new topological transforms, of which there are two types: (1) intrinsic transforms, which compute the persistence diagrams and Euler curves of linear combinations of eigenfunctions, and (2) embedded transforms, which are the composition of the PHT or ECT with the distance kernel embedding. 

\begin{convention}
	Recall that the coordinate functions $\alpha_{i}$ defining the distance kernel embedding are complex-valued. For the purpose of computing topological invariants, we would like our height functions to be real-valued. We will thus replace each function $\alpha_{i}$ with its real and imaginary parts: $\alpha_{i}^{R} = \operatorname{Re}(\alpha_{i})$ and $\alpha_{i}^{I} = \operatorname{Im}(\alpha_{i})$, and likewise replace the embedding $\Phi_{k}$ by two $\mathbb{R}^{k}$ valued maps:  $\Phi_{k}^{R} = \operatorname{Re}(\Phi_{k})$ and $\Phi_{k}^{I} = \operatorname{Im}(\Phi_{k})$. Note that, for any index $i$, one of $\alpha_{i}^{R}$ or $\alpha_{i}^{I}$ is identically zero, and the other is equal to $\sqrt{|\lambda_i}|\phi_{i}$.
\end{convention}

\begin{definition}	
	\label{def:transforms}	
	For $k$ finite, the \emph{embedded persistence kernel transform} e-$PKT_{k}(X)$ is the PHT applied to the image of the embedding $(\Phi_{k}^{R}(X),\Phi_{k}^{I}(X)) \subset \mathbb{R}^{2k}$, which takes as input vectors in $\mathbb{S}^{2k-1}$ and takes values in $\mathbf{GrDiag}$. The \emph{intrinsic persistence kernel transform} i-$PKT_{k}(X)$ is the map from $\mathbb{S}^{2k-1}$ to $\mathbf{GrDiag}$ that maps $(u,v) \in \mathbb{S}^{2k-1} \subset \mathbb{R}^{k} \times \mathbb{R}^{k}$ to the sublevel set persistence of the pair: 
	\[\left(X,\sum_{i=0}^{k} u_{i}\alpha_{i}^{R} + v_{i}\alpha_{i}^{I} \right).\]
		Using Euler curves in place of persistent homology gives rise to the embedded and intrinsic \emph{Euler kernel transforms}, noted e-$EKT_{k}$ and i-$EKT_{k}$, respectively.
\end{definition}

\begin{remark}
	The difference between the intrinsic and embedded transforms is that the intrinsic transforms are defined using the metric measure space $X$ as the base topological space, whereas the embedded transforms use its image $(\Phi_{k}^{R}(X),\Phi_{k}^{I}(X)) \subset \mathbb{R}^{2k}$. When the mapping $(\Phi_{k}^{R},\Phi_{k}^{I})$ is injective (which is equivalent to the injectivity of $\Phi_{k}$) the two transforms are equivalent, as $X$ is homeomorphic to $(\Phi_{k}^{R}(X),\Phi_{k}^{I}(X))$. Otherwise, they contain different information, as $(\Phi_{k}^{R}(X),\Phi_{k}^{I}(X))$ is homeomorphic to the quotient of $X$ obtained by identifying points $x,x' \in X$ whenever $\phi_{i}^{R}(x)= \phi_{i}^{R}(x')$ and $\phi_{i}^{I}(x)= \phi_{I}^{R}(x')$ for all $i = 1, \cdots, k$. Note, in particular, that the topological type of $(\Phi_{k}^{R}(X),\Phi_{k}^{I}(X))$ may change as $k$ increases.\\
	
	To give a simplified example of the failure of the intrinsic and embedding transforms to give the same result, set $k=1$ and suppose that the first coordinate function $\alpha_{1}$ is constant. Then the set $\Phi_{1}(X) \subset \mathbb{R}^2$ is simply a point, and for any vector direction, the PHT or ECT give the persistence diagram or Euler curve of a point. However, the sublevel set filtration of the pair $(X,\alpha_{1})$ gives the empty set until the constant value of $\alpha_{1}$ is reached, at which point the entire space $X$ appears all at once, and its homology groups/Euler characteristic appear in the persistence diagram or Euler curve.
\end{remark}


\subsection{Stability and Inverse Results}
We will make use of the well-known functional stability result in the theory of persistence.

\begin{theorem}[Functional Stability, \cite{chazal2016structure,cohen2005stability}]
	\label{thm:funcstab}
	Let $X$ be a topological space, and let $f,g:X \to \mathbb{R}$ be two functions whose sublevel sets have finite-dimensional homology groups. Then $(X,f)$ and $(X,g)$ give rise to pointwise finite dimensional persistence modules with well-defined graded persistence diagrams $PH(X,f)$ and $PH(X,g)$. Moreover, writing $d_{B}$ for the graded bottleneck distance, we have:
	\[d_{B}(PH(X,f),PH(X,g)) \leq \|f-g\|_{\infty}.\]
\end{theorem}

We now state our first stability result: stability of the i-$PKT_{k}$ and e-$PKT_{k}$ with respect to the vector $v \in \mathbb{S}^{2k-1}$.

\begin{prop}
\label{prop:iPKTlip}
	Suppose that $X$ is homeomorphic to the geometric realization of a finite simplicial complex. If we equip $\mathbf{GrDiag}$ with the graded bottleneck distance and the sphere $\mathbb{S}^{2k-1}$ with the $\ell^1$ distance, then both the i-$PKT_{k}$ and e-$PKT_{k}$ are Lipschitz continuous.
\end{prop}
\begin{proof}
	We begin with the i-$PKT_{k}$. Our embedding functions $\alpha_i^{R}$ and $\alpha_{i}^{I}$ are continuous functions on a compact space, hence they are uniformly bounded by a constant $C$. Then, if $(u,v), (w,z) \in \mathbb{S}^{2k-1} \subset \mathbb{R}^{k} \times \mathbb{R}^{k}$ are vectors with $\|(u,v) - (w,z)\|_{1} = \| u-w\|_{1} + \| v-z\|_{1} \leq \epsilon$, we have:
	\begin{align*}
\|\sum_{i=1}^{k} (u_{i}\alpha_{i}^{R} + v_{i}\alpha_{i}^{I}) - \sum_{i=1}^{k}(w_{i}\alpha_{i}^{R} + z_{i}\alpha_{i}^{I}) \|_{\infty} &= \|\sum_{i=1}^{k} (u_i - w_i)\alpha_i^{R} + (v_i - z_i)\alpha_i^{I}  \|_{\infty}\\
& \leq \sum_{i=1}^{k} |u_i - w_i|\|\alpha_i^{R} \|_{\infty} + \sum_{i=1}^{k} |v_i - z_i|\|\alpha_i^{I} \|_{\infty}\\
& \leq C \sum_{i=1}^{k} \left(|u_i - w_i| + |v_i - z_i|\right)\\
& \leq  C\varepsilon.
	\end{align*}

	Thus, by Theorem \ref{thm:funcstab}, we have:
	\[d_{B}\left(PH \left(X,\sum_{i=1}^{k} (u_{i}\alpha_{i}^{R} + v_{i}\alpha_{i}^{I})\right),PH \left(X,\sum_{i=1}^{k} (w_{i}\alpha_{i}^{R} + z_{i}\alpha_{i}^{I})\right)\right) \leq C\varepsilon.\]
	
	For the e-$PKT_{k}$, the boundedness of our eigenfunctions implies that $(\Phi_{k}^{R}(X),\Phi_{k}^{I}(X))$ sits inside the $\ell^\infty$-ball of radius $C$ in $\mathbb{R}^{2k}$. Then, if $(u,v),(w,z) \in \mathbb{S}^{2k-1}$ are vectors with $\| (u,v) - (w,z)\|_{1} \leq \epsilon$, we deduce, for all $(x,y) \in (\Phi_{k}^{R}(X),\Phi_{k}^{I}(X))$:
	\begin{align*}
|\Theta(u,v)(x,y) - \Theta(w,z)(x,y)| & =|\langle u,x \rangle + \langle v,y \rangle - \langle w,x\rangle - \langle z,y\rangle|\\
& = |\langle u - w,x\rangle + \langle v-z,y \rangle|\\
&  = |\langle (u-w,v-z),(x,y)\rangle| \\
\mbox{(H\"{o}lder's Inequality)} &\leq \|(x,y)\|_{\infty}\|(u-w,v-z)\|_{1}\\
& \leq C\varepsilon.  
	\end{align*}
	And thus, again by Theorem \ref{thm:funcstab},
		\[d_{B}(PH(X,\Theta(u,v)),PH(X,\Theta(w,z))) \leq C\varepsilon.\]
\end{proof}

The next stability result, which applies to the i-$ECT_{k}$ and e-$ECT_{k}$, relies on the fact that the map from a persistence diagram to its Euler curve is stable:

	\begin{lemma}
		\label{clm:bar2euler}
	Let $B_1$ and $B_2$ be graded persistence diagrams. Suppose that $B_1$ and $B_2$ have at most $L$ points each. Define the following operation that sends a graded persistence diagram to the Euler curve associated to its barcode: 
	\[\chi(B) = \sum_{I \in B} (-1)^{\operatorname{deg}(I)}1_{I}, \] 
	where we write $\operatorname{deg}(I)$ to indicate the homological degree of an interval $I$. Set  $\chi_{1} = \chi(B_1)$ and $\chi_{2} = \chi(B_2)$. Then for any $0<p<1$,
	\[\|\chi_{1} - \chi_{2}\|_{p} \leq (4L)^{q}d_{B}(B_1,B_2)^{q},\]
	and for any $p \geq 1$,
	\[\|\chi_{1} - \chi_{2}\|_{p} \leq (2^{1 + q}L) d_{B}(B_1,B_2)^{q},  \]
	where $q = 1/p$.
	
\end{lemma}
\begin{proof}
	By the definition of the Bottleneck distance, there are decompositions:
	\[B_1 = \{I_{i}^{1}\}_{i=1}^{n_1} \cup \{J_{i}^{1}\}_{i=1}^{n_2} \]
	\[B_2 = \{I_{i}^{2}\}_{i=1}^{n_1} \cup \{J_{i}^{2}\}_{i=1}^{n_3}, \]
	where the intervals $I_{i}^{1}$ are paired with the intervals $I_{i}^{2}$ (both have the same degree), and the intervals $J_{i}^{1}$ and $J_{i}^{2}$ are paired with the diagonal. This means that, for $i \in \{1, \cdots, n_1\}$, the function $1_{I^{1}_{i}} - 1_{I^{2}_{i}}$ is supported on a set of measure at most $2d_{B}(B_1,B_2)$. Moreover, as the difference of two indicator functions, $1_{I^{1}_{i}} - 1_{I^{2}_{i}}$ is bounded in absolute value by one. Additionally, all the intervals $\{J_{i}^{1}\}_{i=1}^{n_2}$ and $\{J_{i}^{2}\}_{i=1}^{n_3}$ have length at most $2d_{B}(B_1,B_2)$. Note that $n_1 + n_2 + n_3 \leq (n_1 + n_2) + (n_1 + n_3) \leq 2L$. Thus:
	\[\chi_{1} - \chi_{2} = \sum_{i=1}^{n_1}(-1)^{\operatorname{deg}(I_{i}^{\ast})}(1_{I_{i}^{1}} - 1_{I_{i}^{2}}) + \sum_{i=1}^{n_2}(-1)^{\operatorname{deg}(J_{i}^{1})}1_{J_{i}^{1}} - \sum_{i=1}^{n_3}(-1)^{\operatorname{deg}(J_{i}^{2})}1_{J_{i}^{2}},  \]
where the $\ast$ indicates that the choice of $I_{i}^{1}$ or $I_{i}^{2}$ does not matter, as both intervals have the same homological degree. We now split our analysis into two cases: (a) $0< p<1$, and (b) $p \geq 1$.\\

(a) In this case, we use the modified Minkowski's inequality $\|f-g\|_{p}^{p} \leq \|f\|_{p}^{p} + \|g\|_{p}^{p}$. We deduce:
\begin{align*}
\|\chi_{1} - \chi_{2}\|_{p}^{p} &\leq \sum_{i=1}^{n_1}\|1_{I_{i}^{1}} - 1_{I_{i}^{2}}\|_{p}^{p} + \sum_{i=1}^{n_2}\|1_{J_{i}^{1}}\|_{p}^{p} +  \sum_{i=1}^{n_3}\|1_{J_{2}^{1}}\|_{p}^{p}\\
& \leq (n_1 + n_2 + n_3)\|1_{[0,2d_{B}(B_1,B_2)]}\|_{p}^{p} \\
& \leq 2L \|1_{[0,2d_{B}(B_1,B_2)]}\|_{p}^{p}\\
& = 2L(2d_{B}(B_1,B_2))\\
& = 4Ld_{B}(B_1,B_2).
\end{align*}

Hence, $\|\chi_{1} - \chi_{2}\|_{p} \leq (4L)^{q}d_{B}(B_1,B_2)^{q}$.

(b) In this case, we can use Minkowski's inequality to deduce:
\begin{align*}
\|\chi_{1} - \chi_{2}\|_{p} &\leq \sum_{i=1}^{n_1}\|1_{I_{i}^{1}} - 1_{I_{i}^{2}}\|_{p} + \sum_{i=1}^{n_2}\|1_{J_{i}^{1}}\|_{p} +  \sum_{i=1}^{n_3}\|1_{J_{2}^{1}}\|_{p}\\
& \leq (n_1 + n_2 + n_3)\|1_{[0,2d_{B}(B_1,B_2)]}\|_{p} \\
& \leq 2L \|1_{[0,2d_{B}(B_1,B_2)]}\|_{p}\\
& = 2L \sqrt[p]{2d_{B}(B_1,B_2)}\\
& = 2^{1 + 1/p}L \sqrt[p]{d_{B}(B_1,B_2)}\\
& = (2^{1 + q}L) d_{B}(B_1,B_2)^{q}.
\end{align*}
	
\end{proof}

\begin{cor}
	\label{cor:ECTlip}
 	Suppose that $X$ is homeomorphic to the geometric realization of a finite simplicial complex which implies bounded degree-$q$ total persistence, for some $q > 0$. Suppose further that there is a uniform bound on the number of points in the persistence diagrams obtained when evaluating the i-$PKT_{k}$ and e-$PKT_{k}$ at an arbitrary vector $(u,v) \in \mathbb{S}^{2k-1}$. If we equip the sphere $\mathbb{S}^{2k-1}$ with the $\ell^1$ distance, and the space of Euler curves with the $L^{1/q}$ distance, the i-$ECT_{k}$ and e-$ECT_{k}$ are $q$-H\"{o}lder continuous on $\mathbb{S}^{2k-1}$.
\end{cor}
\begin{proof}
	The result follows from Proposition \ref{prop:iPKTlip} and Lemma \ref{clm:bar2euler}.
\end{proof}

We conclude this section with a coarse injectivity result, this time only for our embedded transforms. Note that, when $\Phi_{k}$ is injective, and our transforms agree, this result also holds for the intrinsic transforms.

\begin{theorem}
\label{thm:ePKTcoarseinj}
Let $(X,\dist_{X}, \mu_X)$ and $(Y,\dist_{Y}, \mu_Y)$ be compact metric measure spaces, with eigenvalues $\{\lambda_{i}\}$ and $\{\nu_{i}\}$ respectively. Let $k \in \mathbb{N}$ be a positive integer, and suppose that $\Phi_{k}(X)$ and $\Phi_{k}(Y)$ are definable. Then if either e-PKT$_{k}(X)$ = e-PKT$_{k}(Y)$ or e-EKT$_{k}(X)$ = e-EKT$_{k}(Y)$, we have:

\[d_{GH}(X,Y) \leq  \| E_{X,k} \|_{\infty} + \| E_{Y,k} \|_{\infty}.\]
\end{theorem}
\begin{proof}
	This follows from the injectivity of the PHT and ECT, as demonstrated in Theorem \ref{thm:PHTinj}, together with Theorem \ref{thm:invstabmeas}, where we take $\epsilon = 0$.
\end{proof}

The condition that the distance kernel embeddings be definable is always satisfied when the spaces are finite. Demonstrating definability more generally will require techniques and results from the theory of o-minimal geometry, and will therefore constitute a line of inquiry quite distinct from the focus of this paper.


%

%% file: experiments.tex
In this section, we show the results of some numerical experiments for the distance kernel embedding. The spaces we look at are discrete samples, as diverse resolutions, of the $L(7,1)$ Lens space, the $L(7,4)$ Lens space, the torus $\mathbb{T}^2$, the 2-sphere $\mathbb{S}^2$, and the 3-sphere $\mathbb{S}^3$.Note that L(7,1) and L(7,4) are homotopic Lens spaces, but they are not homeomorphic, and hence not isometric.

%

The quantities we measure are: the Hausdorff distances between distance kernel embeddings, the distribution of the absolute values of the eigenvalues, the distribution of mass and $L^{\infty}$ norms of the eigenfunctions $\phi_i$ and the distance kernel embedding $\Phi_{k}$, and the constants $\|E_{X,k}\|_{\infty}$ and $\max_{x \in X} \|\Phi_{k}(x)\|_{2}$ appearing in Theorem \ref{thm:invstabmeas}, as well as the explicit upper bounds presented in Lemmas \ref{lem:errorbound} and \ref{lem:embedbound}.

\subsection*{Distribution of Eigenvalues and Hausdorff Distances between DKEs}

The goal of this section is to illustrate the results of Sections~\ref{sec:discretization}, \ref{sec:stability}, and \ref{sec:stabinv}. In the following experiments, we compute the DKE for a variety of discrete samples on the torus and 2-sphere with metric induced by their embedding in Euclidean space, the 3-sphere, $L(7,1)$ Lens space, and $L(7,4)$ Lens space, with spherical geometry. The measures on these samples are uniform. 
These spaces have distinct integer homology, except for the two Lens spaces that have the same homotopy type but are not homeomorphic, and therefore not isometric. This makes $L(7,1)$ and $L(7,4)$ difficult to distinguish by purely topological methods. We see that the DKE (and, therefore, the resulting topological transforms) is capable of distinguishing these Lens spaces.

\smallskip

\noindent
\emph{Spectra of various manifolds.}  
In Figure~\ref{fig:eig}, we have plotted the first $8$ eigenvalues of five discrete metric spaces, sampled from each of these five manifolds, normalized by the number of points in each sample. We can observe the following: (1) the two Lens spaces have relatively similar eigenvalues, (2) the $2$- and $3$-sphere have many similar eigenvalues, but their first and fourth eigenvalues are significantly different, and (3) the torus has the most distinct spectrum.

\smallskip

\noindent
\emph{Spectra of Lens spaces for various samples.} In Figure \ref{fig:Lenspec}, we compare the spectra of a number of different random i.i.d. samples of the two Lens spaces $L(7,1)$ and $L(7,4)$. To be precise, for each Lens space we compute the spectra of two distinct random samples with $2000$ points, and a third sample with $5000$ points. The spectra for the different samples of the same Lens space are virtually impossible to distinguish, and only two curves---one for the spectrum of $L(7,1)$, and one of the spectrum of $L(7,4)$---are visible in Figure \ref{fig:Lenspec}. This attests to the stability of the eigenvalues of the distance kernel operator under random i.i.d. sampling, in line with Theorem \ref{thm:approxDKE}. 

Notably, the two Lens spaces are distinguished by the first, third, and fourth eigenvalues of their distance kernel operators. $L(7,1)$ and $L(7,4)$ having same homotopy type, this illustrates the ability of the operator to capture geometric information and distinguish between non-isometric spaces. 

\smallskip

\noindent
\emph{Hausdorff distance between DKEs.} Finally, in Figure \ref{fig:Haus}, we compare the Hausdorff distances between various pairs of distance kernel embeddings. We observe the following: 
(1) The two samples of the same size coming from the $L(7,1)$ Lens space are the closest in Hausdorff distance, and that distance is close to zero up to dimension $k=4$. Indeed, if we had taken samples of sufficiently high resolution, we would see the Hausdorff distances going to zero for larger values of $k$, as proven in Theorem \ref{thm:approxDKE}. 
(2) The second closest pair of spaces are the Lens spaces $L(7,1)$ and $L(7,4)$, that have same homotopy type and have both spherical geometry. 
(3) The third closest pair of spaces are the Lens space $L(7,1)$ and the 3-Sphere, both with spherical geometry (Lens spaces are constructed as quotients of 3-Spheres). 
(4) The manifold that appears to be most distinct from the rest is the torus. 
(5) For all pairs of manifolds, the Hausdorff distance stabilizes at around $k=10$, after which eigenvalues are close to $0$.	

Figure \ref{fig:Haus} illustrates that manifolds sharing topological or geometrical properties lead to more similar DKE. However, the DKEs of distinct samples of $L(7,1)$ remain substantially closer in Hausdorff distance compare to the DKE of samples of $L(7,4)$, with same homotopy type and same geometry locally, and the 3-sphere, with same local geometry.  

\medskip

In conclusion, these experiments illustrate that the spectra and embedding of the distance kernel operator can be approximated by finite samples, as predicted by Theorem \ref{thm:approxDKE}. Moreover, by combining the DKE with the Hausdorff metric on Euclidean space, we obtain a pseudo-metric on the space of compact metric measure spaces that succeeds in distinguishing a variety of diverse manifolds.

\begin{figure}[t!]
	\centering
	\begin{subfigure}[t]{0.5\textwidth}
		\centering
		\includegraphics[scale = 0.5]{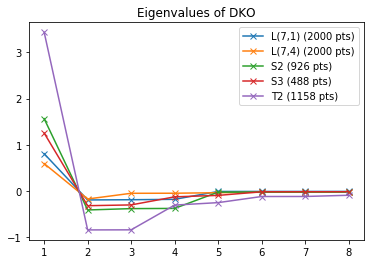}
		\caption{Eigenvalues of the DKO for a variety of\\ spaces, normalized by the number of points\\ in the sample.}
		\label{fig:eig}
	\end{subfigure}%
	~ 
	\begin{subfigure}[t]{0.5\textwidth}
		\centering
		\includegraphics[scale = 0.5]{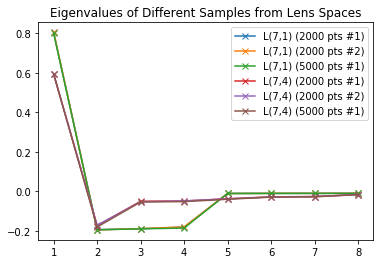}
		\caption{A comparison of the eigenvalues of various samples, at different resolutions, of these two Lens spaces.}
			\label{fig:Lenspec}
	\end{subfigure}
\vspace{10mm}

	\begin{subfigure}[t]{1\textwidth}
	\centering
	\includegraphics[scale = 0.7]{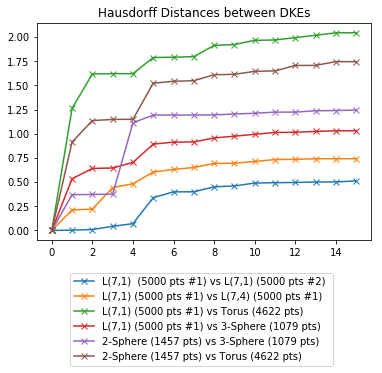}
	\caption{A comparison of Hausdorff distances between various samples of 2- and 3-manifolds.}
		\label{fig:Haus}
\end{subfigure}%

\caption{Spectra and DKE for samples of various manifolds.}
\end{figure}

\subsection*{Distribution of mass and $L^{\infty}$ norms of $\phi_i$ and $\Phi_{k}$.}
Complementing the prior simulations, we now explore the distribution of mass of (the absolute value) our eigenfunctions $\phi_i$. This is important for estimating the $L^{\infty}$ norm of these coordinate functions in the continuous case (in the finite case, the fact that $\|\phi_i\|_{L^2} = 1$ forces all the entries of the vector to have norm at most $1$). Our stability theory (Section \ref{sec:stability}) tells us that the histograms of the continuous eigenfunctions is approximated by the histograms of the discrete eigenfunctions. We can estimate the $L^{\infty}$ norm of $\phi_{i}$ by considering the histogram of values of $|\phi_{i}(x)|$ (i.e. its distribution of mass) as follows. If $|\phi_{i}(x)|$ assumes values close to its maximum on a large (in measure) subset of $X$, the fact that $\|\phi_i\|_{L^2} = 1$ ensures that this maximum cannot be too large, relative to the volume of $X$. Whether or not this is the case can be determined by inspecting this histogram. Figures \ref{fig:T1eighist}, \ref{fig:SS1eighist}, \ref{fig:SSS1eighist}, and \ref{fig:A1eighist} show the histogram of absolute values for the tenth and twentieth eigenfunctions of a range of shapes. Observe that for all of these histograms, the eigenfunctions assumes within 20\% of their maximum value on a subset of measure at least $\vol(X)/2$. This tells us that the $L^{\infty}$ norms of these eigenfunctions cannot be too big.

We can carry out the same analysis for the distance kernel embedding $\Phi_{k}$. In Figures \ref{fig:T1hist500}, \ref{fig:SS1hist500}, \ref{fig:SSS2hist500}, and \ref{fig:A1hist500}, we see the histograms of the magnitudes of these embedding vectors for a variety of shapes. We see here in every example that the magnitude of the embedding vector is never below 50\% of the maximum, and is often much greater. As before, this rules out the possibility of large embedding vectors, relative to the volume of our manifold.
%
%

\begin{figure}[htb!]
	\includegraphics[scale=0.5]{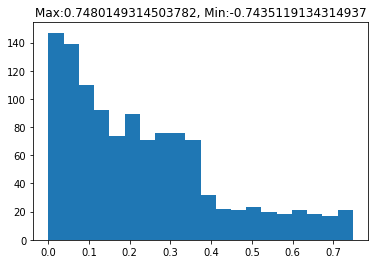} \includegraphics[scale=0.5]{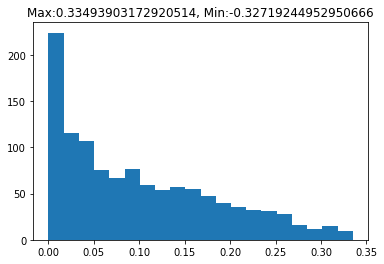}
	\caption{The histogram of values, together with the maximum and minimum values, for the tenth and twentieth eigenvalues of a sample from the torus (1158 points).}
	\label{fig:T1eighist}
\end{figure}

\begin{figure}[htb!]
	\includegraphics[scale=0.5]{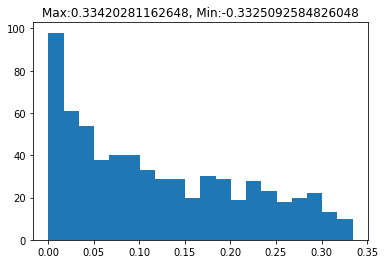} \includegraphics[scale=0.5]{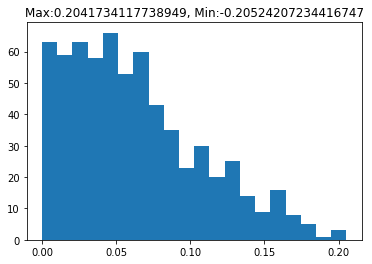}
	\caption{The histogram of values, together with the maximum and minimum values, for the tenth and twentieth eigenvalues of a sample from the 2-sphere (564 points).}
	\label{fig:SS1eighist}
\end{figure}

\begin{figure}[htb!]
	\includegraphics[scale=0.5]{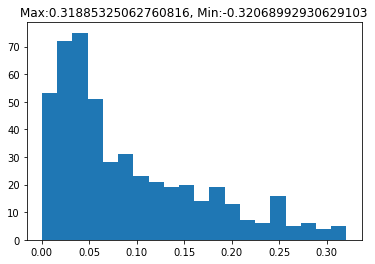} \includegraphics[scale=0.5]{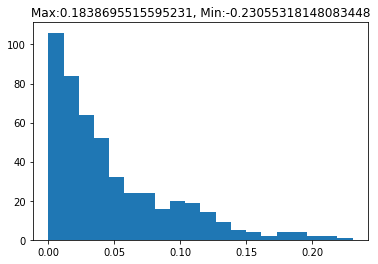}
	\caption{The histogram of values, together with the maximum and minimum values, for the tenth and twentieth eigenvalues of a sample from the 3-sphere (488 points).}
	\label{fig:SSS1eighist}
\end{figure}

\begin{figure}[htb!]
	\includegraphics[scale=0.5]{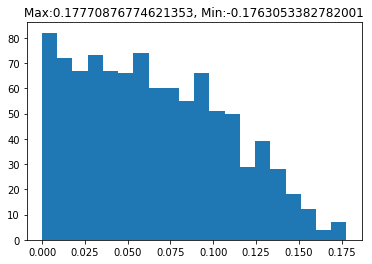} \includegraphics[scale=0.5]{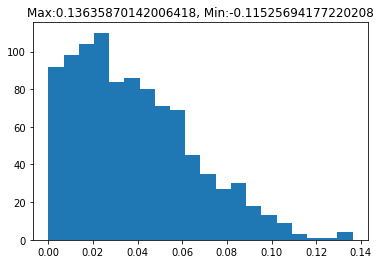}
	\caption{The histogram of values, together with the maximum and minimum values, for the tenth and twentieth eigenvalues of a sample from the $L(11,3)$ Lens space(980 points).}
	\label{fig:A1eighist}
\end{figure}

\begin{figure}[t!]
	\centering
	\begin{subfigure}[t]{0.5\textwidth}
		\centering
		\includegraphics[scale = 0.5]{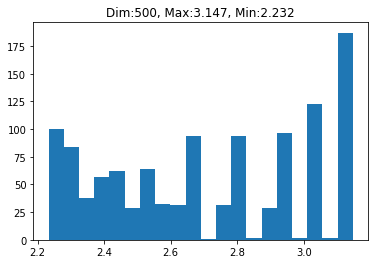}
		\caption{The histogram of the magnitude of the\\ embedding vectors for a sample from the\\ torus (1158 points) in dimension $k=500$.}
		\label{fig:T1hist500}
	\end{subfigure}%
	~ 
	\begin{subfigure}[t]{0.5\textwidth}
		\centering
		\includegraphics[scale = 0.5]{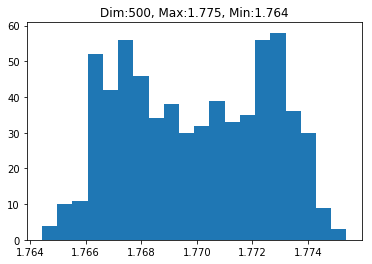}
		\caption{The histogram of the magnitude of the\\ embedding vectors for a sample from the 2-sphere\\ (564 points) in dimension $k=500$.}
		\label{fig:SS1hist500}
	\end{subfigure}
	\vspace{10mm}
	
	\begin{subfigure}[t]{0.5\textwidth}
		\centering
		\includegraphics[scale = 0.5]{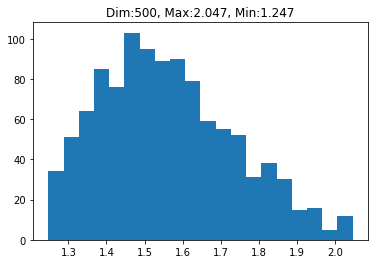}
		\caption{The histogram of the magnitude of the\\ embedding vectors for a sample from the\\ 3-sphere (1079 points) in dimension $k=500$.}
		\label{fig:SSS2hist500}
	\end{subfigure}%
	~ 
	\begin{subfigure}[t]{0.5\textwidth}
		\centering
		\includegraphics[scale = 0.5]{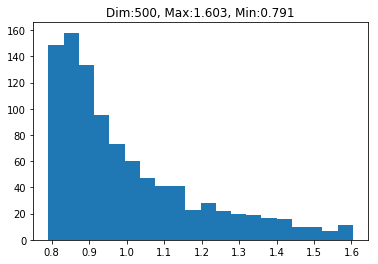}
		\caption{The histogram of the magnitude of the\\ embedding vectors for a sample from the $L(11,3)$ Lens space (980 points) in dimension $k=500$.}
		\label{fig:A1hist500}
	\end{subfigure}
	\caption{}
\end{figure}
%
%
%
%
%
%
%
%


\clearpage
\subsection*{Embedding Bounds}
In this section, we consider some simulations for the quantites $\|E_{k}\|_{\infty}$ and  $\max_{x \in X} \|\Phi_{k}(x)\|_{2}$ appearing in Theorem \ref{thm:invstabmeas}, as well as the explicit upper bounds presented in Lemmas \ref{lem:errorbound} and \ref{lem:embedbound}. Let us introduce some notations used in our table of results:
\begin{itemize}
	\item $A =  \|E_{X,k}\|_{\infty}$ is our additive constant.
	\item $B = \max_{x \in X} \|\Phi_{k}(x)\|_{2}$ is our multiplicative constant.
	\item A\_Bound is the bound on $A$ given in Lemma \ref{lem:errorbound}.
	\item B\_Bound is the bound on $B$ given in Lemma \ref{lem:embedbound}.
\end{itemize}

We compute these quantities across a range of dimensions for the torus, 2-sphere, and 3-sphere. We also note the diameters of these spaces, for comparison. In each table, we observe the same phenomenon. The worst-case bounds on B and A grow in the embedding dimension. However, the multiplicative constant B stabilizes at a small value (relative to $1$), and the additive constant A goes to zero, quickly becoming much smaller than the diameter.  

\begin{table}[htb!]
	\centering
	
	\begin{tabular}{|l|l|l|l|l|l|}
		\hline
		
		Dimension & 2 & 5 & 10 & 50 & 200 \\ \hline
		B\_Bound & 7.098810562 & 11.95365386 & 18.41320952 & 50.85533019 & 91.71628432 \\ \hline
		B & 2.785385644 & 2.866424003 & 2.937811041 & 3.061283988 & 3.127787041 \\ \hline
		A\_bound & 50.42347825 & 126.1558938 & 292.5051042 & 2313.723784 & 4249.850655 \\ \hline
		A & 2.831494962 & 0.863195908 & 0.904480257 & 0.188507231 & 0.080790547 \\ \hline
		
	\end{tabular}
	\caption{Simulations of the constants A and B, and their upper bounds, for a sample from the torus (1158 points, diam 7.52945)}
\end{table}

\begin{table}[htb!]
	\centering
	
	\begin{tabular}{|l|l|l|l|l|l|}
		\hline
		
		Dimension & 2 & 5 & 10 & 50 & 200 \\ \hline
		B\_Bound & 4.415703664 & 8.423681386 & 12.7145307 & 41.13417364 & 44.98760706 \\ \hline
		B & 1.655840841 & 1.703608832 & 1.711791639 & 1.747744118 & 1.767800296 \\ \hline
		A\_Bound & 19.05590903 & 63.08320762 & 141.1412916 & 1115.905333 & 1115.912121 \\ \hline
		A & 1.529116741 & 0.369015434 & 0.291928199 & 0.07784093 & 0.018153383 \\ \hline
		
	\end{tabular}
	\caption{Simulations of the constants A and B, and their upper bounds, for a sample from the 2-sphere (564 points, diam 3.10087)}
\end{table}

\begin{table}[htb!]
	\centering
	
	\begin{tabular}{|l|l|l|l|l|l|}
		\hline
		
		Dimension & 2 & 5 & 10 & 50 & 200 \\ \hline
		B\_Bound & 4.658734251 & 7.958344597 & 14.99125753 & 33.83567486 & 34.71946412 \\ \hline
		B & 1.769845855 & 1.895902581 & 1.942573759 & 1.978396949 & 2.011362641 \\ \hline
		A\_Bound & 21.43974332 & 56.15815487 & 215.9544069 & 653.2712051 & 653.2911455 \\ \hline
		A & 1.138554951 & 0.364063761 & 0.290519127 & 0.123081673 & 0.032296336 \\ \hline

	\end{tabular}
	\caption{Simulations of the constants A and B, and their upper bounds, for a sample from the 3-sphere(488 points, diam 3.06147)}
\end{table}
%
%
%
%
%
%
%
%
%
%

%% file: conclusion.tex
The results of this paper justify the consideration of a new method for comparing metric-measure objects and a new class of topological transforms.\\

In terms of future work, there are both theoretical and applied lines of inquiry:
\begin{itemize}
	\item More carefully study the spectral theory of the distance kernel operator to understand the decay of the eigenvalues and the histograms of $\sqrt{\lambda_i}\phi_i, \|\Phi_{k}\|$, and $E_{X,k} = |\dist_{X}(x,y) - \sum_{i=1}^{k} \lambda_{i}\phi_{i}(x)\phi_{i}(y)|$. In particular, our experimental data shows these distributions are much more concentrated than the worst-case estimates coming from Markov's inequality.
	\item Determine if there exists a sufficiently large dimension $K$ making the map $\Phi_{K}:X \to \mathbb{C}^{K}$ an honest embedding, i.e. injective.
	\item Obtain stability results for our topological transforms under perturbation and discretization of the underlying shape (in line with the results of Section \ref{sec:discretization} and \ref{sec:stability} for the distance kernel embedding).
	\item Proving the definability of the DKE embedding for sufficiently nice metric measure spaces.
	\item Introduce the PHT or ECT into the experimental pipeline and observe if the resulting invariants are effective feature vectors across a variety of learning tasks. 
\end{itemize}